\documentclass[reqno,10pt]{article}

\usepackage{amsmath}
\usepackage{amsfonts,amsthm,amssymb}
\usepackage{indentfirst}
\usepackage{graphicx}
\usepackage{graphics}
\usepackage{pict2e}
\usepackage{epic}
\numberwithin{equation}{section}
\usepackage[margin=1.7cm]{geometry}
\usepackage{epstopdf} 
\usepackage[colorlinks,linkcolor=blue]{hyperref}
\usepackage[capitalise,noabbrev]{cleveref}
\usepackage{authblk}

\usepackage{todonotes}

\makeatletter
\def\thanks#1{\protected@xdef\@thanks{\@thanks
		\protect\footnotetext{#1}}}
\makeatother

\crefformat{equation}{(#2#1#3)}
\crefmultiformat{equation}{(#2#1#3)}{ and~(#2#1#3)}{, (#2#1#3)}{ and~(#2#1#3)}
\crefrangeformat{equation}{(#3#1#4) to~(#5#2#6)}

\usepackage[normalem]{ulem}

\allowdisplaybreaks

\theoremstyle{plain}
\newtheorem{Thm}{Theorem}[section]
\newtheorem*{Thm*}{Theorem}
\newtheorem{Lem}[Thm]{Lemma}
\newtheorem{Cor}[Thm]{Corollary}
\newtheorem{Prop}[Thm]{Proposition}

\theoremstyle{definition}

\newtheorem{Rem}[Thm]{Remark}
\newtheorem{?}[Thm]{Problem}

\newcommand{\p}{\partial}

\newcommand{\R}{\mathbb{R}}
\newcommand{\T}{\theta}
\newcommand{\e}{\varepsilon}
\newcommand{\E}{\mathcal{E}}

\newcommand{\Do}{\mathbf{D}_0}

\newcommand{\Torus}{\mathbb{T}}

\def \r {\rho}

\newcommand{\dv}{\text{div}}

\newcommand{\m}{\textbf{m}}
\newcommand{\abs}[1]{\left\lvert#1\right\rvert}

\newcommand{\norm}[1]{\left\lVert#1\right\rVert}

\usepackage{color}

\usepackage{todonotes}
\title{Low Mach number limit around the planar diffusion wave for 3D Navier-Stokes-Fourier equations}
\author[a,b]{Qiangchang Ju}
\author[c]{Rui Li}
\author[a,d]{Fanrui Meng \thanks{E-mail:
		ju\_qiangchang@iapcm.ac.cn (Q. Ju);
		ruili001@cuhk.edu.hk (R. Li);
		mengfanruimath@163.com (F. Meng).}}

\affil[a]{Institute of Applied Physics and Computational Mathematics, Beijing 100088, China.}
\affil[b]{National Key Laboratory of Computational Physics, Beijing 100088, China.}
\affil[c]{Department of Mathematics, The Chinese University of Hong Kong, Shatin, Hong Kong, China.}
\affil[d]{Graduate School of China Academy of Engineering Physics, Beijing 100193, China.}

\date{}

\begin{document}

	\maketitle
	\date{\vspace{-5ex}}
	
	\begin{abstract}
		
		We investigate the low Mach number limit of the three-dimensional full compressible Navier–Stokes-Fourier (NSF) equations on \(\mathbb{R}\times\mathbb{T}^2\) for two different classes of initial data, corresponding respectively to a well-prepared regime and an ill-prepared regime. The density and temperature are allowed to approach different asymptotic states at infinity. For the well-prepared regime, the solutions of compressible NSF equations converge to a planar diffusion wave solution globally in time as the Mach number tends to zero, where the difference between the states at the far fields is  small independently of the Mach number and the non-zero modes of the initial perturbations are exponentially small.
		Moreover, the optimal time-decay rate can be obtained.
		It can be viewed as the first global-in-time result on the low Mach number limit of three-dimensional NSF equations with large temperature variations.
		The corresponding local-in-time result is proved
		%by performing separate energy estimates for the zero and non-zero modes. 
		%the corresponding local-in-time result is obtained by decomposing 
		%and an auxiliary convergence lemma proposed by Métivier-Schochet in \cite{Métivier2001}.
		while the difference between the states at the far fields can be arbitrarily large for the ill-prepared data.
		
		
		
		% The low Mach limit around the diffusion wave for 3D non-isentropic compressible Navier–Stokes equations with large temperature variations is verified rigorously in Eulerian coordinates. 
		% Both well-prepared and ill-prepared initial data are considered. 
		% For the well-prepared initial data, the solutions of compressible Navier–Stokes equations converge to a planar entropy wave solution globally in time as the Mach number tends to zero, where the difference between the states at the far fields is suitably small and independent of the Mach number.
		% Moreover, the optimal time-decay rate can be obtained.
		% It can be viewed as the first global-in-time result on the low Mach number limit of full Navier-Stokes equations with large temperature variations.
		% For the ill-prepared initial data, the result is obtained by performing separate energy estimates for the zero and non-zero modes, and an auxiliary convergence lemma proposed by Métivier-Schochet in \cite{Métivier2001}.
		% It is remarked that the difference between the states at the far fields is allowed to be arbitrarily large.
	\end{abstract}
	\setcounter{tocdepth}{2}
	\tableofcontents
	
	\section{Introduction}
	We consider the three-dimensional(3D) full Navier-Stokes equations in Eulerian coordinates as follows
	\begin{equation}
		\left\{
		\begin{aligned}
			&\rho_t + \dv (\rho \mathbf{u}) = 0, \\
			&(\rho \mathbf{u})_t + \dv (\rho \mathbf{u} \otimes \mathbf{u}) + \nabla P = \dv \mathbb{S}, \\
			&\E_t + \dv (\E \mathbf{u} + p \mathbf{u} ) = \kappa \Delta \T + \dv (\mathbf{u} \mathbb{S}), \\
			&(\rho , \mathbf{u}, \T)(x,0) = (\rho_0, \mathbf{u}_0, \T_0),
		\end{aligned}
		\right.
	\end{equation}
	where $(x, t) \to \mathbb{R}^3 \times \mathbb{R}_+$. $\rho$, $u$ and $\T$ denote the density, velocity and absolute temperature of the fluid, respectively. P denotes the pressure and $\E$ denotes the total energy given by $\E := \rho (e + \frac{1}{2} |\mathbf{u}|^2)$ with $e$ being the internal energy. The pressure function and internal function are defined by
	\begin{align*}
		P =  R \rho\T, ~ e = c_v \T,
	\end{align*}
	where the parameters $R>0$ and $c_v >0$ are the gas constant and the heat capacity at constant volume, respectively.
	$\mathbb{S}$ is the viscous stress tensor given by
	\begin{align*}
		\mathbb{S} = 2\mu \mathbb{D}(\mathbf{u}) + \lambda \dv \mathbf{u} \mathbb{I},
	\end{align*}
	where $\mathbb{D}(\mathbf{u}) :=\frac{\nabla \mathbf{u} + (\nabla \mathbf{u})^t}{2} \in \mathbb{R}^{3 \times 3}$ represents the deformation tensor, $\mathbb{I}$ is the $3 \times 3$ identity matrix, and $(\nabla \mathbf{u})^t$ is the transpose of the matrix $\nabla \mathbf{u}$. The viscosity coefficients $\lambda$ and $\mu$ of the flow satisfy $\mu > 0$ and $2\mu + 3\lambda > 0$. For simplicity, we assume that $\mu$, $\lambda$ and $\kappa$ are constants.
	In this paper, since we focus on the low Mach number limit for smooth solutions of non-isentropic compressible Navier–Stokes equations, we set $R = 1$ and $c_v = 1$ for simplicity of presentation.
	
	In order to study the low Mach number approximation, the strategy is to introduce the Mach number $\e \in (0,1]$, which is a dimensionless number defined as the ratio of the reference flow velocity to the reference sound speed in the fluid. As in \cite{Alazard2005LowMN}, we consider the following change of variables
	\begin{align*}
		\rho (x,t) \to \rho^\e (x, \e t), \quad \mathbf{u} (x,t) \to \e \mathbf{u}^\e (x, \e t), \quad \T \to \T^\e (x, \e t),
	\end{align*}
	and
	\begin{align*}
		\mu \to \e \mu, \quad \lambda \to \e \lambda, \quad \kappa \to \e \kappa.
	\end{align*}
	Then we rewrite the 3D compressible Navier-Stokes equations in a non-dimensional form:
	\begin{equation} \label{nondimensional NS Eqs}
		\left\{
		\begin{aligned}
			&\partial_t \rho^{\e} + \text{div}(\rho^{\e} \mathbf{u}^{\e})=0, \\
			& \partial_t (\rho^{\e} \mathbf{u}^{\e}) + \dv (\rho^{\e} \mathbf{u}^{\e} \otimes \mathbf{u}^{\e}) + \frac{\nabla P^{\e}}{\e^2} = \dv\mathbb{S}(\mathbf{u}^\e),  \\
			&\p_t (\rho^{\e} (\T^{\e} + \frac{1}{2} | \e \mathbf{u}^{\e}|^2)) + \dv [\rho^{\e} (\T^{\e} + \frac{1}{2} | \e \mathbf{u}^{\e}|^2) \mathbf{u}^{\e} + \rho^{\e} \mathbf{u}^{\e} \T^{\e}] = \kappa \Delta \T^{\e} + \e^2 \dv (\mathbf{u}^{\e} \mathbb{S}(\mathbf{u}^\e)), \\
			&(\rho^\e , \mathbf{u}^\e, \T^\e)(x,0) = (\rho^\e_0, \mathbf{u}^\e_0, \T^\e_0),
		\end{aligned}
		\right.
	\end{equation}
	We consider that the pressure is a small perturbation of $\underline{P}$:
	\begin{align*}
		P = \underline{P} + O(\e),
	\end{align*}
	which $\underline{P} > 0$ is a given constant state and can be normalized to 1. According to $P^\e = \rho^\e \T^\e$, it is straightforward to obtain the equation of $P^\e$:
	\begin{align}\label{P eq}
		\p_t P^\e  + \dv (\mathbf{u}^\e P^\e) + P^\e \dv \mathbf{u}^\e = \dv (\kappa \nabla \T^\e) + \e^2 [2\mu |\mathbb{D}(\mathbf{u}^\e)|^2 + \lambda |\dv \mathbf{u}^\e|^2].
	\end{align}
	As the Mach number tends to zero, (\ref{P eq}), $(\ref{nondimensional NS Eqs})_2$ and $(\ref{nondimensional NS Eqs})_3$ formally become
	\begin{equation}\label{the limit of NS Eqs}
		\left\{
		\begin{aligned}
			&\dv (2 \mathbf{u} - \kappa \nabla \T) =0, \\
			&\rho (\mathbf{u}_t + \mathbf{u} \cdot \nabla \mathbf{u}) + \nabla \pi = \mu \Delta \mathbf{u} + (\lambda + \mu) \nabla \dv \mathbf{u}, \\
			&\rho (\T_t + \mathbf{u} \cdot \nabla \T) + \dv \mathbf{u} = \kappa \Delta\T,
		\end{aligned}
		\right.
	\end{equation}
	where $\pi$ is the formal limit of  $\frac{P^\e - 1}{\e^2}$, $\rho = \frac{1}{\T}$ and 
	\begin{align}\label{the equs of rho}
		\rho_t + \dv (\rho \mathbf{u})= 0.
	\end{align}
	
	The low Mach number limit is an interesting topic in fluid dynamics and applied mathematics. We will briefly review some related results on the Euler and Navier–Stokes equations. For the isentropic Euler equations, Klainerman and Majda developed foundational uniform estimates and incompressible limit results for well-prepared data in periodic domains in \cite{Klainerman1981}, and in the whole space in \cite{Klainerman1982}. 
	The case of ill-prepared initial data was studied in the whole space by Ukai\cite{Ukai1986}, in the periodic domain by Schochet\cite{Schochet1994}, and in the bounded domain by Secchi\cite{Secchi2000}.
	%Ukai\cite{Ukai1986} studied the whole-space problem with ill-prepared data by using the fast decay of acoustic waves. 
	For the non-isentropic cases, Métivier and Schochet\cite{Métivier2001} established the incompressible limit in the whole space by exploiting the local decay of acoustic energy. Moreover, Alazard\cite{Alazard2005} subsequently obtained uniform regularity estimates for the non-isentropic Euler equations in domains with smooth compact boundaries, including bounded and exterior domains, and justified convergence in the exterior-domain case for ill-prepared data. There are also many other interesting works, one can see for example \cite{Iguchi1997, Schochet1986} and the references therein.
	
	For the isentropic Navier-Stokes equations, Lions and Masmoudi\cite{Lions1998} established incompressible limit results for finite-energy weak solutions in the whole space, periodic domains, and bounded domains with suitable boundary conditions, including slip conditions. Desjardins and Grenier\cite{Desjardins1999} obtained local strong convergence in the whole space using acoustic dispersive estimates. In the strong solution framework, Danchin\cite{Danchin2002} developed the low Mach number limit in critical spaces in the whole space. In bounded domains with no-slip boundary conditions, Desjardins, Grenier, Lions and Masmoudi\cite{DJLM1999} showed that acoustic boundary layers yield strong velocity convergence for ill-prepared initial data under an appropriate geometric condition. For general Navier slip conditions in three-dimensional bounded domains, Masmoudi, Rousset and Sun\cite{Masmoudi2002} established uniform conormal regularity and justified the incompressible limit for ill-prepared initial data. In exterior domains with slip boundary conditions, Donatelli, Feireisl and Novotný\cite{Donatelli2010} justified the incompressible limit using local acoustic energy decay. For other interesting works, see \cite{Feireisl2011, Ou2014, Ou2022, Xiong2018} and the references therein.
	
	For the non-isentropic Navier-Stokes equations, the heat conductivity and the temperature variations play essential roles.  When the heat conductivity vanishes, Kim and Lee\cite{Kim2005} established the incompressible limit for well-prepared data in the whole space \(\R^3\), while Jiang and Ou\cite{Jiang2011} treated a three-dimensional bounded domain with no-slip boundary conditions. For positive heat conductivity, Feireisl and Novotný\cite{Feireisl2007} justified the low Mach number limit of weak solutions to the Navier-Stokes–Fourier system in the periodic domain \(\mathbb{T}^3\) for ill-prepared initial data. In three-dimensional bounded domains, Dou, Jiang and Ou\cite{Dou2015} established the low Mach number limit of local strong solutions to the full Navier-Stokes equations for well-prepared data. For global strong solutions, Ren and Ou studied the same problem in \cite{Ren20161, Ren20162}, considering slip and non-slip boundary conditions respectively. It is noted that the results in references \cite{Dou2015, Ren20161, Ren20162} all rely on the assumption of small temperature variations. The large temperature variations bring more difficulties, since the density is no longer close to a constant state and spatial differentiation generates variable-coefficient terms that are difficult to control. For large temperature variations, Alazard\cite{Alazard2005LowMN} established uniform regularity in the whole space and periodic domains, and proved strong local convergence in the whole space under additional far-field assumptions. Ju-Ou\cite{Ju2022} verified the low Mach number limit of Navier-Stokes equations in a 3D bounded domain with the vorticity-slip condition for well-prepared initial data, while Sun\cite{Sun2022} considered the same problem in a smooth domain with general initial data and the Navier-slip boundary condition. It is worth noting that results concerning global solutions in the regime of large temperature variations remain scarce. The major result is due to Huang, Wang and Wang\cite{HWW}, in which they justified the low Mach number limit for 1D Navier–Stokes flows, whose density and temperature have different asymptotic states at infinity. They proved the convergence globally in time as Mach number goes to zero for well-prepared data. 
	Recently, Li and Yin\cite{Li2026} extended this result to the one-dimensional Navier–Stokes–Korteweg system.
	To the best of our knowledge, no results on global solutions with large temperature variations are currently available in higher dimensions.

	The aim of the current paper is to verify the low Mach number limit of compressible non-isentropic 3D Navier-Stokes equations in $\Omega = \R\times\Torus^2$, where a special condition is mentioned:
	\begin{align}\label{entropy wave condition}
		(\rho^\e, \T^\e) (x,t) \to (\rho_{\pm}, \T_{\pm}) \quad \text{as} \quad x_1 \to \pm \infty \quad \text{with} \quad \rho_- \T_- = \rho_+ \T_+,
	\end{align}
	and \(\mathbf{u} \to 0\) as \(x_1 \to \pm \infty\). We start from the construction of the special solution of (\ref{the limit of NS Eqs}): 
	\[(\hat{\rho}, \hat{u}_1, 0, 0, \hat{\T}) = (\hat{\rho}(x_1), \hat{u}_1 (x_1), 0, 0, \hat{\T}(x_1)).\]
	According to \((\ref{the limit of NS Eqs})_1\), we choose
	\begin{align}\label{hat_u1}
		\hat{u}_1 =\frac{\kappa}{2} \partial_{x_1} \hat{\T} = -\frac{\kappa}{2} \frac{\partial_{x_1} \hat{\rho}}{|\hat{\rho}|^2}.
	\end{align}
	Substituting (\ref{hat_u1}) into (\ref{the equs of rho}) leads to
	\begin{align}\label{diffusion wave of rho}
		\hat{\rho}_t - \partial_{x_1} \big(\frac{\kappa \partial_{x_1} \hat{\rho}}{2 \hat{\rho}}\big) =0.
	\end{align}
	From \cite{Hsiao1993,Liu1997}, it is known that the nonlinear diffusion equation (\ref{diffusion wave of rho}) admits a unique self-similar solution \(\hat{\rho}(\frac{x_1}{\sqrt{1+t}})\), satisfying
	\begin{align*}
		\partial_{x_1} \hat{\rho} = \frac{O(1) \delta}{\sqrt{1+t}} e^{-\frac{d_0 x_1^2}{1+t}},
	\end{align*}
	where \(\delta\) denotes the strength of the diffusion wave. Indeed, setting
	\begin{align}\label{special solution of the limit system}
		\hat{\rho} = \hat{\rho}, \quad\quad \hat{\mathbf{u}} =(-\frac{\kappa}{2} \frac{\partial_{x_1} \hat{\rho}}{|\hat{\rho}|^2}, 0, 0), \quad\quad \hat{\T} = \frac{1}{\hat{\rho}},
	\end{align}
	we can easily find that \((\hat{\rho}, \hat{u}, \hat{\T})\) is a special solution of (\ref{the limit of NS Eqs}), that is,
	\begin{equation}
		\left\{
		\begin{aligned}
			&\partial_{x_1} (2\hat{u}_1 - \kappa \partial_{x_1} \hat{\T}_1) = 0, \quad \hat{\rho} = \frac{1}{\hat{\T}}, \\
			&\hat{\rho} (\partial_t \hat{u}_1 + \hat{u}_1 \partial_{x_1} \hat{u}_1) + \partial_{x_1} \hat{\pi} = (2\mu+\lambda) \partial_{x_1}^2 \hat{u}_1, \\
			&\hat{\rho} (\partial_t \hat{\T} + \hat{u}_1 \partial_{x_1} \hat{\T}) + \partial_{x_1} \hat{u}_1 = \kappa \partial_{x_1}^2 \hat{\T},
		\end{aligned}
		\right.
	\end{equation}
	where \(\hat{\pi} = (2\mu+\lambda) \partial_{x_1} \hat{u}_1 - \hat{\rho} \hat{u}_1^2 + \frac{\kappa}{2} \frac{\partial_t \hat{\rho}}{\hat{\rho}}\).
	It should be noted that (\ref{special solution of the limit system}) is not a solution to the original equations (\ref{nondimensional NS Eqs}). Define
	\begin{align}\label{bar}
		(\bar{\rho}, \bar{\mathbf{u}}, \bar{\T})= (\hat{\rho}, \hat{\mathbf{u}}, \hat{\T} - \frac{1}{2}|\varepsilon \bar{u}_1|^2) \quad \text{and} \quad \bar{p} = \bar{\rho} \bar{\T}.
	\end{align}
	A direct calculation gives that
	\begin{equation}
		\left\{
		\begin{aligned}
			&\partial_t \bar{\rho} + \partial_{x_1} (\bar{\rho} \bar{u}_1) =0, \\
			&\varepsilon^2 \partial_t (\bar{\rho} \bar{u}_1) + \partial_{x_1}(\bar{\rho} |\varepsilon \bar{u}_1|^2 + \bar{p}) = (2\mu +\lambda) \varepsilon^2 \partial_{x_1}^2\bar{u}_1 + \partial_{x_1} Q_1, \\
			&\partial_t (\bar{\rho} (\bar{\T} + \frac{1}{2}|\varepsilon \bar{u}_1|^2)) + \partial_{x_1} (\bar{\rho} \bar{u}_1 (\bar{\T} + \frac{1}{2}|\varepsilon \bar{u}_1|^2) + \bar{p} \bar{u}_1) = \kappa \partial_{x_1}^2 \bar{\T} + (2\mu+\lambda) \varepsilon^2 \partial_{x_1} (\bar{u}_1 \partial_{x_1} \bar{u}_1) + \partial_{x_1} Q_2,
		\end{aligned}
		\right.
	\end{equation}
	where
	\begin{align*}
		&Q_1 = -\varepsilon^2 \frac{\kappa}{2} \frac{\partial_t \bar{\rho}}{\bar{\rho}} + \frac{1}{2} \bar{\rho} |\varepsilon \bar{u}_1|^2 - (2\mu+\lambda) \varepsilon^2 \partial_{x_1} \bar{u}_1 = O(1) \delta \varepsilon^2 (1+t)^{-1} e^{-\frac{dx_1^2}{1+t}}, \\
		&Q_2 = -\frac{1}{2} \varepsilon^2 \bar{\rho} \bar{u}_1^3 + \frac{1}{2} \kappa  \partial_{x_1} (|\varepsilon \bar{u}_1|^2) - (2\mu+\lambda) \e^2 (\bar{u}_1 \partial_{x_1} \bar{u}_1) = O(1) \delta \varepsilon^2 (1+t)^{-\frac{3}{2}} e^{-\frac{dx_1^2}{1+t}}.
	\end{align*}
	It can be observed that 
	\begin{align}
		\|(\bar{\rho} - \hat{\rho}, \bar{\mathbf{u}} - \hat{\mathbf{u}}, \bar{\T} - \hat{\T}) \|_{L^\infty(\Omega)} \le C \varepsilon^2 (1+t)^{-1},
	\end{align}
	which suggests that \((\bar{\rho}, \bar{u}, \bar{\T})\) provides a very good approximation to the diffusive wave solution \((\hat{\rho}, \hat{u}, \hat{\T})\) when \(\varepsilon\) is small.
	
	We will study the low Mach number limit near the planar diffusion waves. The paper is organized as follows. In Section \ref{section2}, we construct a new ansatz and state our main theorems. In Section \ref{section3}
	we consider the low Mach number limit for the well-prepared data and in Section \ref{section4} for the ill-prepared data.

	\subsection*{Notation}
	\begin{itemize}
		\item To state our main theorems, we need to introduce the zero and non-zero modes in Fourier space. We set
		\begin{align*}
			\mathbf{D}_0 f^\e ( x_1, t) := \mathring{f}^\e ( x_1, t) = \int_{\mathbb{T}^2} f^\e d x_2 d x_3, \qquad \mathbf{D}_{\neq} f^\e ( x, t) := f^\e_{\neq} ( x, t) := f^\e( x, t)-\mathring{f}^\e( x_1, t),
		\end{align*}
		for any function \(f^\e\) integrable on \(\mathbb T^2\).
		
		When we introduce the scaling \eqref{transformation}, the zero mode is 
		\begin{align*}
			\mathbf{D}_0 f(y_1, \tau) := \mathring{f}(y_1, \tau)=\varepsilon^2 \int_{(\frac{\Torus}{\varepsilon})^2} f (y, \tau) dy_2dy_3,
		\end{align*}
		and the non-zero mode is
		\begin{align*}
			\mathbf{D}_{\neq} f(y, \tau) := f_{\neq} (y, \tau) = f(\tau, y) - \mathring{f}(y_1, \tau),
		\end{align*} 
		for any function $f(y, \tau)=f^\varepsilon(x, t)$ integrable on \(\frac{\mathbb{T}^2}{\e^2}\). 
		
		\item The anti-derivative of \(f(t,x)\) is defined by 
		\begin{align*}
			F^\e(x_1,t):=\int_{-\infty}^{x_1} \mathring{f}^\e(s,t)ds=\int_{-\infty}^{x_1}\int_{\Torus^2}f^\e(x,t)dx_1dx_2dx_3.
		\end{align*}
		Under the scaling \eqref{transformation}, the anti-derivative of $f(y,\tau)$ is
		\begin{align*}
			& F(y_1,\tau):=\int_{-\infty}^{y_1}\mathring{f}(s,\tau) ds =\varepsilon^2\int_{-\infty}^{y_1}\int_{(\frac{\Torus}{\varepsilon})^2} f(s,y_2,y_3,\tau)dsdy_2dy_3.
		\end{align*}
		
		\item Commutators are defined by \([\mathcal{L}, f]g = \mathcal{L}(fg)- f(\mathcal{L} g)\), \([f,\mathcal{L}]g = -[\mathcal{L}, f]g\), where \(\mathcal{L}\) is a differential operator and \(f, g\) are functions.
		
		\item The Frobenius product of the matrix \(A=(A_{ij})\) and \(B=(B_{ij})\) is defined by \(A :B := \sum_{i,j} A_{ij}B_{ij}\).
		
		\item \(c\), \(c_i\), \(C\), \(C_i\) and \(\tilde{C}_i\), \(i \in \mathbb{N}\) are always used to denote positive constants independent of \(\e\).

	\end{itemize}
	We  describe the relationship between the norms of the original and rescaled functions.
	\begin{Rem}
		According to the definition of the zero mode and the non-zero mode, we have
		\begin{align}\label{relation1}
			\norm{f_0}_{L_y^2({\R\times(\frac{\Torus}{\varepsilon})^2)}}^2=\norm{\mathring f_0}_{L_y^2({\R{\times(\frac{\Torus}{\varepsilon})^2})}}^2+\norm{f_{0\neq}}_{L_y^2({\R\times(\frac{\Torus}{\varepsilon})^2)}}^2={\varepsilon^{-2}}\norm{\mathring f_0}_{L_y^2({\R)}}^2+\norm{f_{0\neq}}_{L_y^2({\R\times(\frac{\Torus}{\varepsilon})^2)}}^2,
		\end{align}
		for any scaled $f(y,\tau)$.
		Moreover, the relationship between the norms of $f^\varepsilon(x,t)$ and $f(y,\tau)$ can be described as follows:
		\begin{align}\label{relation2}
			\norm{\nabla_x^k f^{\varepsilon}}_{L_x^2(\R\times\Torus^2)}^2 = \varepsilon^{3-2k}\norm{\nabla_y^k f}_{L_y^2(\R\times(\frac{\Torus}{\varepsilon})^2)}^2,\qquad  \norm{\p_{x_1}^k \mathring{f}^{\varepsilon}}_{L_x^2(\R)}^2=\varepsilon^{1-2k}\norm{\p_{y_1}^k \mathring{f}}_{L_y^2(\R)}^2.
		\end{align}
	\end{Rem}
	
	\begin{Rem}
		The relationship between \(F^{\varepsilon}(x_1,t)\) and \(F(y_1,\tau)\) are given by
		\begin{align*}
			& F(y_1,\tau):=\int_{-\infty}^{y_1}\varepsilon^2\int_{(\frac{\Torus}{\varepsilon})^2} f(s,y_2,y_3,\tau)dsdy_2dy_3\\
			&\qquad\quad\;\;\;=\int_{-\infty}^{y_1}\int_{(\frac{\Torus}{\varepsilon})^2} f^{\varepsilon}(\varepsilon s,\varepsilon y_2,\varepsilon y_3,\varepsilon^2 \tau) ds d(\varepsilon y_2)d(\varepsilon y_3)\\
			% &\qquad\quad\;\;\;=\int_{-\infty}^{y_1}\int_{\Torus^2} f^{\varepsilon}(\varepsilon s,x_2,x_3,t) ds dx_2 dx_3=\varepsilon^{-1}\int_{-\infty}^{y_1}\int_{\Torus^2} f^{\varepsilon}(\varepsilon s,x_2,x_3,t) d(\varepsilon s) dx_2 dx_3\\
			&\qquad\quad\;\;\;=\varepsilon^{-1}\int_{-\infty}^{\varepsilon y_1}\int_{\Torus^2} f^{\varepsilon}(s',x_2,x_3,t) ds' dx_2 dx_3 =\varepsilon^{-1} F^{\varepsilon}(x_1,t).
		\end{align*}
		As for the relationship between the norm, we have
		\begin{align*}
			&\norm{F(y_1,\tau)}_{L_y^2(\R)}^2=\varepsilon^{-3}\norm{F^{\varepsilon}(x_1,t)}_{L_x^2(\R)}^2 ,\qquad\p_{y_1} F(y_1,\tau)=\varepsilon^{-1} \p_{y_1} F^{\varepsilon}(x_1,t)=\p_{x_1}F^{\varepsilon}(x_1,t).
			% &\norm{\mathring{f}(y_1,\tau)}_{L_y^2(\R)}^2=\norm{\p_{y_1} F(y_1,\tau)}_{L_y^2(\R)}^2=\varepsilon^{-1}\norm{\p_{x_1}F^{\varepsilon}(x_1,t)}_{L^2_x(\R)}^2=\varepsilon^{-1}\norm{\mathring{f}^{\varepsilon}(x_1,t)}_{L_x^2(\R)}^2.
		\end{align*}
	\end{Rem}

	\section{Construction of ansatz and main theorems} \label{section2}
	
	\subsection{Well-prepared initial data}
	Introduce the following scaled variables
	\begin{align}\label{transformation}
		y=\frac x \varepsilon,\qquad \tau=\frac {t} {\varepsilon^2},
	\end{align}
	and set 
	\begin{align*}
		&\mathbf{m}^\e = \rho^\e \mathbf{u}^\e, \quad \E^\e = \rho^{\varepsilon}(\T^{\varepsilon}+\frac{|\varepsilon \mathbf{u}^{\varepsilon}|^2}{2})\\
		&\nabla_y:=(\p_{y_1},\p_{y_2},\p_{y_3}),\quad \dv_{y}:=\nabla_y \cdot,\quad \Delta_y:=\p_{y_1}^2+\p_{y_2}^2+\p_{y_3}^2.
	\end{align*}
	Here we omit the superscript \(\e\) of the variables after the transformation (\ref{transformation}).
	Then the new unknown function $(\rho,\e \mathbf{m},\E)(y,\tau)$ and \((\bar{\rho}, \e \bar{\mathbf{m}}, \bar{\E})(y,\tau)\) satisfy, respectively,
	\begin{equation}\label{scaled NS Eqs}
		\left\{
		\begin{aligned}
			&\partial_{\tau} \rho + \dv_{y} (\varepsilon \mathbf{m}) = 0, \\
			&\partial_{\tau}(\varepsilon \mathbf{m}) + \dv_{y} (\varepsilon \mathbf{m} \otimes \varepsilon \mathbf{u}) + \nabla_y p = \dv_{y} \mathbb{S}(\e \mathbf{u}), \\
			&\partial_{\tau} \E + \dv_{y} [\e \mathbf{u} \E + \e  \rho \mathbf{u} \T] = \kappa \Delta_y \T + \dv_{y}(\e \mathbf{u} \mathbb{S}(\e \mathbf{u})),
		\end{aligned}
		\right.
	\end{equation}
	and
	\begin{equation}
		\left\{
		\begin{aligned}
			&\partial_{\tau} \bar{\rho} + \partial_{y_1}(\e \bar{m}_1) = 0, \\
			&\partial_{\tau} (\e \bar{m}_1) + \partial_{y_1}(\e \bar{m}_1 \e \bar{u}_1) + \partial_{y_1} \bar{p} = (2\mu+\lambda) \partial_{y_1}^2 (\e \bar{u}_1) + \partial_{y_1} \bar{Q}_1, \\
			&\partial_{\tau} \bar{\E} + \partial_{y_1} (\e \bar{u}_1\bar{\E} + \e \bar{\rho} \bar{u}_1 \bar{\T}) = \kappa \partial_{y_1}^2 \bar{\T} + (2\mu+\lambda) \partial_{y_1}(\e \bar{u}_1 \partial_{y_1}(\e \bar{u}_1)) + \partial_{y_1} \bar{Q}_2,
		\end{aligned}
		\right.
	\end{equation}
	where
	\begin{align*}
		&\bar{Q}_1 = -\frac{\kappa}{2} \frac{\p_\tau \bar{\rho}}{\bar{\rho}} + \frac{1}{2}\bar{\rho} |\e \bar{u}_1|^2 - (2\mu+\lambda)\p_{y_1} (\e \bar{u}_1) = O(1) \delta \e^2 (1+\e^2 \tau)^{-1} e^{-\frac{d \e^2 y_1^2}{1+\e^2 \tau}} ,\\
		&\bar{Q}_2 = -\frac{1}{2} \bar{\rho} (\e \bar{u}_1)^3 + \frac{\kappa}{2} \p_{y_1}(|\e \bar{u}_1|^2) - (2\mu+\lambda) \e \bar{u}_1 \p_{y_1}(\e \bar{u}_1)= O(1) \delta \e^3 (1+\e^2 \tau)^{-\frac{3}{2}} e^{-\frac{d \e^2 y_1^2}{1+\e^2 \tau}}.
	\end{align*}

	\subsubsection{The extra initial mass and the diffusion waves}
	For convenience, we denote the conserved quantities by
	\begin{align*}%\label{conserved-law}
		& U^\varepsilon=(\rho^\varepsilon,\varepsilon(m_1^\varepsilon,m_2^\varepsilon,m_3^\varepsilon),\E^\varepsilon)^{t},\;\bar{U}^\varepsilon=(\bar{\rho}^\varepsilon,\varepsilon(\bar{m}^\varepsilon_1, \bar{m}^\varepsilon_2,\bar{m}^\varepsilon_3),\bar{\E}^\varepsilon)^{t},\; U^{\varepsilon\#}=(\rho^\varepsilon,\varepsilon m_1^\varepsilon,\E^\varepsilon)^{t}, \;\bar{U}^{\varepsilon\#}=(\bar{\rho}^\varepsilon,\varepsilon\bar{m}^\varepsilon_1, \bar{\E}^\varepsilon)^{t},\\
		&  U=(\rho,\varepsilon(m_1,m_2,m_3),\E)^{t},\quad\;\bar{U}=(\bar{\rho},\varepsilon(\bar{m}_1, \bar{m}_2,\bar{m}_3),\bar{\E})^{t},\quad\; U^{\#}=(\rho,\varepsilon m_1,\E)^{t}, \quad\;\bar{U}^{\#}=(\bar{\rho},\varepsilon\bar{m}_1, \bar{\E})^{t},  
	\end{align*}
	where $\E^{\varepsilon}=\rho^{\varepsilon}(e^{\varepsilon}+\frac{|\varepsilon \mathbf{u}^{\varepsilon}|^2}{2})$ and $\E=\rho(e+\frac{|\varepsilon \mathbf{u}|^2}{2})$.  
	%    We define 
	% \begin{align*}
		% 0\neq A=\int_{\R\times\Torus^2} U^{\varepsilon}-\bar{U}^{\varepsilon} dx \le \norm{U^{\varepsilon}-\bar{U}^{\varepsilon}}_{L^1(\R\times\Torus^2)}.
		% %:=\delta\varepsilon^{\ast}.
		% \end{align*}
	% And we also have
	% \begin{align*}
		% \varepsilon^{-2} \int_{\R} (\mathring{U}-\bar{U}^{\varepsilon}) dy_1 = \varepsilon^{-3}A, \quad \Rightarrow \quad \int_{\R} (\mathring{U}-\bar{U}^{\varepsilon}) dy_1 = \varepsilon^{-1} A.
		% \end{align*}
	We will construct diffusion waves to carry the extra initial mass, which is inspired by \cite{Liu1,SX,HXY}. Note that the extra initial mass is distributed along the $x_1$-direction. The Jacobi matrices for the flux of 1-d Navier-Stokes equations in Eulerian coordinates \cref{scaled NS Eqs} at $(\rho_+,\varepsilon m_{1+},\E_+)$ and $(\rho_-,\varepsilon m_{1-},\E_-)$ are given, respectively, as 
	\begin{align*}
		A_\pm=\left(\begin{array}{ccc}
			0&1&0\\
			0&0&1\\
			0&2\frac{\E_\pm}{\rho_\pm}&0
		\end{array}\right).
	\end{align*}
	Here we have used \(m_{1\pm}=0\), since
	the planar diffusion wave under consideration has zero far-field velocity.
	A direct computation then shows that $\lambda_{1}^-=-\sqrt{2\frac{\E_-}{\rho_-}}$ is the first eigenvalue of $A_-$ corresponding with $r_1^{-}=(1,\lambda_1^-,(\lambda_1^{-})^2)^t$ and $\lambda_{3}^+=\sqrt{2\frac{\E_+}{\rho_+}}$ is the third eigenvalue of $A_+$ corresponding with $r_3^{+}=(1,\lambda_3^+,(\lambda_3^{+})^2)^t$. Since the following three vectors, $r_1^-$, $(\rho_--\rho_+,\varepsilon(m_{1-}-m_{1+}),\E_--\E_+)$ and $r_3^+$ are linearly independent, we have
	\begin{align}\label{eq-initial mass}
		\int_{\R}\mathring{U}^{\#}(y_1, 0)-\bar{U}^{\#}(y_1, 0)dy_1=\bar{\Theta}_1r_1^-+\bar{\Theta}_2(\rho_--\rho_+,\varepsilon(m_{1-}-m_{1+}),\E_--\E_+)^t+\bar{\Theta}_3r_3^+,
	\end{align}
	where the coefficients \(\bar{\Theta}_i, (i = 1,2,3)\) are determined by the initial data \((\bar{\rho}, \e \bar{m}_1, \bar{\E})(y_1, 0)\).
	We will couple two diffusion waves with $(\bar{\rho},\e\bar m_1,\bar{\mathcal E})(y_1,\tau)$, which are defined by
	\begin{align*}
		\Theta_1(y_1,\tau)=\frac{1}{\sqrt{4\pi(\varepsilon^{-2}+\tau)}}e^{-\frac{(y_1-\lambda_1^-(\varepsilon^{-2}+\tau))^2}{4(\varepsilon^{-2}+\tau)}},\qquad\Theta_3(y_1,\tau)=\frac{1}{\sqrt{4\pi(\varepsilon^{-2}+\tau)}}e^{-\frac{(y_1-\lambda_3^+(\varepsilon^{-2}+\tau))^2}{4(\varepsilon^{-2}+\tau)}}.
	\end{align*}
	It is direct to verify that they satisfy
	\begin{align*}
		\left\{\begin{aligned}
			&\p_\tau\Theta_{1}+\lambda_1^{-}\p_{y_1}\Theta_{1}=\p_{y_1}^2\Theta_{1},\\
			& \int_{-\infty}^{+\infty}\Theta_1(y_1,\tau)dy_1=1,
		\end{aligned}\right.\qquad\left\{\begin{aligned}
			&\p_{\tau}\Theta_{3}+\lambda_3^{+}\p_{y_1}\Theta_{3}=\p_{y_1}^2\Theta_{3},\\
			&\int_{-\infty}^{+\infty}\Theta_3(y_1,\tau)dy_1=1.
		\end{aligned}\right.
	\end{align*}
	Let
	\begin{align*}
		&\bar{\Theta}_{i+2}:=\varepsilon\int_{\R}\mathring{m}_i(y_1,0)dy_1, \\
		&\Theta_{i+2} := \frac{1}{\sqrt{4 \pi(\varepsilon^{-2}+\tau)}} e^{-\frac{y_1^{2}}{4(\varepsilon^{-2}+\tau)}}, \qquad i=2,3,
	\end{align*}
	Then we define the new ansatz $(\tilde{\rho},\varepsilon\tilde{\mathbf{m}},\tilde{\E})(y_1,\tau)$ by
	\begin{align}\label{tilde}
		\begin{aligned}
			&\tilde{\rho}(y_1,\tau)={\bar{\rho}}\left(y_1-\bar{\Theta}_2, \tau\right)+\bar{\Theta}_{1} \Theta_{1}+\bar{\Theta}_{3} \Theta_{3}, \\
			&\varepsilon\tilde{m}_{1}(y_1,\tau)=\varepsilon{\bar{m}}_{1}\left(y_1-\bar{\Theta}_2, \tau\right)+\lambda_{1}^{-} \bar{\Theta}_{1} \Theta_{1}+\lambda_{3}^{+} \bar{\Theta}_{3} \Theta_{3}, \\
			&\varepsilon\tilde{m}_{i}(y_1,\tau)=\bar{\Theta}_{i+2} \Theta_{i+2}, \quad i=2,3, \\
			&\tilde{\E}(y_1,\tau)=\bar{\E}\left(y_1-\bar{\Theta}_2, \tau\right)+\left((\lambda_1^-)^2\bar{\Theta}_{1} \Theta_{1}+(\lambda_3^+)^2\bar{\Theta}_{3} \Theta_{3}\right).
		\end{aligned}
	\end{align}
	% It follows that all the new variables introduced in (\ref{tilde}) are well-defined.
	Without loss of generality, we assume $\bar{\Theta}_2=0$ \cite{HXY}.
	By a slight abuse of notation, we identify
	\((a,b,c)^t\in\mathbb{R}^3\) with
	\((a,b,0,0,c)^t\in\mathbb{R}^5\) in the following identity.
	Let \(\mathbf e_j\) denote the \(j\)-th standard basis vector
	of \(\mathbb{R}^5\).
	Then the initial perturbation relative to the new ansatz
	\eqref{tilde} satisfies the zero-mass condition
	\begin{align*}
		\int_{\mathbb{R}}
		\left[\mathring{U}(y_1,0)-\tilde{U}(y_1,0)\right]\,dy_1 =
		\int_{\mathbb{R}}
		\left[\mathring{U}^{\#}(y_1,0)-\bar{U}^{\#}(y_1,0)\right]\,dy_1
		+\sum_{i=2}^{3}\bar{\Theta}_{i+2}\mathbf e_{i+1} +
		\int_{\mathbb{R}}
		\left[\bar{U}(y_1,0)-\tilde{U}(y_1,0)\right]\,dy_1
		= 0.
	\end{align*}
	% Then the initial extra mass for new ansatz \cref{tilde} is
	%\todo{A minor issue: Dimensions are not consistent?} 
	% \begin{align*}%\label{eq-initial mass-1}
		% \begin{aligned}
			% &\int_{\R}\left[\mathring{U}(y_1,0)-\tilde{U}(y_1,0)\right]dy_1=\int_{\R}\left[\mathring{U}^{\#}(y_1,0)-\bar{U}^{\#}(y_1,0)\right]dy_1+\sum_{i=2}^3\bar{\Theta}_{i+2}+\int_{\R}\left(\bar{U}-\tilde{U}\right)(y_1,0)dy_1=0,
			% \end{aligned}
		% \end{align*}
	where \(\tilde{U} = (\tilde{\rho}, \e\tilde{m}_1, \e\tilde{m}_2, \e\tilde{m}_3, \tilde{\E})\).
	It is easy to see that
	\begin{align*}
		\sum_{i=1}^5  \abs{\bar{\Theta}_i} \le C\varepsilon^{-1} \norm{(U^{\e}-\bar{U}^{\e})(y, 0)}_{L^1(\Omega)}.
		%\le C\delta\varepsilon^{-1}\varepsilon^{\ast}.
	\end{align*}
	% And we  notice that
	% \begin{align*}
		% & \p_{y_1}^{k+1} \bar{\rho} = \p_{y_1}^k (\e \bar{u}_1) = O(1)\delta\varepsilon^{k+1} (1+\varepsilon^2 \tau)^{-\frac{k+1}{2}}e^{-\frac{\varepsilon^2 y_1^2}{1+\varepsilon^2 \tau}} =O(1)\delta(\varepsilon^{-2}+\tau)^{-\frac{k+1}{2}} e^{-\frac{y_1^2}{\varepsilon^{-2}+\tau}},\\
		% & \p_{y_1}^k(\bar{\Theta}_1\Theta_1)= O(1) A \varepsilon^{-1} (\varepsilon^{-2}+\tau)^{-\frac{k+1}{2}} e^{-\frac{(y_1-\lambda_1^{-} (\varepsilon^{-N}+\tau) )^2}{\varepsilon^{-N}+\tau}}, \\
		% %\delta\varepsilon^{\ast}\varepsilon^{-1} (\varepsilon^{-N}+\tau)^{-\frac{k+1}{2}} e^{-\frac{(y_1-\lambda_1^{-} (\varepsilon^{-N}+\tau) )^2}{\varepsilon^{-N}+\tau}},
		% & \p_{y_1}^k (\varepsilon \tilde{m}_i)=O(1) A \varepsilon^{-1} (\varepsilon^{-2}+\tau)^{-\frac{k+1}{2}} e^{-\frac{y_1^2}{\varepsilon^{-2}+\tau}}.
		% %O(1)\delta \varepsilon^{\ast} \varepsilon^{-1} (\varepsilon^{-N}+\tau)^{-\frac{k+1}{2}} e^{-\frac{y_1^2}{\varepsilon^{-N}+\tau}}.
		% \end{align*}
	% So, we choose $A=\varepsilon^2$ and $N=2$, then 
	% %So, we choose $\varepsilon^{\ast}=\varepsilon$ and $N=2$, then 
	% \begin{align*}
		% \norm{\p_{y_1}^{k+1} \bar{\rho}}_{L^2_y}\approx\norm{ \p_{y_1}^k(\e \bar{u}_1)}_{L^2_y}\approx\norm{ \p_{y_1}^k(\bar{\Theta}_1\Theta_1)}_{L^2_y}\approx\norm{\p_{y_1}^k (\varepsilon \tilde{m}_i)}_{L^2_y},
		% \end{align*}
	% which means entropy waves are similar with diffusion waves.
	
	Moreover, for \(i=1,2,3,\) non-conserved quantities, the velocity \(\tilde{u}\), the temperature \(\tilde{\T}\) and the pressure \(\tilde{p}\), can be defined as follows:
	\begin{align}
		\tilde{u}_i := \frac{\tilde{m}_i}{\tilde{\rho}}, \quad \tilde{\T} :=\frac{\tilde{\E}}{\tilde{\rho}} - \frac{|\e \tilde{\mathbf{m}}|^2}{2\tilde{\rho}^2}, \quad \tilde{p} = \tilde{\E} -\frac{|\e \tilde{\mathbf{m}}|^2}{2\tilde{\rho}}.
	\end{align}
	A direct calculation gives that
	\begin{equation}\label{tilde U equs}
		\left\{
		\begin{aligned}
			&\p_\tau \tilde{\rho} + \dv_{y} (\e \tilde{\mathbf{m}}) = \p_{y_1} \tilde{\Gamma}_1, \\
			&\p_\tau (\e \tilde{m}_1) + \dv_{y} \big(\frac{\e \tilde{m}_1 \e \tilde{\mathbf{m}}}{\tilde{\rho}}\big) + \p_{y_1} \tilde{p} = \mu \Delta_y (\e \tilde{u}_1) + (\mu + \lambda) \p_{y_1} (\e \dv_{y} \tilde{\mathbf{u}}) + \dv_{y} \tilde{\mathbf{\Gamma}}_{21}, \\
			&\p_\tau (\e \tilde{m}_2) + \dv_{y} \big(\frac{\e \tilde{m}_2 \e \tilde{\mathbf{m}}}{\tilde{\rho}}\big) + \p_{y_2} \tilde{p} = \mu \Delta_y (\e \tilde{u}_2) + (\mu + \lambda) \p_{y_2} (\e \dv_{y} \tilde{\mathbf{u}}) + \dv_{y} \tilde{\mathbf{\Gamma}}_{22}, \\
			&\p_\tau (\e \tilde{m}_3) + \dv_{y} \big(\frac{\e \tilde{m}_3 \e \tilde{\mathbf{m}}}{\tilde{\rho}}\big) + \p_{y_3} \tilde{p} = \mu \Delta_y (\e \tilde{u}_3) + (\mu + \lambda) \p_{y_3} (\e \dv_{y} \tilde{\mathbf{u}}) + \dv_{y} \tilde{\mathbf{\Gamma}}_{23}, \\
			&\p_\tau \tilde{\E} + \dv_{y} \big(\e \tilde{\mathbf{u}} (\tilde{\E} + \tilde{p})\big) = \dv_{y} \big(\e \tilde{\mathbf{u}} \tilde{\mathbb{S}}(\e \tilde{\mathbf{u}}) + \kappa \nabla_y \tilde{\T}\big) + \dv_{y} \tilde{\mathbf{\Gamma}}_3,
		\end{aligned}
		\right.
	\end{equation}
	where
	\begin{align*}
		\tilde{\Gamma}_1~ =& \bar{\Theta}_1 \p_{y_1} \Theta_1 + \bar{\Theta}_3 \p_{y_1} \Theta_3,\\
		\tilde{\mathbf{\Gamma}}_{21} =& \lambda_{1}^- \bar{\Theta}_1 \p_{y_1} \Theta_1 \mathbb{I}_1 + \lambda_{3}^+ \bar{\Theta}_3 \p_{y_1} \Theta_3 \mathbb{I}_1 - (\lambda+\mu) \dv_{y} \big(\frac{\e \tilde{\mathbf{m}}}{\tilde{\rho}} - \frac{\e \bar{\mathbf{m}}}{\bar{\rho}}\big) \mathbb{I}_1 + \bar{Q}_1 \mathbb{I}_1 \\
		&+ \big(\frac{\e \tilde{m}_1 \e \tilde{\mathbf{m}}}{\tilde{\rho}} - \frac{|\e \tilde{\mathbf{m}}|^2}{2\tilde{\rho}} - \frac{\e \bar{m}_1 \e \bar{\mathbf{m}}}{\bar{\rho}} + \frac{|\e \bar{\mathbf{m}}|^2}{2\bar{\rho}}\big)\mathbb{I}_1 - \mu \p_{y_1} \big(\frac{\e \tilde{m}_1}{\tilde{\rho}} - \frac{\e \bar{m}_1}{\bar{\rho}}\big) \mathbb{I}_1, \\
		\tilde{\mathbf{\Gamma}}_{2i} =& \bar{\Theta}_{i+2} \p_{y_1} \Theta_{i+2} \mathbb{I}_1 
		+ \big(\frac{\e \tilde{m}_i \e \tilde{m}_1}{\tilde{\rho}}  - \frac{\e \bar{m}_i \e \bar{m}_1}{\bar{\rho}} \big)\mathbb{I}_1 - \mu \p_{y_1} \big(\frac{\e \tilde{m}_i}{\tilde{\rho}} - \frac{\e \bar{m}_i}{\bar{\rho}}\big) \mathbb{I}_1, \quad i=2,3,\\
		\tilde{\mathbf{\Gamma}}_3~ =& (\lambda_{1}^-)^2 \bar{\Theta}_1 \p_{y_1} \Theta_1 \mathbb{I}_1 + (\lambda_{3}^+)^2 \bar{\Theta}_3 \p_{y_1} \Theta_3 \mathbb{I}_1 - [\e \tilde{\mathbf{u}} \mathbb{S}(\e \tilde{\mathbf{u}}) - \e \bar{\mathbf{u}} \mathbb{S}(\e \bar{\mathbf{u}})] + \bar{Q}_2 \mathbb{I}_1 \\
		&- \big(\frac{\e \tilde{\mathbf{m}}|\e \tilde{\mathbf{m}}|^2 }{2\tilde{\rho}^2} - \frac{\e \bar{\mathbf{m}} |\e \bar{\mathbf{m}}|^2}{2\bar{\rho}^2}\big) - \kappa \nabla_y \big(\frac{\tilde{\E}}{\tilde{\rho}} - \frac{1}{2} \big|\frac{\e \tilde{\mathbf{m}}}{\tilde{\rho}}\big|^2 - \frac{\bar{\E}}{\bar{\rho}} + \frac{1}{2} \big|\frac{\e \bar{\mathbf{m}}}{\bar{\rho}}\big|^2\big) \\
		&+ 2\big(\frac{\e \tilde{m}_1 \tilde{\E}}{\tilde{\rho}} - \frac{\e \bar{m}_1 \bar{\E}}{\bar{\rho}} - \lambda_{1}^- \frac{\E_-}{\rho_-} \bar{\Theta}_1 \Theta_{1} - \lambda_{3}^+ \frac{\E_+}{\rho_+} \bar{\Theta}_3 \Theta_{3} \big) \mathbb{I}_1,
	\end{align*}
	with \(\mathbb{I}_1 = (1,0,0)^t\). It is easy to check that
	\begin{equation}\label{2026-9-20}
		\begin{aligned}
			&|\tilde{\Gamma}_1| \le C \big(\sum_{i=1}^5 \abs{\bar{\Theta}_i} \big)\frac{1}{\e^{-2} + \tau} \big( e^{-\frac{d(y_1-\lambda_1^-(\varepsilon^{-2}+\tau))^2}{\varepsilon^{-2}+\tau}} + e^{-\frac{d(y_1-\lambda_3^+(\varepsilon^{-2}+\tau))^2}{\varepsilon^{-2}+\tau}}\big),   \\
			&|\tilde{\mathbf{\Gamma}}_{2i}| \le C \big(\delta + \sum_{i=1}^5 \abs{\bar{\Theta}_i} \big)\frac{1}{\e^{-2} + \tau} \big(e^{-\frac{d y_1^2}{\varepsilon^{-2}+\tau}}+ e^{-\frac{d(y_1-\lambda_1^-(\varepsilon^{-2}+\tau))^2}{\varepsilon^{-2}+\tau}} + e^{-\frac{d(y_1-\lambda_3^+(\varepsilon^{-2}+\tau))^2}{\varepsilon^{-2}+\tau}}\big),  \quad i=2,3,\\
			&|\tilde{\Gamma}_3| \le C \big(\sum_{i=1}^5 \abs{\bar{\Theta}_i} \big)\frac{1}{\e^{-2} + \tau} \big( e^{-\frac{d(y_1-\lambda_1^-(\varepsilon^{-2}+\tau))^2}{\varepsilon^{-2}+\tau}} + e^{-\frac{d(y_1-\lambda_3^+(\varepsilon^{-2}+\tau))^2}{\varepsilon^{-2}+\tau}}\big)+ \frac{C \big(\delta + \sum_{i=1}^5 \abs{\bar{\Theta}_i} \big)}{(\e^{-2} + \tau)^\frac{3}{2}} e^{-\frac{d y_1^2}{\varepsilon^{-2}+\tau}},
		\end{aligned}
	\end{equation}
	where \(d>0\) is a constant independent of any small parameter in this paper.
	
	For convenience, we set
	\begin{equation}
		\begin{aligned}
			&D_{-\alpha}^1(y_1, \tau) = \frac{1}{(\e^{-2} + \tau)^{\alpha}} e^{-\frac{d y_1^2}{\varepsilon^{-2}+\tau}}, \quad  D_{-\alpha}^2 (y_1, \tau) = \frac{1}{(\e^{-2} + \tau)^{\alpha}} \big( e^{-\frac{d(y_1-\lambda_1^-(\varepsilon^{-2}+\tau))^2}{\varepsilon^{-2}+\tau}} + e^{-\frac{d(y_1-\lambda_3^+(\varepsilon^{-2}+\tau))^2}{\varepsilon^{-2}+\tau}}\big),  \\
			&D_{-\alpha}(y_1, \tau) = D_{-\alpha}^1(y_1, \tau) + D_{-\alpha}^2(y_1, \tau),\quad \tilde D_{-\frac12}^1(y_1, \tau) = \frac{1}{(\e^{-2} + \tau)^{\frac12}} e^{-\frac{\tilde{d}_0 y_1^2}{\varepsilon^{-2}+\tau}} \quad\text{for}\quad \tilde{d}_0 < d_0.
			%       \\
			% &\omega_{-\alpha} (y_1, \tau) = (\e^{-2} + \tau)^{\alpha} e^{-\frac{c y_1^2}{\varepsilon^{-2}+\tau}}, \\
			% &\tilde{\omega}_{-\frac{1}{2}}(y_1, \tau) = (\e^{-2} + \tau)^{-\frac{1}{2}} e^{-\frac{\tilde{c}_0  y_1^2}{\varepsilon^{-2}+\tau}}, \quad \text{for} \quad \tilde{c}_0 < c_0. 
		\end{aligned}
	\end{equation}
	% Moreover, we introduce following new notations:
	% \begin{align*}
		% \hat{N}_2 (U) =\frac{\e \mathbf{m} \otimes \e \mathbf{m}}{\rho} - \frac{|\e \mathbf{m}|^2}{2\rho} \mathbb{I}, \qquad \hat{N}_3 (U) = \frac{\e \mathbf{m} \E}{\rho} + p \frac{\e \mathbf{m}}{\rho} - \frac{\e \mathbf{m}}{\rho} \mathbb{S}(\e \mathbf{u}).
		% \end{align*}
	% Then, (\ref{scaled NS Eqs}) can be rewritten as
	% \begin{equation}\label{scaled NS Eqs2}
		% 	\left\{
		% 	\begin{aligned}
			% 	&\p_\tau \rho +\dv_{y} (\e \mathbf{m}) = 0, \\
			% 	&\p_\tau \E + \dv_{y} \hat{N}_3 (U) = \kappa \Delta_y \big(\frac{\E}{\rho} - \frac{|\e \mathbf{u}|^2}{2}\big).
			% 	\end{aligned}
		% 	\right.
		% \end{equation}
	
	We define the perturbation as
	\begin{align*}
		&(\phi, \psi, \zeta, \varphi, w) (y, \tau) = (\rho - \tilde{\rho}, \e \mathbf{u} - \e \tilde{\mathbf{u}}, \T - \tilde{\T}, \e \mathbf{m} - \e \tilde{\mathbf{m}}, \E - \tilde{\E}), \\
		&(\phi_0, \psi_0, \zeta_0, \varphi_0, w_0) = (\phi, \psi, \zeta, \varphi, w) (y, 0),
	\end{align*}
	and the anti-derivatives of the perturbation for conserved quantities as
	\begin{align*}
		&(\Phi, \Psi, W) = \int_{-\infty}^{y_1} (\mathring{\rho}- \tilde{\rho}, \e \mathring{\mathbf{m}} -\e \tilde{\mathbf{m}}, \mathring{\E}- \tilde{\E})(s,\tau) ds,
	\end{align*}
	with the initial data
	\begin{align*}
		(\Phi_0, \Psi_0, W_0) = (\Phi, \Psi, W) (y_1, 0).
	\end{align*}
	\subsubsection{Main results for well-prepared initial data}    
	Then we state the main results for well-prepared initial data as follows.
	\begin{Thm}[Uniform estimates for well-prepared initial data]\label{Uniform estimates for well-prepared initial data}
		Let  \((\bar{\rho}, \bar{\mathbf{u}}, \bar{\T})\) be the diffusion wave defined in (\ref{bar}) with  wave strength \(\delta = |\rho_+ - \rho_-| + |\T_+ - \T_-|\). There exist positive constants $\eta_0$, $\delta_0$, and
		$\varepsilon_0$ such that, for any fixed
		$\eta\le\eta_0$ and $\delta\le\delta_0$,
		independent of $\varepsilon$, and for every
		$\varepsilon\le\varepsilon_0$, if the initial data satisfy
		% Let \(\Omega = \mathbb{R} \times \mathbb{T}^2\). Assume that \((\rho_-, \mathbf{u}_-, \T_-)\) and \((\rho_+, \mathbf{u}_+, \T_+)\) are the constant states satisfying (\ref{entropy wave condition}), and \((\bar{\rho}, \bar{\mathbf{u}}, \bar{\T})\) is the diffusion wave defined in (\ref{bar}) with the wave strength \(\delta = |\rho_+ - \rho_-| + |\T_+ - \T_-|\).
		% 	    There exist positive constants \(\delta_0\) and \(\e_0\), such that for \(\delta \le \delta_0\), \(\e \le \e_0\), if the initial data satisfies that
		\begin{equation}\label{initial data}
			\left\{
			\begin{aligned}
				&\|(\Phi_0^\varepsilon, \Psi_0^\varepsilon, W_0^\varepsilon) \|_{H^3(\mathbb{R})}^2 + \|(\phi_0^\varepsilon, \psi_0^\varepsilon, \zeta_0^\varepsilon)\|_{H^3(\mathbb{R} \times \mathbb{T}^2)}^2 + \|(\rho^{\varepsilon}-\bar\rho,\varepsilon(\m^{\varepsilon}-\bar\m),\E^{\varepsilon}-\bar\E)|_{t=0}\|_{L^1(\mathbb{R} \times \mathbb{T}^2)}^2 \le \eta \varepsilon^3, \\
				&\|(\phi_{\neq 0}^\varepsilon, \psi_{\neq 0}^\varepsilon, \zeta_{\neq 0}^\varepsilon) \|_{H^1(\mathbb{R} \times \mathbb{T}^2)}^2 \le Ce^{-c_0\varepsilon^{-2}},
			\end{aligned}
			\right.
		\end{equation}
		for some sufficiently large constant $c_0$ independent of $\e$, then the Cauchy problem (\ref{nondimensional NS Eqs}) admits a unique global smooth solution \((\rho^\e, \mathbf u^\e,\T^\e)\) satisfying
		\begin{equation}\label{decay rate}
			\left\{
			\begin{aligned}
				&\|(\rho^\varepsilon - \bar{\rho}, \varepsilon \mathbf{u}^\varepsilon - \varepsilon \bar{\mathbf{u}}, \T^\varepsilon - \bar{\T}) \|_{L_x^2(\R\times\Torus^2)}^2 \le C \varepsilon^3 (\sqrt{\eta}+\delta) (1+t)^{-\frac{1}{2}}, \\
				&\|\nabla_x(\rho^\varepsilon - \bar{\rho}, \varepsilon \mathbf{u}^\varepsilon - \varepsilon \bar{\mathbf{u}}, \T^\varepsilon - \bar{\T}) \|_{L_x^2(\R\times\Torus^2)}^2 \le C \varepsilon^3 (\sqrt{\eta}+\delta) (1+t)^{-\frac{3}{2}}, \\
				&|\nabla_x^2(\rho^\varepsilon - \bar{\rho}, \varepsilon \mathbf{u}^\varepsilon - \varepsilon \bar{\mathbf{u}}, \T^\varepsilon - \bar{\T}) \|_{L_x^2(\R\times\Torus^2)}^2 \le C \varepsilon (\sqrt{\eta}+\delta) (1+t)^{-\frac{3}{2}},
			\end{aligned}
			\right.
		\end{equation}
		where \(C\) is  a positive constant independent of \(\e\). Moreover, the non-zero modes of the  solution satisfy
		% \(\delta\) and \(\e\), \(\eta > 0\) is a sufficiently small constant and \(\bar\delta = \e + \delta\).
		% Moreover, the non-zero modes of the smooth solution have a decay rate in time:
		\begin{equation}\label{non zero decay}
			\left\{
			\begin{aligned}
				& \norm{(\rho_{\neq}^{\varepsilon}, \e \mathbf{u}_{\neq}^{\varepsilon}, \T^{\varepsilon}_{\neq})}_{H_x^1(\mathbb{R} \times \mathbb{T}^2)}^2 
				\le C \e^{-2}e^{-\frac{c}{2}  t - c_0 \e^{-2}}, \qquad \qquad\qquad \qquad\qquad \quad~\text{if}~~ t \ge \left(\frac{2C_1 }{c \e}\right)^2 - 1, \\
				&\norm{(\rho^{\varepsilon}_{\neq}, \e \mathbf{u}^{\varepsilon}_{\neq}, \T^{\varepsilon}_{\neq})}_{H_x^1(\mathbb{R} \times \mathbb{T}^2)}^2 \le C \e^{-2N-2} (1 + t)^{-\frac{N}{2}} e^{-c_0 \e^{-2}},~ \forall N \in \mathbb{Z}^+, \qquad \text{if}~~ t < \left(\frac{2C_1 }{c \e}\right)^2 - 1,
			\end{aligned}
			\right.
		\end{equation}
		where \(C_1\)  and \(c\) are positive constants independent of  \(\e\).
	\end{Thm}
	Based on Theorem \ref{Uniform estimates for well-prepared initial data} and Sobolev embedding, we justify the following low Mach number limit.
	\begin{Cor}[Low Mach number limit for well-prepared initial data]
		Under the assumptions of Theorem \ref{Uniform estimates for well-prepared initial data}, as \(\e\) tends to 0, it holds that
		\begin{equation}
			\left\{
			\begin{aligned}
				&\|(\rho^\varepsilon - \bar{\rho}, \T^\varepsilon - \bar{\T})\|_{L^\infty_x(\R\times\Torus^2)} \le C  \varepsilon^{\frac{3}{2}} (\sqrt{\eta}+\delta)^{\frac{1}{2}} (1+t)^{-\frac{1}{2}} \to 0, \\
				&\|\mathbf{u}^\varepsilon- \bar{\mathbf{u}}\|_{L^\infty_x(\R\times\Torus^2)} \le C  \varepsilon^{\frac{1}{2}} (\sqrt{\eta}+\delta)^{\frac{1}{2}} (1+t)^{-\frac{1}{2}} \to 0.
			\end{aligned}
			\right.
		\end{equation}
	\end{Cor}
	\begin{proof}
		We take  \(\rho^\e - \bar{\rho}\) as an example. Due to \eqref{decay rate} and \eqref{non zero decay}, one has
		\begin{align*}
			&\norm{\rho^\e - \bar{\rho}}_{L^\infty(\R \times \Torus^2)} \le\norm{\mathring{\rho}^\e - \bar{\rho}}_{L^\infty(\R)} + \norm{\rho_{\neq}^\e}_{L^\infty(\R \times \Torus^2)} \\
			\le& \norm{\mathring{\rho}^\e - \bar{\rho}}_{L^2(\R)}^{\frac{1}{2}} \norm{\p_{x_1} (\mathring{\rho}^\e - \bar{\rho})}_{L^2(\R)}^{\frac{1}{2}} + \norm{\rho_{\neq}^\e}_{L^2(\R \times \Torus^2)}^{\frac{1}{4}} \norm{\nabla^2_x \rho_{\neq}^\e}_{L^2(\R \times \Torus^2)}^\frac{3}{4} \\
			\le& C  \varepsilon^{\frac{3}{2}} (\sqrt{\eta}+\delta)^{\frac{1}{2}} (1+t)^{-\frac{1}{2}}.
		\end{align*}
		The other terms can be treated in a same way.
	\end{proof}
	\begin{Rem}
		For the 1D case \cite{HWW}, the initial perturbation is assumed to zero, whereas the initial data considered in our work have non-zero mass.
		% This gives rise to additional difficulties in the energy estimates.
		Moreover, we obtain an improved convergence rate with respect to \(\e\) compared with \cite{HWW}.
	\end{Rem}
	\begin{Rem}
		The assumption that the initial data in \((\ref{initial data})_1\) are \(O(\e^3)\) 
		leads to a faster convergence rate with respect to \(\e\) for both \((\r^\e,\mathbf{u}^\varepsilon,\T^\e)\). In fact, under a weaker smallness assumption on the initial data, convergence of the \((\r^\e,\mathbf{u}^\varepsilon,\T^\e)\) with a reduced rate is still expected to hold.
	\end{Rem}
	\begin{Rem}
		The condition \eqref{initial data}\(_2\) allows us to establish a global-in-time result for the low Mach number limit in higher dimensions, even for large temperature variations. More precisely, this condition ensures the uniform-in-time smallness of the non-zero modes.
	\end{Rem}
	
	\subsubsection{Challenges and strategy.}
	\begin{itemize}
		\item The limitation of structural conditions.
		
		We focus on the equations for \((\rho,\varepsilon m_1,\mathcal E)\) in the inviscid version of system \eqref{scaled NS Eqs}, since the background state involves only the components \((\bar\rho,\varepsilon\bar m_1,\bar\theta)\)
		\begin{equation}\label{planar entropy waves of scaled NS Eqs}
			\left\{
			\begin{aligned}
				&\partial_{\tau} \rho + \partial_{y_1} (\e m_1) = 0, \\
				&\partial_{\tau} (\e m_1) + \partial_{y_1} [\E + \frac{1}{2} \frac{(\e m_1)^2}{\rho}] = 0, \\
				&\partial_{\tau} \E + \partial_{y_1} \big(\frac{2\e m_1 \E}{\rho} - \frac{(\e m_1)^3}{2\rho^2}\big) = 0.
			\end{aligned}
			\right.
		\end{equation}
		The Jacobi matrix of the flux for (\ref{planar entropy waves of scaled NS Eqs}) is
		\[A(\rho, \e m_1, \mathcal{E}) = 
		\begin{pmatrix}
			0 & 1 & 0 \\
			-\frac{|\e m_1|^2}{2\rho^2} & \frac{\e m_1}{\rho} & 1 \\
			-\frac{2 \e m_1 \mathcal{E}}{\rho^2} + \frac{(\e m_1)^3}{\rho^3} & \frac{2 \mathcal{E}}{\rho} - \frac{3|\e m_1|^2}{2\rho^2} & \frac{2 \e m_1}{\rho}
		\end{pmatrix}.\]
		A direct calculation gives that the second eigenvalue of \(A(\rho, m_1, \mathcal{E})\) is \(\lambda_2 = \frac{\e m_1}{\rho}\), and the corresponding left and right eigenvectors are
		\begin{align*}
			r_2 = \big(1, \frac{\e m_1}{\rho}, \frac{|\e m_1|^2}{2\rho^2}\big)^t, \quad\quad l_2 = \big(\frac{2\E}{\rho} - \frac{3|\e m_1|^2}{2\rho^2}, \frac{\e m_1}{\rho}, -1\big).
		\end{align*}
		Then we have
		\begin{align}\notag
			\nabla l_2 \cdot r_2 \neq 0, \quad \nabla r_2 \cdot r_2 = 0,
		\end{align}
		which means that the classical left and right structural conditions as in \cite{Liu1997} do not work for the Eulerian Navier-Stokes equations in the 3D case. 
		Thus, the emergence of slowly decaying lower-order terms, such as
		\begin{align}\label{111}
			\int_{\Omega}(1+t)^{-\frac12}e^{-\frac{cx^2}{1+t}}
			\Big((\rho-\bar \rho)^2+\e (m_1-\bar m_1)^2+(\theta-\bar \theta)^2\Big)dx,
		\end{align}
		prevents the method developed in \cite{HWW} from being directly applicable.
		% Thus, the method developed in \cite{HWW} from being directly applicable, due to the emergence of slowly decaying lower-order terms such as 
		% \begin{align*}
			%   \int_{\Omega}(1+t)^{-\frac12}e^{-\frac{cx^2}{1+t}}\Big((\rho-\bar \rho)^2+\e (m_1-\bar m_1)^2+(\theta-\bar \theta)^2\Big)dx.
			% \end{align*}
		To overcome this difficulty, motivated by \cite{DHLX-1,DHLX-2}, we introduce a family of transformations \eqref{tilde Phi, Psi, W} for perturbation of $(\bar\rho,\e\bar m,\bar\theta)$, which allows us to improve the decay rate \((1+t)^{-1/2}\) in \eqref{111} to \((1+t)^{-1}\). As a result, the lower-order terms take a form similar to those in \cite{HWW}.

		\item Uniform convergence with respect to \(\e\).
		
		To obtain uniform estimates in the Mach number, we need to control the error terms $(\tilde{\Gamma}_1,\tilde{\mathbf{\Gamma}}_{2i},\tilde{\mathbf{\Gamma}}_3)$ \eqref{2026-9-20} in the anti-derivative system \eqref{the perturbation equs} for the perturbation variables \((\Phi,\Psi_i,W)\) \eqref{anti}.  If these error terms are estimated as in \cite{DHLX-1}, due to $\norm{(\tilde{\Gamma}_1^{\frac k 2},\tilde{\mathbf{\Gamma}}_{2i}^{\frac k 2},\tilde{\mathbf{\Gamma}}_3^{\frac k 2})}_{L^2}\lesssim (\delta+\e^{\frac12})^{\frac k2}(\e^{-2}+\tau)^{-\frac k2+\frac14}$ for $k\in\mathbb{N}^{+}$, the resulting estimate is
		\begin{align*}
			&\norm{\p_{y_1}^{k+1} (\Phi, \Psi, W)}_{L^2(\R)}^2 \lesssim  \delta\e^{2k+1}  \left(1 + t\right)^{-\frac{2k+1}{2}}, \quad k=0,1.
		\end{align*}
		By (\ref{relation1}) and (\ref{relation2}), we have
		\begin{align*}
			\|\mathbf{u}^\varepsilon- \bar{\mathbf{u}}\|_{L^\infty_x(\R\times\Torus^2)} \le C\delta (1+t)^{-\frac{1}{2}},
		\end{align*}
		which is too weak for our purpose in the low Mach number limit.  To overcome this difficulty, we decompose the error terms into two parts: the contributions arising from the diffusion waves,  which carry the extra initial mass, and those associated with the background solution. We then estimate these two parts separately.
		The diffusion wave contribution can be treated  in the same way as \cite{DHLX-1}, since it contains an additional factor $\e^{\frac12}$. 
		For the background solution contribution, however, we exploit the special dissipation structure \(G_k\) in \eqref{2026-9-20-2}, which allows us to estimate \((\tilde{\Gamma}_1^{\frac{3k}{4}},
		\tilde{\mathbf{\Gamma}}_{2i}^{\frac{3k}{4}},
		\tilde{\mathbf{\Gamma}}_3^{\frac{3k}{4}})\) instead of \((\tilde{\Gamma}_1^{\frac{k}{2}},
		\tilde{\mathbf{\Gamma}}_{2i}^{\frac{k}{2}},
		\tilde{\mathbf{\Gamma}}_3^{\frac{k}{2}});\) see the proof of Lemma \ref{main lemma} for details.

		\item The dependence of the Mach number on the Poincar\'e inequality.
		
		Under the scaling \((\ref{transformation})\), the domain is rescaled to
		$
		\Omega_\varepsilon=\mathbb{R}\times(\mathbb{T}/\varepsilon)^2,
		$
		where the transverse periods are of order \(\varepsilon^{-1}\). Accordingly, the Poincar\'e inequality for the non-zero modes takes the form
		\[
		\|f_{\neq}\|_{L^2(\Omega_\varepsilon)}
		\le C\varepsilon^{-k}\|\nabla_y^k f_{\neq}\|_{L^2(\Omega_\varepsilon)}.
		\]
		The resulting \(\varepsilon^{-k}\)-loss creates two main difficulties, preventing a straightforward adaptation of the non-zero mode estimates in \cite{DHLX-1}.
		The first difficulty lies in controlling the coupling terms between the background solution and the perturbation, such as  \(\frac{\mathring\theta}{\mathring\rho\tilde\rho}\phi_{\neq}\psi_{\neq} \cdot \nabla \tilde{\rho}\) arising from $\frac{\mathring{\T}}{\mathring{\rho} \tilde{\rho}} \phi_{\neq} \mathcal{A}_{0\neq}$ and appearing in \(\hat{J}_0\) \eqref{non zero est1}.   To overcome the difficulty, we divide the time interval into two parts and treat them separately. When \(\tau \ge \left(\frac{2C_1}{c \e^2}\right)^2 - \e^{-2}\), $\partial_{y_1} \tilde\rho$ becomes sufficiently small to compensate for the \(\varepsilon\)-dependent loss arising from the Poincar\'e inequality. Consequently, terms of the form \(\frac{\mathring\theta}{\mathring\rho\tilde\rho}\phi_{\neq}\psi_{\neq} \cdot \nabla \tilde{\rho}\) can be absorbed into the viscous dissipation. On the other hand, for \(\tau \le \left(\frac{2C_1}{c \e^2}\right)^2 - \e^{-2}\),  such terms can be effectively controlled by the Gr\"{o}nwall inequality, together with an additional smallness assumption on the initial data for non-zero modes.
		The second difficulty arises from the coupling between the time derivatives of the zero modes and the non-zero modes of the perturbation, such as
		$
		\left(\frac{\mathring{\T}}{\mathring{\rho}\tilde{\rho}}\right)_{\tau}
		\frac{\phi{\neq}^{2}}{2}$
		in \(\hat{J}_0\) \eqref{non zero est1}. Since
		\(
		\partial_\tau \mathring{\theta}
		\approx \partial_{y_1}^2\mathring{\theta}+\cdots
		\approx \partial_{y_1}^2\mathring{\zeta}+\cdots,
		\)
		the control of such terms requires additional time-decay estimates for higher-order  derivatives of the perturbation, including \(\partial_{y_1}^2\mathring{\zeta}\). This additional decay estimate is not required in the argument of \cite{DHLX-1}.

	\end{itemize}
	
	\subsection{Moderately ill-prepared data}
	
	To understand the role of the thermodynamics, we introduce
	\begin{align}\label{transformation1 for ill}
		p^\e(x,t) = e^{\e \underline{p}^\e (x,t)}, \qquad \T^\e (x,t) = e^{\underline{\T}^\e (x,t)}.
	\end{align}
	Then we have
	\begin{align}\label{transformation2 for ill}
		\rho^\e (x,t) = e^{\e \underline{p}^\e (x,t) - \underline{\T}^\e (x,t)}.
	\end{align}
	Using these variables, we can see that  
	\begin{align}
		\underline{p}^{\e} \to 0, \quad \mathbf{u}^{\varepsilon} \to 0 \quad \text{and} \quad \underline{\T}^{\e} \to \underline{\T}_{\pm} \quad \text{as} \quad x_1 \to \pm \infty,
	\end{align}
	where $\underline{\T}_{\pm} = \ln \T_{\pm}$.
	
	Under these changes of variables and coefficients, the non-isentropic compressible Navier-Stokes system (\ref{nondimensional NS Eqs}) and the limiting system take the following form, respectively,
	\begin{equation} \label{NS system in scaling}
		\left\{
		\begin{aligned}
			&\partial_t \underline{p}^\varepsilon
			+(\mathbf{u}^\varepsilon\cdot\nabla)\underline{p}^\varepsilon
			+\frac{1}{\varepsilon}
			\operatorname{div}\left(
			2\mathbf{u}^\varepsilon
			-\kappa e^{-\varepsilon\underline{p}^\varepsilon
				+\underline{\theta}^\varepsilon}
			\nabla\underline{\theta}^\varepsilon
			\right) 
			=\varepsilon e^{-\varepsilon\underline{p}^\varepsilon}
			\left[
			\mathbb{S}(\mathbf{u}^\varepsilon):\nabla\mathbf{u}^\varepsilon
			\right]
			+\kappa e^{-\varepsilon\underline{p}^\varepsilon
				+\underline{\theta}^\varepsilon}
			\nabla\underline{p}^\varepsilon
			\cdot\nabla\underline{\theta}^\varepsilon,
			\\
			&e^{-\underline{\theta}^\varepsilon}
			\left[
			\partial_t\mathbf{u}^\varepsilon
			+(\mathbf{u}^\varepsilon\cdot\nabla)\mathbf{u}^\varepsilon
			\right]
			+\frac{\nabla\underline{p}^\varepsilon}{\varepsilon}
			=e^{-\varepsilon\underline{p}^\varepsilon}
			\operatorname{div}\mathbb{S}(\mathbf{u}^\varepsilon),
			\\
			&\partial_t\underline{\theta}^\varepsilon
			+(\mathbf{u}^\varepsilon\cdot\nabla)\underline{\theta}^\varepsilon
			+\operatorname{div}\mathbf{u}^\varepsilon
			=\varepsilon^2e^{-\varepsilon\underline{p}^\varepsilon}
			\left[
			\mathbb{S}(\mathbf{u}^\varepsilon):\nabla\mathbf{u}^\varepsilon
			\right]
			+\kappa e^{-\varepsilon\underline{p}^\varepsilon}
			\operatorname{div}\left(
			e^{\underline{\theta}^\varepsilon}
			\nabla\underline{\theta}^\varepsilon
			\right).
		\end{aligned}
		\right.
	\end{equation}
	and
	\begin{equation} \label{limit system in scaling}
		\left\{
		\begin{aligned}
			&\left(
			2u_1-\kappa e^{\underline{\theta}}
			\underline{\theta}_{x_1}
			\right)_{x_1}=0,
			\\
			&e^{-\underline{\theta}}
			\left(
			\partial_tu_1+u_1\partial_{x_1}u_1
			\right)
			+\pi_{x_1}
			=(2\mu+\lambda)\partial_{x_1x_1}u_1,
			\\
			&e^{-\underline{\theta}}
			\left(
			\partial_tu_2+u_1\partial_{x_1}u_2
			\right)
			=\mu \partial_{x_1x_1}u_2, \\
			&e^{-\underline{\theta}}
			\left(
			\partial_tu_3 +u_1\partial_{x_1}u_3
			\right)
			=\mu\partial_{x_1x_1}u_3, \\
			&\partial_t\underline{\theta}
			+u_1\partial_{x_1}\underline{\theta}
			+\partial_{x_1}u_1
			=\kappa\partial_{x_1}
			\left(
			e^{\underline{\theta}}
			\underline{\theta}_{x_1}
			\right).
		\end{aligned}
		\right.
	\end{equation}
	We notice that $\underline{\theta}_+$ may not be equal to $\underline{\theta}_-$, so it is necessary to introduce a background profile $\widetilde{\theta}$ for $\theta^{\varepsilon}$.
	Define
	\begin{align}
		\underline{\Theta} = -\ln \hat{\rho},
	\end{align}
	satisfying \(\underline{\Theta} \to \underline{\theta}_{\pm}\) as \( x_1 \to \pm \infty.\)
	To establish the uniform estimate, we shall use the following energy functional:
	\begin{align*}
		\norm{(\underline{p}^\e, \mathbf{u}^\e, \underline{\T}^\e- \underline{\Theta})}_{H^{s,\e}}^2 :=& \sum_{|\alpha| = 0}^{s} \norm{\p^\alpha (\underline{p}^\e, \mathbf{u}^\e) (t)}_{L^2(\Omega)}^2 + \sum_{|\alpha| = 0}^{s+1} \norm{\p^\alpha (\e \underline{p}^\e, \e \mathbf{u}^\e) (t)}_{L^2}^2 + \norm{(\underline{\T}^\e - \underline{\Theta}) (t)}_{L^2}^2 \\
		&+ \sum_{|\alpha| = 1}^{s+1} \norm{\p^\alpha \underline{\T}^\e (t)}_{L^2}^2 + \sum_{|\alpha| = 0}^{s} \int_{0}^{t} \norm{\p^\alpha \nabla (\underline{p}^\e, \mathbf{u}^\e) (\tau)}_{L^2}^2 d \tau \\
		&+ \sum_{|\alpha| = 0}^{s+1} \int_{0}^{t} \norm{\p^\alpha \nabla (\e \mathbf{u}^\e, \underline{\T}^\e) (\tau)}_{L^2}^2 d \tau,
	\end{align*}
	where \(\p^{\alpha} := (\e \p_t)^{\alpha_0} \p_{x_1}^{\alpha_1}  \p_{x_2}^{\alpha_2}  \p_{x_3}^{\alpha_3}\) with multi-index \(\alpha = (\alpha_0, \alpha_1, \alpha_2, \alpha_3)\).
	
	\subsubsection{Main results for ill-prepared initial data}
	\begin{Thm}[\cite{JuMengMMAS}](Uniform estimates for ill-prepared initial data) \label{Uniform estimates Thm for ill-prepared initial data}
		Given an integer \( s \geq 4 \) and a family of initial data \((\underline{p}_0^\e, \mathbf{u}_0^\e, \underline{\T}_0^\e)\) satisfying
		\begin{align}
			\sup_{\e\in (0, 1]} \| (\underline{p}_0^\e, \mathbf{u}_0^\e, \underline{\T}_0^\e - \underline{\Theta}) \|_{H^{s,\e}}^2 \leq \hat{C}_0 < \infty,
		\end{align}
		where \(\hat{C}_0\) is a positive constant independent of \(\varepsilon\). Then there exist positive constants \(T_0\) and \(\varepsilon_0\) depending only on \(\hat{C}_0\) and \(|\underline{\theta}_+ - \underline{\theta}_-|\) such that, for all \(t \in [0, T_{0}]\) and \(\varepsilon \in (0, \varepsilon_{0}]\), the Cauchy problem (\ref{NS system in scaling}) with the initial data \((\underline{p}_0^\e, \mathbf{u}_0^\e, \underline{\T}_0^\e)\) has a unique smooth solution \((\underline{p}^\e, \mathbf{u}^\e, \underline{\T}^\e)\) satisfying
		\begin{align}\label{uniform estimates for ill data}
			\|(\underline{p}^{\e}, \mathbf{u}^{\e}, \underline{\T}^{\e} - \underline{\Theta})(t)\|_{H^{s, \e}}^{2} \leq \tilde{C}_{0},
		\end{align}
		where \(\tilde{C}_{0}\) is a positive constant depending only on \(\hat{C}_0\) and \(|\underline{\T}_+ - \underline{\T}_-|\).
	\end{Thm}
	We remark that in \cite{JuMengMMAS}, we
	% were only able to 
	established the uniform energy estimates  \eqref{uniform estimates for ill data}, while the convergence as \(\varepsilon\to0\) 
	% on \(\e\) 
	remained unresolved due to persistent acoustic oscillations in \(\R \times \Torus^2\). In this paper, we address this convergence problem by decomposing the perturbation into  zero and non-zero modes.
	
	%To obtain the convergence result, we also need to decompose the solution into the principal and transversal parts corresponding to the zero and non-zero modes in Fourier space, respectively. Here, we have
	%\begin{align}
	%	\mathbf{D}_0 f := \mathring{f} = \int_{\mathbb{T}^2} f dx_2 dx_3, \qquad \mathbf{D}_{\neq} f := f_{\neq} = f - \mathring{f},
	%\end{align}
	%for any integrable function f on \(\mathbb{T}^2\). 
	Using the \(L^2\)- orthogonality between the zero and non-zero modes, we obtain
	\begin{align}\label{zero mode L^2}
		\norm{\mathbf{D}_0(\underline{p}^\e, \mathbf{u}^\e, \underline{\T}^\e - \underline{\Theta})}_{H^{s,\e}(\mathbb{R})}^2 \le 
		\norm{(\underline{p}^\e, \mathbf{u}^\e, \underline{\T}^\e - \underline{\Theta})}_{H^{s,\e}(\Omega)}^2 \le \tilde{C}_0.
	\end{align}
	On the other hand, applying \(\mathbf{D}_{\neq}\) to (\ref{nondimensional NS Eqs}), we obtain the non-zero mode system as follows:
	\begin{equation}\label{non-zero system}
		\left\{
		\begin{aligned}
			&\p_t \rho_{\neq}^\e + \mathring{\mathbf{u}}^\e \cdot \nabla \rho_{\neq}^\e + \mathring{\rho}^\e \dv \mathbf{u}_{\neq}^\e = f_{1\neq}^\e, \\
			&\mathring{\rho}^\e \p_t \mathbf{u}_{\neq}^\e + \mathbf{D}_0 (\rho^\e \mathbf{u}^\e) \cdot \nabla \mathbf{u}_{\neq}^\e+\frac{1}{\e^2}\left(\mathbf{D}_{\neq} (\rho^\e \nabla \T^\e) + \mathbf{D}_{\neq} (\T^\e \nabla \rho^\e)\right) = \mu \Delta \mathbf{u}_{\neq}^\e + (\lambda + \mu) \nabla \dv \mathbf{u}_{\neq}^\e + \mathbf{f}_{2\neq}^\e, \\
			&\mathring{\rho}^\e \p_t \T_{\neq}^\e + \mathbf{D}_0 (\rho^\e \mathbf{u}^\e) \cdot \nabla \T_{\neq}^\e + \mathbf{D}_0 (\rho^\e \T^\e) \dv \mathbf{u}_{\neq}^\e = \kappa \Delta \T_{\neq}^\e + \e^2 \nabla \mathring{\mathbf{u}}^\e : \mathbb{S}(\mathbf{u}_{\neq}^\e) + f_{3\neq}^\e,
		\end{aligned}
		\right.
	\end{equation}
	where
	\begin{align*}
		&f_{1\neq}^\e := \mathring{\mathbf{u}}^\e \cdot \nabla \rho_{\neq}^\e - \mathbf{D}_{\neq} (\mathbf{u}^\e \cdot \nabla \rho^\e) + \mathring{\rho}^\e \dv \mathbf{u}_{\neq}^\e - \mathbf{D}_{\neq} (\rho^\e \dv \mathbf{u}^\e), \\
		&\mathbf{f}_{2\neq}^\e :=\mathring{\rho}^\e \p_t \mathbf{u}_{\neq}^\e - \mathbf{D}_{\neq} (\rho^\e \p_t \mathbf{u}^\e) + \mathbf{D}_0 (\rho^\e \mathbf{u}^\e) \cdot \nabla \mathbf{u}_{\neq}^\e - \mathbf{D}_{\neq} (\rho^\e \mathbf{u}^\e \cdot \nabla \mathbf{u}^\e), \\
		&f_{3\neq}^\e := \mathring{\rho} \p_t \T_{\neq}^\e - \mathbf{D}_{\neq} (\rho^\e \p_t \T^\e + \mathbf{D}_0 (\rho^\e \mathbf{u}^\e) \cdot \nabla \T_{\neq}^\e - \mathbf{D}_{\neq} (\rho^\e \mathbf{u}^\e \cdot \nabla \T^\e) + \mathbf{D}_0 (\rho^\e \T^\e) \dv \mathbf{u}^\e_{\neq} - \mathbf{D}_{\neq} (\rho^\e \T^\e \dv \mathbf{u}^\e) \\
		&\qquad ~~-\e^2 \nabla \mathring{\mathbf{u}}^\e : \mathbb{S}(\mathbf{u}_{\neq}^\e) + \e^2 \mathbf{D}_{\neq} [\nabla \mathbf{u}^\e : \mathbb{S} (\mathbf{u}^\e)].
	\end{align*}
	The uniform estimates of non-zero modes can be obtained as follows:
	\begin{Thm}(Uniform estimates for non-zero modes)\label{Uniform estimates for non-zero modes Thm}
		Under the assumptions of Theorem \ref{Uniform estimates Thm for ill-prepared initial data} and further assume that the initial data satisfy 
		\begin{align}
			\norm{(\rho^{\varepsilon}_{\neq}, \mathbf{u}^{\varepsilon}_{\neq}, \T^{\varepsilon}_{\neq})(0)}_{H^s{(\mathbb{R} \times \mathbb{T}^2)}} \le e^{-c_0\varepsilon^{-4}}, \label{non-zero initial data for ill}
		\end{align}
		where \(c_0\) is a positive constant independent of \(\e\). Then there exist positive constants \(T_1 \in (0, T_0]\) and \(\e_1 \in (0, \e_0]\) such that for all \(t \in [0, T_1]\) and \(\e \in (0, \e_1]\), the Cauchy problem (\ref{non-zero system}) with the initial data \((\rho^{\varepsilon}_{\neq}, \mathbf{u}^{\varepsilon}_{\neq}, \T^{\varepsilon}_{\neq})(0)\) has a unique smooth solution \((\rho^{\varepsilon}_{\neq}, \mathbf{u}^{\varepsilon}_{\neq}, \T^{\varepsilon}_{\neq})\) satisfying
		\begin{align}\label{non-zero mode est for ill}
			\sup_{0 \le t \le T_1}\norm{(\rho_{\neq}^\e, \mathbf{u}_{\neq}^\e, \T_{\neq}^\e)(t)}_{H^s}^2
			\le e^{-C \e^{-4}},
		\end{align}
		where \(C\) is a positive constant.
	\end{Thm}
	Then we have the following low Mach number limit.
	\begin{Thm}(Low Mach number limit for Moderately ill-prepared initial data) \label{Convergence Thm for ill-prepared initial data}
		Under the assumptions of Theorem \ref{Uniform estimates for non-zero modes Thm} and further assume that the initial data satisfy 
		\begin{align}
			&(\mathring{\underline{p}}_0^\e, \mathring{\mathbf{u}}_0^\e, \mathring{\underline{\T}}_0^\e - \underline{\Theta}) 
			\to 
			(\mathring{\underline{p}}_0, \mathring{\mathbf{u}}_0, \mathring{\underline{\T}}_0 - \underline{\Theta}), 
			\quad \text{in}~ H^s(\mathbb{R}) \quad \text{as}~ \e \to 0, \notag\\
			&|\mathring{\underline{\T}}_0 - \underline{\T}_+| \le C x_1^{-1-\sigma}, 
			\quad \text{for}~ x_1 \in [1, +\infty), 
		\end{align}
		where \(\sigma\) and \(C\) are positive constants. Then the solution \((\underline{p}^\e, \mathbf{u}^\e, \underline{\T}^\e)\) of (\ref{NS system in scaling}) converges strongly in \(L^2(0, T_1; H^{s'}_{loc}(\mathbb{R}\times \Torus^2))\) for all \(s' < s\) to \((0, \bar{u}_1, \bar{u}_2, \bar{u}_3, \bar{\underline{\T}})\), where \((\bar{u}_1, \bar{u}_2, \bar{u}_3, \bar{\underline{\T}})\) is the unique solution of (\ref{limit system in scaling}) with the initial data \((\underline{\mathring{w}}_0, \mathring{u}_{20}, \mathring{u}_{30}, \underline{\mathring{\T}}_0)\), where \(\underline{\mathring{w}}_0\) is determined by 
		\begin{align*}
			\underline{\mathring{w}}_0 = \frac{1}{2} \kappa e^{\underline{\mathring{\T}}_0} \p_{x_1} \underline{\mathring{\T}}_0.
		\end{align*}
	\end{Thm}
	
	\subsubsection{Challenges and strategy.}
	
	Compared to \(\R^3\), the periodic transverse geometry weakens acoustic dispersion in \(\mathbb{R} \times \mathbb{T}^2\). The convergence theorem established in \cite{Métivier2001} cannot be applied directly in this case. To overcome this difficulty, we decompose the system (\ref{nondimensional NS Eqs}) into two parts: the zero mode system and the non-zero mode system, and derive energy estimates for each of the parts separately. On the one hand, the zero mode is controlled by the estimate \eqref{zero mode L^2}. On the other hand, the rapid convergence of the non-zero modes follows from the Poincar\'e inequality together with the smallness assumption on their initial data \eqref{non-zero initial data for ill}. Thus, we reduce the convergence analysis to estimates of one-dimensional type for the remaining acoustic modes along \(\R\). This method enable us to
	establish the low Mach number limit despite the weaker dispersion with Moderately ill-prepared initial data.
	% These estimates enable us to control the fast acoustic oscillations and establish the low Mach number limit despite the weaker dispersion.

	% On the other hand,
	% the Poincar\'e inequality,  together with the smallness assumption on their initial data \eqref{non-zero initial data for ill}, yields the required decay estimates.

	% the rapid convergence of the non-zero modes follows from the Poincar\'e inequality together with the smallness assumption on their initial data \eqref{non-zero initial data for ill}. Thus, we reduce the convergence analysis to estimates of one-dimensional type for the remaining acoustic modes along \(\R\). These estimates enable us to control the fast acoustic oscillations and establish the low Mach number limit despite the weaker dispersion.
	
	%We impose an additional initial condition (\ref{non-zero initial data for ill}) so that the non-zero modes tend to \(0\) as \(\e \to 0\). Under this condition, we reduce the convergence analysis to estimates of one-dimensional type for the remaining acoustic modes along \(\R\). These estimates enable us to control the fast acoustic oscillations and establish the low Mach number limit despite the weaker dispersion.

	\section{Uniform estimates for well-prepared initial data} \label{section3}
	
	In this section, we are devoted to the Mach number limit for well-prepared initial data. For simplicity, we omit the subscript \(y\) in \(\dv_{y}\), \(\nabla_y\) and \(\Delta_y\).
	
	\subsection{Perturbation equations}
	Recall the definition of the perturbations,
	\begin{align*}
		(\phi, \varphi, w, \psi, \zeta) = (\rho-\tilde{\rho}, \e \mathbf{m} - \e \tilde{\mathbf{m}}, \E-\tilde{\E}, \e \mathbf{u} - \e \tilde{\mathbf{u}}, \T-\tilde{\T}).
	\end{align*}
	Combining (\ref{scaled NS Eqs}) and (\ref{tilde U equs}), we can write the equations of \((\phi, \psi, \zeta)\) as follows,
	\begin{equation}\label{the perturbation equs}
		\left\{
		\begin{aligned}
			&\p_\tau \phi + \e \mathbf{u} \cdot \nabla \phi + \rho \dv \psi = \mathcal{R}_1, \\
			&\rho \p_\tau \psi + \e \rho \mathbf{u} \cdot \nabla \psi + (\T \nabla \phi + \rho \nabla \zeta) = \mu \Delta \psi + (\mu + \lambda) \nabla \dv \psi + \mathcal{R}_2, \\
			&\rho \p_\tau \zeta + \e \rho \mathbf{u} \cdot \nabla \zeta + \rho \T \dv \psi = \kappa \Delta \zeta + \mathcal{R}_3,
		\end{aligned}
		\right.
	\end{equation}
	where
	\begin{align*}
		\mathcal{R}_1 =& - \p_{y_1} \tilde{\Gamma}_1 - \psi \cdot \nabla \tilde{\rho} - \dv (\e \tilde{\mathbf{u}}) \phi, \\
		\mathcal{R}_2 =& - \rho \psi \cdot \nabla (\e \tilde{\mathbf{u}}) + \nabla \tilde{\rho} \big(\frac{\tilde{\T}}{\tilde{\rho}}\phi - \zeta\big) - \frac{\phi}{\tilde{\rho}} [\mu \Delta (\e \tilde{\mathbf{u}}) + (\mu + \lambda) \nabla \dv (\e \tilde{\mathbf{u}})] -\frac{\rho}{\tilde{\rho}}(\p_{y_1} \tilde{\mathbf{\Gamma}}_2 - \e \tilde{\mathbf{u}} (\p_{y_1} \tilde{\Gamma}_1)), \\
		\mathcal{R}_3 =& \frac{\mu}{2} |\e \nabla \mathbf{u} + (\e \nabla \mathbf{u})^t |^2 + \lambda |\dv (\e \mathbf{u})|^2 - \frac{\mu}{2} |\e \nabla \tilde{\mathbf{u}} + (\e \nabla \tilde{\mathbf{u}})^t|^2- \lambda |\dv (\e \tilde{\mathbf{u}})|^2 \\
		&- \rho \psi \cdot \nabla \tilde{\T} - \rho \dv (\e \tilde{\mathbf{u}}) \zeta - \frac{\phi}{\tilde{\rho}} \big[\kappa \Delta \tilde{\T} + \frac{\mu}{2} |\e \nabla \tilde{\mathbf{u}} + (\e \nabla \tilde{\mathbf{u}})^t|^2 + \lambda |\dv (\e \tilde{\mathbf{u}})|^2\big] \\ 
		&- \frac{\rho}{\tilde{\rho}} (\dv \tilde{\mathbf{\Gamma}}_3 - \e \tilde{\mathbf{u}} \cdot \dv \tilde{\mathbf{\Gamma}}_2 + (\frac{|\e \tilde{\mathbf{u}}|^2}{2} - \tilde{\T}) \p_{y_1} \tilde{\Gamma}_1).
	\end{align*}
	
	We study the perturbations for the zero mode of \((\rho, \e \mathbf{m}, \E)\). The anti-derivatives for the zero mode is crucial in our analysis, and thus we denote
	\begin{align}\label{anti}
		\Phi:= \int_{-\infty}^{y_1} \mathring{\phi} dz_1, \quad \Psi := \int_{-\infty}^{y_1} \mathring{\varphi} dz_1, \quad W:=\int_{-\infty}^{y_1} \mathring{w} dz_1.
	\end{align}
	Applying \(\mathbf{D}_0\) to (\ref{scaled NS Eqs}) and using (\ref{tilde U equs}), we have
	\begin{equation}\label{the anti equs}
		\left\{
		\begin{aligned}
			&\p_\tau \Phi + \p_{y_1} \Psi_1 = \mathcal{S}_1, \\
			&\p_\tau \Psi_1 + \p_{y_1} W = (2\mu + \lambda) \p_{y_1} (\e \mathring{u}_1 - \e \tilde{u}_1) + \mathcal{S}_{21}, \\
			&\p_\tau 
			\Psi_i = \mu \p_{y_1} (\e \mathring{u}_i - \e \tilde{u}_i) + \mathcal{S}_{2i}, ~i=2,3, \\
			&\p_\tau W +2\tilde{\T} \p_{y_1} \Psi_1 = \kappa \p_{y_1} \mathring{\zeta} + \mathcal{S}_3,
		\end{aligned}
		\right.
	\end{equation}
	where
	\begin{align*}
		&\mathcal{N} (U) = \frac{|\e m_1|^2}{\rho} - \frac{|\e \mathbf{m}|^2}{2\rho}, \qquad \mathcal{S}_1 = -\tilde{\Gamma}_1, \\
		&\mathcal{S}_{21} = \big(-\mathring{\mathcal{N}}(U) + \mathcal{N}(\tilde{U})\big) - \tilde{\mathbf{\Gamma}}_{21} := \mathcal{S}_{21}^{(1)} + \mathcal{S}_{21}^{(2)},\\
		&\mathcal{S}_{2i} = \left(-\mathbf{D}_0 \left(\frac{\e m_1 \e m_i}{\rho}\right) + \mathbf{D}_0 \left(\frac{\e \tilde{m}_1 \e \tilde{m}_i}{\tilde{\rho}}\right)\right) - \tilde{\mathbf{\Gamma}}_{2i} := \mathcal{S}_{2i}^{(1)} + \mathcal{S}_{2i}^{(2)}, \quad i = 2,3,\\
		&\mathcal{S}_3 = (2\tilde{\T} \p_{y_1} \Psi_1 - \mathring{\hat{N}}_{31}(U) + \mathring{\hat{N}}_{31}(\tilde{U})) - \tilde{\mathbf{\Gamma}}_{3} := \mathcal{S}_3^{(1)} + \mathcal{S}_3^{(2)}, \quad \hat{N}_3 (U) = \frac{\e \mathbf{m} \E}{\rho} + p \frac{\e \mathbf{m}}{\rho} - \frac{\e \mathbf{m}}{\rho} \mathbb{S}(\e \mathbf{u}).
	\end{align*}
	
	\subsection{Statement of the a priori estimates}
	
	\subsubsection{Local existence theorem \ref{Uniform estimates for well-prepared initial data}.}

	In this subsection, we use \(\Omega_\e\) to replace \(\mathbb{R} \times \frac{\mathbb{T}^2}{\e^2}\) for convenience. 
	We state the local existence theorem first. Let us introduce the solution space.
	\begin{align*}
		X_{m,M} := \sup_{0 \le \tau \le T} \{ (\phi, \psi, \zeta, \Phi, \Psi, W) |& (\phi, \psi, \zeta) \in C(0, +\infty; H^3(\Omega_\e)), \quad (\phi, \nabla \psi, \nabla \zeta) \in L^2(0,+\infty; H^3(\Omega_\e)), \\
		&(\Phi, \Psi, W) \in C(0, +\infty; H^2(\mathbb{R})), \quad (\p_{y_1} \Phi, \p_{y_1} \Psi, \p_{y_1} W) \in L^2(0, +\infty; H^2(\mathbb{R})), \\
		&\|(\Phi, \Psi, W)\|_{H^2(\mathbb{R})} + \|(\phi, \psi, \zeta)\|_{H^3(\Omega_\e)} \le M, \\
		&\inf_{y, \tau} (\phi+\bar{\rho}) \ge m >0. \}
	\end{align*}
	Using a standard argument as \cite{HXY,HWW,MatsumuraNishida1980}, we have the property of local existence
	\begin{Prop}\label{local existence}
		There exist constants \(m, M >0\), so that if the initial data satisfies that
		\begin{align*}
			\|(\Phi_0, \Psi_0, W_0)\|_{H^3(\mathbb{R})} + \|(\phi_0, \psi_0, \zeta_0)\|_{H^3(\Omega_\e)} + \|(\phi_0, \varphi_0, w_0)\|_{L^1(\Omega_\e)} \le M, \quad \text{and} \quad \inf_y (\phi_0 + \bar{\rho}(y_1, 0)) \ge m > 0,
		\end{align*}
		then there exists a positive time \(T = T(m,M) >0\), so that the system (\ref{scaled NS Eqs}) admits a unique solution \((\rho, \mathbf{u}, \T)\) satisfying
		\begin{align*}
			(\phi, \psi, \zeta, \Phi, \Psi, W) \in X_{\frac{1}{2}m, 2M} (T).
		\end{align*} 
		Moreover, \(\|(\Phi, \Psi, W)\|_{L^\infty (\Omega_\e)} \le 2M.\)
	\end{Prop}
	
	\subsubsection{The a priori assumptions.}
	
	Frist, we assume the following a priori assumptions:
	\begin{equation}\label{a priori assumption}
		\left\{
		\begin{aligned}
			&\sup_{0 \le \tau \le T} \biggl\{\|(\Phi, \Psi, W)\|_{L^\infty_y(\mathbb{R})}^2 + \frac{(1+\e^2 \tau)^{\frac{1}{2}}}{\e} \|(\mathring{\phi}, \mathring{\psi}, \mathring{\zeta})\|_{L^2_y(\mathbb{R})}^2 + \frac{(1+\e^2 \tau)^{\frac{3}{2}}}{\e^3} \|\p_{y_1}(\mathring{\phi}, \mathring{\psi}, \mathring{\zeta})\|_{L^2_y(\mathbb{R})}^2 \\
			& \qquad \qquad+\frac{(1+\e^2 \tau)^{\frac{3}{2}}}{\e^2} \|\nabla_y^2 (\phi, \psi, \zeta)\|_{L_y^2(\Omega_\e)} ^2+ \frac{(1+\e^2 \tau)^{\frac{3}{2}}}{\e^2} \|\nabla_y^3 (\phi, \psi, \zeta)\|_{L_y^2(\Omega_\e)}^2 \biggr\} \le \chi^2, \\
			&\sup_{0 \le \tau \le T} (\e^{-2} + \tau)^{\frac{N}{2}} \|(\phi_{\neq}, \psi_{\neq}, \zeta_{\neq})\|_{H^1_y(\Omega_\e)}^2 \le C_N \e^{-2N-3}  e^{-c_0 \e^{-2}}, \qquad \forall N \in \mathbb{Z}^+, \quad \forall N \in \mathbb{Z}^+,
		\end{aligned}
		\right.
	\end{equation}
	where \(\chi\) is a small constant determined later, and \(C_N\) is a positive constant only depending on \(N\). To close the above a priori assumptions, in view of the local existence result, it suffices to establish the following a priori estimates for the proof of Theorem \ref{Uniform estimates for well-prepared initial data}.
	\begin{Prop}[a priori estimates]\label{a priori estimates}
		Assume that \((\phi, \psi, \zeta,\Phi,\Psi,W)\) is the unique solution given in Proposition \ref{local existence} and satisfies the  a priori assumptions \cref{a priori assumption}. Then the following estimates hold
		\begin{equation}
			\left\{
			\begin{aligned}
				&\|(\Phi, \Psi, W)\|_{L^\infty_y(\mathbb{R})}^2 \le C (\sqrt{\eta}+\delta) \e, \\
				&\|(\mathring{\phi}, \mathring{\psi}, \mathring{\zeta})\|_{L^2_y(\mathbb{R})}^2 \le C (\sqrt{\eta}+\delta) \e^2 (1+\e^2 \tau)^{-\frac{1}{2}}, \\
				&\|\p_{y_1} (\mathring{\phi}, \mathring{\psi}, \mathring{\zeta})\|_{L^2_y(\mathbb{R})}^2 \le C (\sqrt{\eta}+\delta) \e^4 (1+\e^2 \tau)^{-\frac{3}{2}}, \\
				&\|\nabla_y^2 (\phi, \psi, \zeta)\|_{L^2_y(\Omega_\e)}^2 \le C \delta \e^2 (1+\e^2 \tau)^{-\frac{3}{2}}, \\
				&\|\nabla_y^3 (\phi, \psi, \zeta)\|_{L^2_y(\Omega_\e)}^2 \le C \delta \e^2 (1+\e^2 \tau)^{-\frac{3}{2}}, \\
				&\|(\phi_{\neq}, \psi_{\neq}, \zeta_{\neq})\|_{H^1_y(\Omega_\e)}^2 \le C \e^{-2N-3} (\e^{-2} + \tau)^{-\frac{N}{2}} e^{-c_0 \e^{-2}}, \qquad \forall N \in \mathbb{Z}^+.
			\end{aligned}
			\right.
		\end{equation}
	\end{Prop}
	\subsubsection*{Proof of Theorem \ref{Uniform estimates for well-prepared initial data}}
	It suffices to consider \(\rho^\e -\bar{\rho}\),
	since the other terms can be treated in the same way. By \eqref{tilde} and Proposition \ref{a priori estimates}, it holds that
	\begin{align*}
		&\norm{\rho^\e -\bar{\rho}}_{L^2(\R \times \Torus^2)}^2 \le \norm{\rho^\e - \tilde{\rho}}_{L^2(\R \times \Torus^2)}^2 + \norm{\tilde{\rho} - \bar{\rho}}_{L^2(\R \times \Torus^2)}^2 \\
		\le& \norm{\mathring{\rho}^\e - \tilde{\rho}}_{L^2(\R)}^2 + \norm{\rho_{\neq}^\e}_{L^2(\R \times \Torus^2)}^2 + \norm{\tilde{\rho} - \bar{\rho}}_{L^2(\R \times \Torus^2)}^2 \\
		\le& C (\sqrt{\eta} + \delta) \e^3 (1+t)^{-\frac{1}{2}}.
	\end{align*}
	By Proposition \ref{a priori estimates}   and Proposition \ref{local existence},  employing the standard continuity argument, we finally obtain Theorem \ref{Uniform estimates for well-prepared initial data}. In the next section, we give the proof of Proposition \ref{a priori estimates}.
	
	\subsection{Estimates for the zero mode}
	
	We introduce the following important transformation:
	\begin{align}\label{tilde Phi, Psi, W}
		\tilde{\Phi} = \tilde{\T} \Phi, \qquad \tilde{\Psi} = \Psi - \e\tilde{\mathbf{u}} \Phi \qquad \tilde{W} = W - \e \tilde{\mathbf{u}} \cdot \tilde{\Psi} - \tilde{\T} \Phi.
	\end{align}
	Taking the first-order derivative with respect to space, we obtain
	\begin{equation}\label{p_1 psi, W}
		\begin{aligned}
			&\p_{y_1} \Psi_i = \p_{y_1} \tilde{\Psi}_i + \p_{y_1} (\e \tilde{\mathbf{u}} \Phi) = \p_{y_1} \tilde{\Psi}_i + \e \tilde{\mathbf{u}} \p_{y_1} \Phi + \p_{y_1}(\e \tilde{u}_i) \Phi := \p_{y_1} \tilde{\Psi}_i + \mathcal{C}_{ir}, \quad i=1,2,3, \\
			&\p_{y_1} W = \p_{y_1} \tilde{W} + \p_{y_1} (\tilde{\T} \Phi) + \p_{y_1} (\e \tilde{\mathbf{u}} \cdot \tilde{\Psi}) := \p_{y_1} \tilde{W} + \p_{y_1} \tilde{\Phi} + \mathcal{C}_{wr}.
		\end{aligned}
	\end{equation}
	Moreover, taking the first-order derivative with respect to time, we obtain
	\begin{equation}\label{p_tau psi, W}
		\begin{aligned}
			\p_\tau \Psi_i =& \p_\tau \tilde{\Psi}_i + \e \p_\tau \tilde{u}_i \Phi + \e \tilde{u}_i \mathcal{S}_1 - \e \tilde{u}_i \p_{y_1} \Psi_1 := \p_\tau \tilde{\Psi}_i + \mathcal{C}_{id}, \\
			\p_\tau W =& \p_\tau \tilde{W} - \tilde{\T} \p_{y_1} \tilde{\Psi}_1 - \tilde{\T} \p_{y_1}(\e \tilde{u}_1 \Phi) + \tilde{\T} \mathcal{S}_1 + \e \tilde{\mathbf{u}} \cdot \p_\tau \tilde{\Psi} + \p_\tau (\e \tilde{\mathbf{u}}) \cdot \tilde{\Psi} + \p_\tau \tilde{\T} \Phi \\
			:=& \p_\tau \tilde{W} - \tilde{\T} \p_{y_1} \tilde{\Psi}_1 + \mathcal{C}_{wd}.
		\end{aligned}
	\end{equation}
	According to (\ref{the anti equs}), (\ref{tilde Phi, Psi, W}), (\ref{p_1 psi, W}) and (\ref{p_tau psi, W}), the system of \((\tilde{\Phi}, \tilde{\Psi}, \tilde{W})\) can be written as follows:
	\begin{equation}\label{the equ of tilde Phi, Psi, W}
		\left\{
		\begin{aligned}
			&\p_\tau \tilde{\Phi} + \tilde{\T} \p_{y_1} \tilde{\Psi}_1 = \tilde{\T} \mathcal{S}_1 + Y_1, \\
			&\p_\tau \tilde{\Psi}_1 + \p_{y_1} (\tilde{W} + \tilde{\Phi}) = \frac{2\mu + \lambda}{\tilde{\rho}} \p_{y_1}^2 \tilde{\Psi}_1 + \mathcal{S}_{21} + Y_{21}, \\
			&\p_\tau \tilde{\Psi}_i = \frac{\mu}{\tilde{\rho}} \p_{y_1}^2 \tilde{\Psi}_i + \mathcal{S}_{2i} + Y_{2i}, \qquad i=2,3, \\
			&\p_\tau \tilde{W} + \tilde{\T} \p_{y_1} \tilde{\Psi}_1 = \frac{\kappa}{\tilde{\rho}} \p_{y_1}^2 \tilde{W} + \mathcal{S}_3 + Y_3,
		\end{aligned}
		\right.
	\end{equation}
	where
	\begin{align*}
		&Y_1 = \p_\tau \tilde{\T} \Phi - \tilde{\T} \mathcal{C}_{1r}, \\
		&Y_{21} = [-\mathcal{C}_{1d} -\mathcal{C}_{wr}] + \left[(2\mu + \lambda) \p_{y_1} (\e \mathring{u}_1 - \e \tilde{u}_1) - \frac{2\mu+\lambda}{\tilde{\rho}} \p_{y_1}^2 \tilde{\Psi}_1\right] :=Y_{21}^{(1)} + Y_{21}^{(2)}, \\
		&Y_{2i} = -\mathcal{C}_{id} +\left[\mu \p_{y_1} (\e \mathring{u}_i- \e \tilde{u}_i) - \frac{\mu}{\tilde{\rho}} \p_{y_1}^2 \tilde{\Psi}_i\right] := Y_{2i}^{(1)} + Y_{2i}^{(2)}, \qquad i=2,3,\\
		&Y_3 = [-\mathcal{C}_{wd} - 2\tilde{\T} \mathcal{C}_{1r}] + \left[\kappa \p_{y_1} \mathring{\zeta} - \frac{\kappa}{\tilde{\rho}} \p_{y_1}^2 \tilde{W}\right] := Y_3^{(1)} + Y_3^{(2)}.
	\end{align*}
	For convenience, we use the following  notations:
	\begin{equation}
		\begin{aligned}
			V :=& (\tilde{\Phi}, \tilde{\Psi}, \tilde{W})^t, \qquad v:= (\phi, \psi, \zeta)^t, \qquad \tilde{V} := (\tilde{\Phi}, \tilde{\Psi}_1, \tilde{W})^t, \\
			\mathcal{D}^{(k)}_1 :=& (\e+\delta)\sum_{j=1}^{k+2} D^1_{-\frac{j}{2}} | \p_{y_1}^{k-j+2} V|, \qquad \mathcal{D}^{(k)}_2 := \sqrt{\eta}\e^{\frac{1}{2}}\sum_{j=1}^{k+2} D^2_{-\frac{j}{2}} | \p_{y_1}^{k-j+2} V|, \\
			\mathcal{L}^{(0)}_1 :=& |\p_{y_1}V|^2 + (\delta + \e) D^1_{-\frac{1}{2}} | \p_{y_1}V|, \qquad \mathcal{L}^{(0)}_2 := \sqrt{\eta} \e^{\frac{1}{2}} D^2_{-\frac{1}{2}} | \p_{y_1}V|, \\
			\mathcal{L}^{(1)}_1 :=& |\p_{y_1}^2 V| |\p_{y_1} V| + (\delta + \e) (D_{-\frac{1}{2}}^1 |\p_{y_1} V|^2 + D_{-1}^1 |\p_{y_1} V| + D_{-\frac{1}{2}}^1 |\p_{y_1}^2 V|), \\
			\mathcal{L}^{(1)}_2 :=& \sqrt{\eta}\e^{\frac{1}{2}} (D_{-\frac{1}{2}}^2 |\p_{y_1} V|^2 + D_{-1}^2 |\p_{y_1} V| + D_{-\frac{1}{2}}^2 |\p_{y_1}^2 V|), \\
			\mathcal{L}^{(2)}_1 :=& |\p_{y_1}^3 V| |\p_{y_1}V| + |\p_{y_1}^2 V|^2 + (\delta+\e) D_{-\frac{3}{2}}^1 |\p_{y_1} V| \\
			&+ (\delta+\e) D_{-\frac{1}{2}}^1 (|\p_{y_1} V| |\p_{y_1}^2 V| + |\p_{y_1}^3 V|) + (\delta+\e) D_{-1}^1 (|\p_{y_1}^2 V| + |\p_{y_1}V|^2), \\
			\mathcal{L}^{(2)}_2 :=& \sqrt{\eta} \e^{\frac{1}{2}} D_{-\frac{3}{2}}^2 |\p_{y_1} V| + \sqrt{\eta} \e^{\frac{1}{2}} D_{-\frac{1}{2}}^2 (|\p_{y_1} V| |\p_{y_1}^2 V| + |\p_{y_1}^3 V|) + \sqrt{\eta} \e^{\frac{1}{2}} D_{-1}^2 (|\p_{y_1}^2 V| + |\p_{y_1}V|^2), \\
			\mathcal{I}^{(0)} :=& |v_{\neq}|^2, \qquad\qquad\qquad\qquad \mathcal{I}^{(1)} := |v_{\neq}| |\nabla v_{\neq}|, \\
			\mathcal{I}^{(2)} :=& |\nabla v_{\neq}|^2 + |\nabla^2 v_{\neq}| |v_{\neq}|, \qquad \mathcal{I}^{(3)} := |\nabla v_{\neq}| |\nabla^2 v_{\neq}| + |\nabla^3 v_{\neq}| |v_{\neq}|. 
		\end{aligned}
	\end{equation}
	Moreover, for \(k=0,1,2\), we denote
	\begin{align*}
		\mathcal{D}^{(k)} = \mathcal{D}^{(k)}_1 + \mathcal{D}^{(k)}_2, \qquad \mathcal{L}^{(k)} = \mathcal{L}^{(k)}_1 +\mathcal{L}^{(k)}_2.
	\end{align*}
	First, we estimate the terms of  \(\mathcal{I}^{(i)}\), \(i=0,1,2,3\) as follows:
	\begin{Lem}\label{Z est}
		Under the same assumptions as Proposition \ref{a priori estimates}, it holds that
		\begin{align}
			\|\mathcal{I}^{(i)}\|_{L^2(\Omega_\e)} \le C \e^{-\frac{2N+3}{8}} e^{-\frac{c_0\e^{-2}}{8}} (\e^{-2} + \tau)^{-\frac{N}{16}}, \qquad i= 0,1,2,3,
		\end{align}
		where \(N\) is the same as that in Theorem \ref{Uniform estimates for well-prepared initial data}.
	\end{Lem}

	\begin{proof}
		According to the a priori assumptions (\ref{a priori assumption}) and the Gagliardo-Nirenberg inequality, one has
		\begin{align*}
			\|\mathcal{I}^{(3)}\|_{L^2(\Omega_\e)} \lesssim& \|\nabla v_{\neq}\|_{L^{\infty}(\Omega_\e)} \|\nabla^2 v_{\neq}\|_{L^2(\Omega_\e)} + \|v_{\neq}\|_{L^{\infty}(\Omega_\e)} \|\nabla^3 v_{\neq}\|_{L^2(\Omega_\e)} \\
			\lesssim& \e^{\frac{3}{2}} \|\nabla_x v_{\neq} \|_{L^\infty(\Omega)} \|\nabla_x^2 v_{\neq} \|_{L^2(\Omega)} + \e^{\frac{3}{2}} \| v_{\neq} \|_{L^\infty(\Omega)} \|\nabla_x^3 v_{\neq} \|_{L^2(\Omega)} \\
			\lesssim& \e^{\frac{3}{2}} \left[\|\nabla_x v_{\neq} \|_{L^2(\Omega)}^{\frac{1}{4}} \|\nabla_x^3 v_{\neq} \|_{L^2(\Omega)}^{\frac{3}{4}} \|\nabla_x^2 v_{\neq} \|_{L^2(\Omega)} \right] \\
			&+ \e^{\frac{3}{2}} \left[ \|v_{\neq} \|_{L^2(\Omega)}^{\frac{1}{4}} \|\nabla_x^2 v_{\neq} \|_{L^2(\Omega)}^{\frac{3}{4}} \|\nabla_x^3 v_{\neq} \|_{L^2(\Omega)}\right] \\
			\lesssim& \e^{\frac{3}{2}} \left[\e^{\frac{3}{8} - 3 \cdot \frac{3}{8} - \frac{1}{2}} \|\nabla v_{\neq} \|_{L^2(\Omega_\e)}^{\frac{1}{4}} \|\nabla^3 v_{\neq} \|_{L^2(\Omega_\e)}^{\frac{3}{4}} \|\nabla^2 v_{\neq} \|_{L^2(\Omega_\e)}\right] \\
			&+ \e^{\frac{3}{2}} \left[\e^{\frac{3}{8} - 1 \cdot \frac{3}{8} - 3 \cdot \frac{1}{2}} \|v_{\neq} \|_{L^2(\Omega_\e)}^{\frac{1}{4}} \|\nabla^2 v_{\neq} \|_{L^2(\Omega_\e)}^{\frac{3}{4}} \|\nabla^3 v_{\neq} \|_{L^2(\Omega_\e)}\right] \\
			\lesssim& \|\nabla v_{\neq} \|_{L^2(\Omega_\e)}^{\frac{1}{4}} \|\nabla^3 v_{\neq} \|_{L^2(\Omega_\e)}^{\frac{3}{4}} \|\nabla^2 v_{\neq} \|_{L^2(\Omega_\e)} + \|v_{\neq} \|_{L^2(\Omega_\e)}^{\frac{1}{4}} \|\nabla^2 v_{\neq} \|_{L^2(\Omega_\e)}^{\frac{3}{4}} \|\nabla^3 v_{\neq} \|_{L^2(\Omega_\e)} \\
			\lesssim& \e^{-\frac{2N+3}{8}} e^{-\frac{c_0\e^{-2}}{8}} (\e^{-2} + \tau)^{-\frac{N}{16}}.
		\end{align*}
		Similarly, we have
		\begin{align*}
			\|\mathcal{I}^{(2)}\|_{L^2(\Omega_\e)} \lesssim& \|\nabla v_{\neq}\|_{L^{\infty}(\Omega_\e)} \|\nabla v_{\neq}\|_{L^2(\Omega_\e)} + \|v_{\neq}\|_{L^{\infty}(\Omega_\e)} \|\nabla^2 v_{\neq} \|_{L^2(\Omega_\e)} \\
			\lesssim& \| \nabla v_{\neq} \|_{L^2(\Omega_\e)}^{\frac{1}{4}} \|\nabla^3 v_{\neq} \|_{L^2(\Omega_\e)}^{\frac{3}{4}} \|\nabla v_{\neq} \|_{L^2(\Omega_\e)} + \|v_{\neq} \|_{L^2(\Omega_\e)}^\frac{1}{4} \|\nabla^2 v_{\neq} \|_{L^2(\Omega_\e)}^{\frac{3}{4}} \|\nabla^2 v_{\neq} \|_{L^2(\Omega_\e)} \\
			\lesssim& \e^{-\frac{2N+3}{8}} e^{-\frac{c_0\e^{-2}}{8}} (\e^{-2} + \tau)^{-\frac{N}{16}}, \\
			\|\mathcal{I}^{(1)}\|_{L^2(\Omega_\e)} \lesssim& \|v_{\neq}\|_{L^{\infty}(\Omega_\e)} \|\nabla v_{\neq}\|_{L^2(\Omega_\e)} \le \| v_{\neq} \|_{L^2(\Omega_\e)}^{\frac{1}{4}} \|\nabla^2 v_{\neq} \|_{L^2(\Omega_\e)}^{\frac{3}{4}} \|\nabla v_{\neq} \|_{L^2(\Omega_\e)} \\
			\lesssim& \e^{-\frac{2N+3}{8}} e^{-\frac{c_0\e^{-2}}{8}} (\e^{-2} + \tau)^{-\frac{N}{16}},
		\end{align*}
		and
		\begin{align*}
			\|\mathcal{I}^{(0)}\|_{L^2(\Omega_\e)} \lesssim& \|v_{\neq}\|_{L^{\infty}(\Omega_\e)} \|v_{\neq}\|_{L^2(\Omega_\e)} \lesssim \|v_{\neq}\|_{L^2(\Omega_\e)}^{\frac{1}{4}} \|\nabla^2 v_{\neq}\|_{L^2(\Omega_\e)}^{\frac{3}{4}}
			\lesssim \e^{-\frac{2N+3}{8}} e^{-\frac{c_0\e^{-2}}{8}} (\e^{-2} + \tau)^{-\frac{N}{16}}.
		\end{align*}
		Therefore, we have completed the estimates for \(\|\mathcal{I}^{(i)}\|_{L^2(\Omega_\e)}\), \(i=0,1,2\).
	\end{proof}
	
	\begin{Lem}\label{Q est}
		Under the same assumptions as Proposition \ref{a priori estimates}, it holds that
		\begin{align*}
			&|\p_{y_1}^k \mathcal{S}_{2i}^{(1)}| \le C \left(\mathcal{L}^{(k)} +\mathbf{D}_0 \mathcal{I}^{(k)}\right), \quad k=0,1,2,~i=1,2,3, \\
			&|\p_{y_1}^k \mathcal{S}_3^{(1)}| \le C \left(\mathcal{L}^{(k)} +\mathcal{L}^{(k+1)} + \mathbf{D}_0 \mathcal{I}^{(k)}+\mathbf{D}_0\mathcal{I}^{(k+1)}\right), \quad k=0,1, \\
			&|\p_{y_1}^k \mathcal{S}_{21}^{(2)}| \le C (\delta+\e) D_{-\frac{2+k}{2}}^1 + C \sqrt{\eta} \e^{\frac{1}{2}} D_{-\frac{2+k}{2}}^2, \quad k=0,1,2, \\
			&|\p_{y_1}^k \mathcal{S}_{2i}^{(2)}| \le C \eta^{\frac{1}{2}} \e^{\frac{1}{2}} D_{-\frac{2+k}{2}}^1, \quad k=0,1,2,~i=2,3, \\
			&|\p_{y_1}^k \mathcal{S}_{3}^{(2)}| \le C (\delta + \e) D_{-\frac{3+k}{2}}^1 + C \sqrt{\eta} \e^{\frac{1}{2}} D_{-\frac{2+k}{2}}^2, \quad k=0,1,2.
		\end{align*}
	\end{Lem}

	\begin{proof}
		First, we estimate \(\mathcal{S}_{21}^{(1)}\). Notice that
		\begin{align*}
			-\mathcal{N}(\mathring{U}) + \mathring{\mathcal{N}}(\tilde{U}) = - \left(\mathcal{N}(\mathring{U}) - \mathcal{N}(\tilde{U})\right) +  \left(\mathcal{N}(\mathring{U}) - \mathring{\mathcal{N}}(U)\right).
		\end{align*}
		On the one hand, we have
		\begin{align*}
			\mathcal{N}(\mathring{U}) - \mathcal{N}(\tilde{U}) - \nabla \mathcal{N}(\tilde{U}) (\mathring{U} - \tilde{U}) \lesssim \left[|\p_{y_1}\Psi|^2 + |\p_{y_1} \Phi|^2\right].
		\end{align*}
		A direct calculations gives that
		\begin{align*}
			\nabla \mathcal{N}(\tilde{U}) (\mathring{U} - \tilde{U}) =& \left[\frac{\e \tilde{m}_1}{\tilde{\rho}} \p_{y_1} \Psi_1 - \frac{|\e \tilde{m}_1|^2}{2\tilde{\rho}^2} \p_{y_1} \Phi\right] - \sum_{i=2,3} \left[\frac{\e \tilde{m}_i}{\tilde{\rho}} \p_{y_1} \Psi_i - \frac{|\e \tilde{m}_i|^2}{2\tilde{\rho}^2} \p_{y_1} \Phi\right] \\
			\le& C (\delta+\e) D_{-\frac{1}{2}}^1 |\p_{y_1}V| + C \sqrt{\eta} \e^{\frac{1}{2}}  D_{-\frac{1}{2}}^2 |\p_{y_1}V| \\
			\le& C\mathcal{L}^{(0)}.
		\end{align*}
		On the other hand, 
		\begin{align*}
			\mathcal{N}(\mathring{U}) - \mathring{\mathcal{N}}(U) = \mathbf{D}_0 \left[\frac{|\e \mathring{m}_1|^2}{\mathring{\rho}} - \frac{|\e \mathbf{\mathring{m}}|^2}{2\mathring{\rho}} - \frac{|\e m_1|^2}{\rho} + \frac{|\e \mathbf{m}|^2}{2\rho}\right] \le C\mathbf{D}_0(|v_{\neq}|^2) \le C \mathbf{D}_0 \mathcal{I}^{(0)}.
		\end{align*}
		Combining the above two inequalities, we find that
		\begin{align*}
			|-\mathcal{N}(\mathring{U}) + \mathring{\mathcal{N}}(\tilde{U})| \le C\mathcal{L}^{(0)} + C \mathbf{D}_0 \mathcal{I}^{(0)}.
		\end{align*}
		Similarly, for \(\mathcal{S}_{22}^{(1)}\). \(\mathcal{S}_{23}^{(1)}\) and \(\mathcal{S}_3^{(1)}\), it is easy to check that
		\begin{align*}
			\left|-\mathbf{D}_0 \left(\frac{\e m_1 \e m_i}{\rho} + \frac{\e \tilde{m}_1 \e \tilde{m}_i}{\tilde{\rho}}\right)\right| \le& C|\p_{y_1} V|^2 + C (\delta +\e) D_{-\frac{1}{2}}^1 |\p_{y_1} V| + C \sqrt{\eta} \e^{\frac{1}{2}} D_{-\frac{1}{2}}^2 |\p_{y_1} V| + C \mathbf{D}_0 (|v_{\neq}|^2) \\
			\le& C\mathcal{L}^{(0)} + C \mathbf{D}_0 \mathcal{I}^{(0)}, 
		\end{align*}
		and
		\begin{align*}
			&|2\tilde{\T} \p_{y_1} \Psi_1 - \mathring{\hat{N}}_{31} (U) + \mathring{\hat{N}}_{31} (\tilde{U})| \\
			\le& C(|\p_{y_1} V|^2 + |\p_{y_1}^2 V| |\p_{y_1}V|) + C(\delta+\e)  D_{-\frac{1}{2}}^1 |\p_{y_1} V| + C \sqrt{\eta} \e^{\frac{1}{2}} D_{-\frac{1}{2}}^2 |\p_{y_1} V| \\
			&+ C (\delta+\e) \left(D_{-\frac{1}{2}}^1 |\p_{y_1} V|^2 + D_{-1}^1 |\p_{y_1} V| + D_{-\frac{1}{2}}^1 |\p_{y_1}^2 V|\right) \\
			&+C \sqrt{\eta} \e^{\frac{1}{2}} \left(D_{-\frac{1}{2}}^2 |\p_{y_1} V|^2 + D_{-1}^2 |\p_{y_1} V| + D_{-\frac{1}{2}}^2 |\p_{y_1}^2 V|\right) + C \mathbf{D}_0 (|v_{\neq}| |\nabla v_{\neq}|) \\
			\le& C \left(\mathcal{L}^{(0)} +\mathcal{L}^{(1)}+\mathbf{D}_0 \mathcal{I}^{(1)}\right), 
		\end{align*}
		Moreover, we have
		\begin{align*}
			&\left|\p_{y_1} \left[\mathring{\mathcal{N}} (U) - \mathcal{N} (\tilde{U})\right]\right| + \left|\p_{y_1} \left[-\mathbf{D}_0 \left(\frac{\e m_1 \e m_i}{\rho}\right) + \mathbf{D}_0 \left(\frac{\e \tilde{m}_1 \e \tilde{m}_i}{\tilde{\rho}}\right)\right]\right| \\
			\le& C|\p_{y_1}^2 V| |\p_{y_1}V| + C(\delta+\e) \left(D_{-\frac{1}{2}}^1 |\p_{y_1}V|^2 + D_{-1}^1 |\p_{y_1} V| + D_{-\frac{1}{2}}^1 |\p_{y_1}^2 V| \right) \\
			&+ C \sqrt{\eta} \e^{\frac{1}{2}} \left(D_{-\frac{1}{2}}^2 |\p_{y_1}V|^2 + D_{-1}^2 |\p_{y_1} V| + D_{-\frac{1}{2}}^2 |\p_{y_1}^2 V| \right) + C \mathbf{D}_0(|\nabla v_{\neq} | |v_{\neq}|) \\
			\le& C \left(\mathcal{L}^{(1)}_1 + \mathcal{L}^{(1)}_2+\mathbf{D}_0\mathcal{I}^{(1)}\right), \\
			&\left|\p_{y_1}^2 \left[\mathcal{N}(\mathring{U}) - \mathcal{N} (\tilde{U})\right]\right| + \left|\p_{y_1}^2 \left[-\mathbf{D}_0 \left(\frac{\e m_1 \e m_i}{\rho}\right) + \mathbf{D}_0 \left(\frac{\e \tilde{m}_1 \e \tilde{m}_i}{\tilde{\rho}}\right)\right]\right| \\
			\le& C (|\p_{y_1}^3 V| |\p_{y_1} V| + |\p_{y_1}^2 V|^2) + C (\delta+\e) D_{-\frac{3}{2}}^1 |\p_{y_1}V| + C\sqrt{\eta} \e^{\frac{1}{2}} D_{-\frac{3}{2}}^2 |\p_{y_1}V| + C(\delta+\e) D_{-1}^1 (|\p_{y_1}^2V| + |\p_{y_1} V|^2) \\
			&+ C\sqrt{\eta} \e^{\frac{1}{2}} D_{-1}^2 (|\p_{y_1}^2V| + |\p_{y_1} V|^2) + C (\delta+\e) D_{-\frac{1}{2}}^1 (|\p_{y_1}V| |\p_{y_1}^2 V| + |\p_{y_1}^3 V|) \\
			&+ C\sqrt{\eta} \e^{\frac{1}{2}}  D_{-\frac{1}{2}}^2 (|\p_{y_1}V| |\p_{y_1}^2 V| + |\p_{y_1}^3 V|) + C\mathbf{D}_0(|\nabla v_{\neq}|^2 + |\nabla^2 v_{\neq}| |v_{\neq}|) \\
			\le& C \left(\mathcal{L}^{(2)}+\mathbf{D}_0\mathcal{I}^{(2)}\right),
		\end{align*}
		and
		\begin{align*}
			&\left| \p_{y_1} \left[2\tilde{\T} \p_{y_1} \Psi_1 - \mathring{\hat{N}}_{31} (U) + \mathring{\hat{N}}_{31} (\tilde{U})\right]\right| \\
			\le& C (|\p_{y_1}^2V| |\p_{y_1}V| + |\p_{y_1}^3 V| |\p_{y_1}V| + |\p_{y_1}^2 V|^2) + C\mathbf{D}_0 (|\nabla v_{\neq}|^2 + |\nabla^2 v_{\neq}| |v_{\neq}|)\\
			&+ C (\delta+\e) D_{-\frac{1}{2}}^1 (|\p_{y_1} V| |\p_{y_1}^2 V| + |\p_{y_1}^2 V| + |\p_{y_1}^3 V|) + C\sqrt{\eta} \e^{\frac{1}{2}} D_{-\frac{1}{2}}^2 (|\p_{y_1} V| |\p_{y_1}^2 V| + |\p_{y_1}^2 V| + |\p_{y_1}^3 V|) \\
			&+ C (\delta + \e) D_{-1}^1 (|\p_{y_1}V| + |\p_{y_1}^2 V| + |\p_{y_1} V|^2) + C\sqrt{\eta}\e^{\frac{1}{2}} D_{-1}^2 (|\p_{y_1}V| + |\p_{y_1}^2 V| + |\p_{y_1} V|^2) \\
			&+ C (\delta+\e) D_{-\frac{1}{2}}^1 |\p_{y_1}^2V| + C \sqrt{\eta} \e^{\frac{1}{2}} D_{-\frac{1}{2}}^2 |\p_{y_1}^2V| \\
			\le& C \left(\mathcal{L}^{(0)} +\mathcal{L}^{(1)}+\mathcal{L}^{(2)}+\mathbf{D}_0 \mathcal{I}^{(2)}\right).
		\end{align*}
		Then we complete the proof of Lemma \ref{Q est}.
	\end{proof}
	
	\begin{Lem} \label{J est}
		Under the same assumptions as Proposition \ref{a priori estimates}, it holds that
		\begin{align*}
			&|\p_{y_1}^k Y_1, \p_{y_1}^k Y_{2i}^{(1)}| \le C \mathcal{D}^{(k)}, \quad k=0,1,2, ~i=1,2,3, \\
			&|\p_{y_1}^k Y_{2i}^{(2)}, \p_{y_1}^k Y_3^{(2)}| \le C(\mathcal{D}^{(k+1)} + \mathcal{L}^{(k+1)} + \mathbf{D}_0 \mathcal{I}^{(k)} + \mathbf{D}_0\mathcal{I}^{(k+1)}), \quad k=0,1, ~i=1,2,3, \\
			&|\p_{y_1}^k Y_3^{(1)}| \le C(\delta + \e) D_{-\frac{k+3}{2}}^1 + C \sqrt{\eta} \e^{\frac{1}{2}} D_{-\frac{k+2}{2}}^2 + C (\mathcal{D}^{(k)}+\mathcal{D}^{(k+1)}+\mathcal{L}^{(k+1)}+\mathbf{D}_0 \mathcal{I}^{(k)} + \mathbf{D}_0 \mathcal{I}^{(k+1)}), \quad k=0,1.
		\end{align*}
	\end{Lem}
	
	\begin{proof}
		We take the estimation of \(Y_3\) as an example. First, in order to estimate \(Y_3^{(2)}\), we write \(\kappa \left(\p_{y_1} \mathring{\zeta} - \frac{\p_{y_1}^2 \tilde{W}}{\tilde{\rho}}\right)\) as
		\begin{align*}
			\kappa \left(\p_{y_1} \mathring{\zeta} - \frac{\p_{y_1}^2 \tilde{W}}{\tilde{\rho}}\right) =& \kappa \p_{y_1} \mathbf{D}_0 \left[\frac{\E}{\rho} - \frac{\mathring{\E}}{\mathring{\rho}} + \frac{|\e \mathring{\mathbf{m}}|^2}{2\mathring{\rho}^2}- \frac{|\e \mathbf{m}|^2}{2\rho^2}\right] - \kappa \p_{y_1} \left[\frac{|\e \mathring{\mathbf{m}}|^2}{2\mathring{\rho}^2} - \frac{|\e \tilde{\mathbf{m}}|^2}{2\tilde{\rho}^2} + \frac{|\e \tilde{\mathbf{u}}|^2 \mathring{\phi}}{2\mathring{\rho}}\right] \\
			&+ \kappa \left[\p_{y_1} \left(\frac{\mathring{w}}{\mathring{\rho}}\right) - \frac{\p_{y_1}^2 W}{\tilde{\rho}} - \p_{y_1} \left(\frac{\mathring{\phi} \tilde{\T}}{\mathring{\rho}}\right) + \frac{\p_{y_1}^2 (\tilde{\T} \Phi)}{\tilde{\rho }} + \frac{\p_{y_1} \mathcal{C}_{wr}}{\tilde{\rho}}\right], \\
			:=& \mathcal{G}_1 +\mathcal{G}_2 + \mathcal{G}_3.
		\end{align*}
		By direct calculations, we obtain
		\begin{align*}
			|\mathcal{G}_1| \le C\mathbf{D}_0(|v_{\neq}| |v_{\neq}|+|\nabla v_{\neq}| |v_{\neq}|).
		\end{align*}
		For \(\mathcal{G}_2\), since
		\begin{align*}
			\frac{|\e \mathring{\mathbf{m}}|^2}{2\mathring{\rho}^2} - \frac{|\e \tilde{\mathbf{m}}|^2}{2\tilde{\rho}^2} = \frac{\e \tilde{\mathbf{m}}}{\tilde{\rho}^2} \cdot (\e \mathring{\mathbf{m}} - \e \tilde{\mathbf{m}}) + \frac{|\e \tilde{\mathbf{m}}|^2}{\tilde{\rho}^3} (\mathcal{\rho} - \tilde{\rho}) + O (|\e \mathring{\mathbf{m}} - \e \tilde{\mathbf{m}}|^2 + |\mathring{\rho} - \tilde{\rho}|), 
		\end{align*}
		it is straightforward to verify that
		\begin{align*}
			\left|\p_{y_1} \left(\frac{|\e \mathring{\mathbf{m}}|^2}{2\mathring{\rho}^2} - \frac{|\e \tilde{\mathbf{m}}|^2}{2\tilde{\rho}^2}\right)\right| \le& C [(\delta + \e)D_{-1}^1 + \sqrt{\eta} \e^{\frac{1}{2}} D_{-1}^2] |\p_{y_1} V| + C [(\delta+\e)D_{-\frac{1}{2}}^1 + \sqrt{\eta} \e^{\frac{1}{2}} D_{-\frac{1}{2}}^2] |\p_{y_1}^2 V| + C |\p_{y_1}^2 V| |\p_{y_1} V|.
		\end{align*}
		On the other hand, one has
		\begin{align*}
			\left|\kappa \p_{y_1} \left(\frac{|\e \tilde{\mathbf{u}}|^2 \mathring{\phi}}{2\mathring{\rho}}\right)\right| \le& C [(\delta + \e) D_{-1}^1 + \sqrt{\eta} \e^{\frac{1}{2}} D_{-1}^2] |\p_{y_1}^2 V| + C [(\delta + \e) D_{-\frac{3}{2}}^1 + \sqrt{\eta} \e^{\frac{1}{2}} D_{-\frac{3}{2}}^2] |\p_{y_1} V| \\
			&+ C [(\delta + \e) D_{-\frac{1}{2}}^1 + \sqrt{\eta} \e^{\frac{1}{2}} D_{-\frac{1}{2}}^2] |\p_{y_1} V|^2.
		\end{align*}
		Combining the above inequalities, we find
		\begin{align*}
			|\mathcal{G}_2| \le& C (\mathcal{D}^{(1)} + \mathcal{L}^{(1)}).
		\end{align*}
		As for \(\mathcal{G}_3\), the key term can be estimated as follows:
		\begin{align*}
			\p_{y_1} \left(\frac{\mathring{w}}{\mathring{\rho}}\right) - \frac{\p_{y_1}^2 W}{\tilde{\rho}} =& \frac{\p_{y_1}^2 W}{\mathring{\rho}} - \frac{\p_{y_1}^2 W}{\tilde{\rho}} - \frac{\p_{y_1}W (\p_{y_1} \mathring{\rho} - \p_{y_1} \tilde{\rho} + \p_{y_1} \tilde{\rho})}{\mathring{\rho}^2} \\
			\le& C|\p_{y_1}^2 V| |\p_{y_1}V| + C[(\delta+\e) D_{-\frac{1}{2}}^1 + \sqrt{\eta} \e^{\frac{1}{2}} D_{-\frac{1}{2}}^2] |\p_{y_1}V|,
		\end{align*}
		which leads to 
		\begin{align*}
			\mathcal{G}_3 \le C (\mathcal{D}^{(1)} + \mathcal{L}^{(1)}).
		\end{align*}
		Then, we have
		\begin{align*}
			Y_3^{(2)} \le C (\mathcal{D}^{(1)} + \mathcal{L}^{(1)} + \mathbf{D}_0\mathcal{I}^{(0)} + \mathbf{D}_0 \mathcal{I}^{(1)}).
		\end{align*}
		For \(Y_3^{(1)}\), it holds that
		\begin{align*}
			\mathcal{C}_{wd} =& -\tilde{\T} \p_{y_1} (\e \tilde{u}_1 \Phi) + \tilde{\T} \mathcal{S}_1 + \e \tilde{\mathbf{u}} \cdot \p_\tau \tilde{\Psi} + \p_\tau(\e \tilde{\mathbf{u}}) \cdot \tilde{\Psi} + \p_\tau \tilde{\T} \Phi \\
			\le&  C(\delta + \e) D_{-\frac{3}{2}}^1 + C\sqrt{\eta} \e^{\frac{1}{2}} D_{-1}^2 + C (\mathcal{D}^{(0)}+\mathcal{D}^{(1)}+\mathcal{L}^{(1)} + \mathbf{D}_0\mathcal{I}^{(0)} + \mathbf{D}_0\mathcal{I}^{(1)}).
		\end{align*}
		The estimates for \(\partial_{y_1}Y_3^{(j)}\), \(j=1,2\),
		follow similarly by differentiating the above expressions
		and applying Leibniz's rule.
		The remaining terms can be treated in the same manner.
		Then we have proved Lemma \ref{J est}.
	\end{proof}

	Recall \(\tilde{V} = (\tilde{\Phi}, \tilde{\Psi}_1, \tilde{W})^t\), then the equation of \(\tilde{V}\) can be written as 
	\begin{align}\label{tilde V equ}
		\p_\tau \tilde{V} + A_1 \p_{y_1} \tilde{V} = A_2 \p_{y_1}^2 \tilde{V} + A_3,
	\end{align}
	where
	\begin{align*}
		A_1 := \begin{pmatrix} 
			0 & \tilde{\theta} & 0 \\ 
			1 & 0 & 1 \\ 
			0 & \tilde{\theta} & 0 
		\end{pmatrix}, \quad
		A_2 := \begin{pmatrix} 
			0 & 0 & 0 \\ 
			0 & \frac{\lambda + 2\mu}{\tilde{\rho}} & 0 \\ 
			0 & 0 & \frac{\kappa}{\tilde{\rho}} 
		\end{pmatrix}, \quad
		A_3 := \begin{pmatrix}
			A_{31} \\
			A_{32} \\
			A_{33}
		\end{pmatrix} = \begin{pmatrix}
			Y_1 + \tilde{\T} \mathcal{S}_1 \\
			Y_{21} + \mathcal{S}_{21} \\
			Y_3 + \mathcal{S}_3
		\end{pmatrix}.
	\end{align*}
	It follows from a direct computation that the eigenvalues of the matrix \(A_1\) are \(\tilde{\lambda}_1\), \(0\), \(\tilde{\lambda}_3\), with \(\tilde{\lambda}_1 = -\tilde{\lambda}_3 = - \sqrt{2\tilde{\T}}\). The corresponding normalized left and right eigenvectors can be chosen as
	\begin{align*}
		\tilde{L} := \begin{pmatrix}
			\frac{1}{2} & \frac{1}{2} \tilde{\lambda}_1& \frac{1}{2} \\
			\frac{1}{\sqrt{2}} & 0 & -\frac{1}{\sqrt{2}} \\
			\frac{1}{2} & -\frac{1}{2} \tilde{\lambda}_1& \frac{1}{2}
		\end{pmatrix}, \qquad
		\tilde{R} := \begin{pmatrix}
			\frac{1}{2} & \frac{1}{\sqrt{2}} & \frac{1}{2} \\
			\frac{\tilde{\lambda}_1}{2\tilde{\T}} & 0 & -\frac{\tilde{\lambda}_1}{2\tilde{\T}} \\
			\frac{1}{2} & -\frac{1}{\sqrt{2}}& \frac{1}{2}
		\end{pmatrix}.
	\end{align*}
	Set \(B = \tilde{L} \tilde{V} = (b_1, b_2, b_3)^t\), then we have \(\tilde{V} = \tilde{R} B\).
	It follows from (\ref{tilde V equ}) that 
	\begin{align}\label{B equ}
		\p_\tau B + \Lambda \p_{y_1} B = \tilde{L} A_2 \tilde{R} \p_{y_1}^2 B + 2 \tilde{L} A_2 \p_{y_1} \tilde{R} \p_{y_1} B + \left[(\p_\tau \tilde{L} + \Lambda \p_{y_1} \tilde{L}) \tilde{R} +\tilde{L} A_2 \p_{y_1}^2 \tilde{R}\right] B + \tilde{L} A_3,
	\end{align}
	where \(\Lambda = \text{diag} \{\tilde{\lambda}_1, 0, \tilde{\lambda}_3\}\).
	For \(i = 0,1,2\), we denote
	\begin{align}
		K_i :=& \int_{\R} \left|\p_{y_1}^{i+1} \tilde{\Phi}\right|^2 + \left|\p_{y_1}^{i+1} \tilde{\Psi}\right|^2 + \left|\p_{y_1}^{i+1} \tilde{W}\right|^2 dy_1, \label{K_i} \\
		E_i :=& \tilde{E}_i + \int_{\R} \left|\p_{y_1}^{i} \tilde{\Psi}_2\right|^2 + \left|\p_{y_1}^{i} \tilde{\Psi}_3\right|^2 dy_1 + \tilde{c} \int_{\R} \left(\p_{y_1}^{i+1} \tilde{\Phi} \p_{y_1}^i \tilde{\Psi}_1 + \frac{\lambda + 2\mu}{2\tilde{\rho} \tilde{\T}} |\p_{y_1}^{i+1} \tilde{\Phi}|^2 - (\lambda+2\mu) \frac{\p_{y_1} \Phi |\p_{y_1}^{i+1} \tilde{\Phi}|^2}{2\mathring{\rho} \tilde{\rho}}\right), \label{E_i}
	\end{align}
	where \(\tilde{E}_i\) is defined later, and \(\tilde{c}\) is a sufficiently small constant.
	For convenience, we introduce the following notation. For any functions \(a\) and \(b\), we write \(a \approx b\) when
	\begin{align*}
		\|a-b\|_{L^2(\Omega_\e)}^2 \le C  \e^{-2N-3} (\e^{-2} + \tau)^{-\frac{N}{2}} e^{-c_0\e^{-2}}, \quad \forall N \in \mathbb{Z}^+.
	\end{align*}
	Then, it follows from (\ref{a priori assumption}) and (\ref{tilde Phi, Psi, W}) that
	\begin{align}
		\mathring{\phi} &= \partial_{y_1} \Phi = \frac{\tilde{\theta} \partial_{y_1} \tilde{\Phi} - \partial_{y_1} \tilde{\theta} \tilde{\Phi}}{\tilde{\theta}^2} \implies \left\| \partial_{y_1} \tilde{\Phi} \right\|_{L^2}^2 \leq C \e \chi^2 (1 + \e^2 \tau)^{-\frac{1}{2}}, \notag\\
		\mathring{\psi}_i &\approx \frac{1}{\mathring{\rho}} \partial_{y_1} \tilde{\Psi}_i + \frac{1}{\mathring{\rho}} \mathcal{C}_{ir} - \frac{\tilde{m}_i}{\mathring{\rho} \tilde{\rho}} \partial_{y_1} \Phi \implies \left\| \partial_{y_1} \tilde{\Psi}_i \right\|_{L^2}^2 \leq C \e \chi^2 (1 + \e^2 \tau)^{-\frac{1}{2}}, \\
		\mathring{\zeta} &\approx \frac{1}{\mathring{\rho}} \partial_{y_1} \tilde{W} + \frac{1}{\mathring{\rho}} \partial_{y_1} \tilde{\Phi} + \frac{1}{\mathring{\rho}} \mathcal{C}_{wr} - \frac{\tilde{\E}}{\mathring{\rho} \tilde{\rho}} \partial_{y_1} \Phi - \left( \frac{|\mathring{\mathbf{m}}|^2}{\mathring{\rho}^2} - \frac{|\tilde{\mathbf{m}}|^2}{\tilde{\rho}^2} \right) \implies \left\| \partial_{y_1} \tilde{W} \right\|_{L^2}^2 \leq C \e \chi^2 (1 + \e^2 \tau)^{-\frac{1}{2}}. \notag
	\end{align}
	Similarly, we have
	\begin{align}
		\left\| \partial_{y_1} (\mathring{\phi}, \mathring{\psi}_i, \mathring{\zeta}) \right\|_{L^2(\R)}^2 \leq C \e^3 \chi^2 (1+\e^2 \tau)^{-\frac{3}{2}}
		\implies \left\| \partial_{y_1}^2 V \right\|_{L^2(\R)}^2 \leq C \e^3 \chi^2 (1+\e^2 \tau)^{-\frac{3}{2}}.
	\end{align}
	
	For convenience, we set
	\begin{align}\label{bar delta and eta}
		\bar{\delta} = \delta + \e, \qquad \bar{\eta} = \chi + \bar{\delta}^{\frac{1}{2}} + \sqrt{\eta}.
	\end{align}
	Next, we will prove the following lemma.
	\begin{Lem}\label{main lemma}
		Under the same assumptions as Proposition \ref{a priori estimates}, it holds that
		\begin{align}
			&\frac{d}{d \tau} \left(\sum_{i=0}^{2} E_i\right) + \sum_{i=0}^{2} (K_i + G_i) \le C \bar{\eta} (\e^{-2} + \tau)^{-1}  \left(\sum_{i=0}^{2} E_i\right) + C \bar{\delta} (\e^{-2} + \tau)^{-1} +  C \sqrt{\eta} \e (\e^{-2} + \tau)^{-\frac{1}{2}}, \label{E_0 est}\\
			&\frac{d}{d \tau} \left(\sum_{i=1}^{2} E_i\right) + \sum_{i=1}^{2} (K_i + G_i) \le C \bar{\eta} (\e^{-2} + \tau)^{-1}  \left(\sum_{i=1}^{2} E_i + G_0\right) + C \bar{\eta} (\e^{-2} + \tau)^{-2} E_0 \notag \\
			&\qquad \qquad \qquad \qquad \qquad \qquad \qquad+ C \bar{\delta}(\e^{-2} + \tau)^{-2} +  C \sqrt{\eta} \e (\e^{-2} + \tau)^{-\frac{3}{2}}, \label{E_1 est}\\
			&\frac{d}{d \tau} E_2 + K_2 + G_2 \le C \bar{\eta} \sum_{i=0}^{2} (\e^{-2} + \tau)^{-(3-i)} E_i + C \bar{\eta} [(\e^{-2} + \tau) ^{-2} G_0 + (\e^{-2} + \tau) ^{-1} G_1] \notag\\
			&\qquad \qquad \qquad \qquad \qquad + C \bar{\delta} (\e^{-2} + \tau)^{-3} +  C \sqrt{\eta} \e (\e^{-2} + \tau)^{-\frac{5}{2}}, \label{E_2 est}
		\end{align}
		where \(G_i, i=1,2,3\) are defined later.
	\end{Lem}
	\begin{proof}
		Applying \(\p_{y_1}^k\), \(k = 0,1,2\) to (\ref{B equ}), we obtain
		\begin{align}\label{p^k B equ}
			\p_{y_1}^k B_{\tau} + \Lambda \p_{y_1}^{k+1} B  = \tilde{L} A_2 \tilde{R} \p_{y_1}^{k+2} B + \hat{M}_k,
		\end{align}
		where
		\begin{align*}
			\hat{M}_k := &\, [\p_{y_1}^k, \Lambda] \p_{y_1} B + [\p_{y_1}^k, \tilde{L} A_2 \tilde{R}]\p_{y_1}^2 B + 2 \sum_{i=0}^k \partial_{y_1}^i (\tilde{L} A_2 \partial_{y_1} \tilde{R}) \partial_{y_1}^{k-i+1} B  \\
			&+ \sum_{i=0}^k \left\{ \partial_{y_1}^i \left[ (\tilde{L}_\tau + \Lambda \partial_{y_1} \tilde{L}) \tilde{R} + \tilde{L} A_2 \ \p_{y_1}^2 \tilde{R} \right] \partial_{y_1}^{k-i} B \right\} + \partial_{y_1}^k (\tilde{L} A_3) \\
			\leq &\, C \sum_{j=1}^{k+1} \left( \left| \partial_{y_1}^j \bar{\rho} \right| \left| \partial_{y_1}^{k-j+1} b_1 \right|, 0, \left| \partial_{y_1}^j \bar{\rho} \right| \left| \partial_{y_1}^{k-j+1} b_3 \right| \right)^t + C \sum_{j=1}^{k+2} \left| (\bar{\delta} D_{-\frac{j}{2}}^1 + \eta \e^{\frac{1}{2}} D_{-\frac{j}{2}}^2) \partial_{y_1}^{k-j+2} B \right| + \partial_{y_1}^k (\tilde{L} A_3) \\
			:=& \sum_{i=1}^3 \hat{M}_k^{(i)}.
		\end{align*}
		
		\subsubsection*{Step 1. The diagonalized system.}

		\(~\)
		
		Without loss of generality, we assume that \(\p_{y_1} \bar{\rho} > 0\). The case that \(\p_{y_1} \bar{\rho} < 0\) can be treated in a similar way. Let \(v_1 = \frac{\bar{\rho}}{\rho_+}\), then we find that \(|v_1 - 1| \le C\delta\). Set \(\bar{B}^{(k)} = (v_1^n \p_{y_1}^k b_1, \p_{y_1}^k b_2, v_1^{-n} \p_{y_1}^k b_3)\), where \(n = 4 [\delta^{-\frac{1}{2}}] + 1\). Multiplying (\ref{p^k B equ}) by \(\bar{B}^{(k)}\) and integrating with respect to \(y_1\), one has
		\begin{align}\label{eq;diag}
			&\int_{\R} \left(\frac{v_1^n}{2} \left|\p_{y_1}^k b_1\right|^2 + \frac{1}{2} \left|\p_{y_1}^k b_2\right|^2 + \frac{v_1^{-n}}{2} \left|\p_{y_1}^k b_3\right|^2\right)_{\tau} + \p_{y_1} \bar{B}^{(k)} A_4 \p_{y_1}^{k+1} B dy_1 + \int_{\R} a_1 |\p_{y_1}^k b_1|^2 + a_2 |\p_{y_1}^k b_3|^2 dy_1 \notag \\
			=& \int_{\R} \bar{B}^{(k)} \p_{y_1}A_4 \p_{y_1}^{k+1} B + \left[\left(\frac{v_1^n}{2}\right)_{\tau} \left|\p_{y_1}^k b_1\right|^2 + \left(\frac{v_1^{-n}}{2}\right)_{\tau} \left|\p_{y_1}^k b_3\right|^2\right] + \bar{B}^{(k)} \hat{M}_k dy_1 \notag \\
			:=& I_1 +I_2 + I_3,
		\end{align}
		where 
		\begin{align*}
			&a_1 := - \frac{v_1^{n-1}}{2}\left(n \tilde{\lambda}_1 \p_{y_1}v_1 + v_1\p_{y_1}\tilde{\lambda}_1\right) \ge C \delta^{-\frac{1}{2}} \p_{y_1}\bar{\rho} - C\left(\bar{\delta} D_{-1}^1 + \eta \e^{\frac{1}{2}} D_{-1}^2\right), \\
			&a_2 := \frac{v_1^{-(n+1)}}{2}\left(n \tilde{\lambda}_3 \p_{y_1}v_1 - v_1\p_{y_1}\tilde{\lambda}_3\right) \ge C \delta^{-\frac{1}{2}} \p_{y_1}\bar{\rho} - C\left(\bar{\delta} D_{-1}^1 + \eta \e^{\frac{1}{2}} D_{-1}^2\right).
		\end{align*}
		and \(A_4 = \tilde{L} A_2 \tilde{R}\) is a  symmetric nonnegative matrix.
		In fact, we have
		\begin{align*}
			A_4 = \begin{pmatrix}
				\frac{\tilde{\mu}}{2} + \frac{\tilde{\kappa}}{4} & a_3 & -\frac{\tilde{\mu}}{2} + \frac{\tilde{\kappa}}{4} \\
				a_3 & \frac{\tilde{\kappa}}{2} & a_3 \\
				-\frac{\tilde{\mu}}{2} + \frac{\tilde{\kappa}}{4} & a_3 & \frac{\tilde{\mu}}{2} + \frac{\tilde{\kappa}}{4}
			\end{pmatrix},
		\end{align*}
		where \(\tilde{\mu} = \frac{\lambda + 2\mu}{\tilde{\rho}}\), \(\tilde{\kappa} = \frac{\kappa}{\tilde{\rho}}\) and \(a_3 = -\sqrt{\frac{1}{2}} \frac{\tilde{\kappa}}{2}\).
		For convenience, we denote energy and two kinds of dissipation norm by $\tilde{E}_k,\tilde{H}_k ,G_k$, respectively
		\begin{align}
			&\tilde{E}_k := \int_{\R} \left(\frac{v_1^n}{2} \left|\p_{y_1}^k b_1\right|^2 + \frac{1}{2} \left|\p_{y_1}^k b_2\right|^2 + \frac{v_1^{-n}}{2} \left|\p_{y_1}^k b_3\right|^2\right) dy_1,  \label{tilde E_k} \\
			&\tilde{H}_k := \int_{\R} \p_{y_1}^{k+1} B^t A_4 \p_{y_1}^{k+1} B dy_1, \qquad G_k := C\delta^{-\frac{1}{2}} \int_{\R} \p_{y_1}\bar{\rho} \left|\p_{y_1}^k b_1\right|^2 + \p_{y_1}\bar{\rho} \left|\p_{y_1}^k b_3\right|^2 dy_1.\label{2026-9-20-2}
		\end{align}
		Indeed, for \(k=0\), \(\tilde{E}_0 = C \|\tilde{V}\|_{L^2(\R)}^2\). For \(k \ge 1\), it is straightforward to check that
		\begin{align}\label{tilde E_k norm}
			C \left( \left\| \partial_{y_1}^k \tilde{V} \right\|_{L^2}^2 - \bar{\eta} \sum_{j=0}^{k-1} (\e^{-2}+\tau)^{-(k-j)} \left\| \partial_{y_1}^j \tilde{V} \right\|_{L^2}^2 \right) 
			\leq \tilde{E}_k 
			\leq C \left( \left\| \partial_{y_1}^k \tilde{V} \right\|_{L^2}^2 + \bar{\eta} \sum_{j=0}^{k-1} (\e^{-2} + \tau)^{-(k-j)} \left\| \partial_{y_1}^j \tilde{V} \right\|_{L^2}^2 \right).
		\end{align}
		For the dissipation norm $\tilde{H}_k$, a straightforward calculation gives that
		\begin{align}
			\p_{y_1}^{k+1} B^t A_4 \p_{y_1}^{k+1} B &= \tilde{\kappa} \left[\frac{1}{2}\left(\p_{y_1}^{k+1} b_1 + \p_{y_1}^{k+1} b_3 \right) - \frac{1}{\sqrt{2}} \p_{y_1}^{k+1} b_2\right]^2 + \tilde{\mu} \left[\frac{1}{\sqrt{2}} \left(\p_{y_1}^{k+1} b_3 - \p_{y_1}^{k+1} b_1\right)\right]^2 \notag \\
			&\ge C \left(\left|\p_{y_1}^{k+1} \tilde{\Psi}_1\right|^2 + \left|\p_{y_1}^{k+1} \tilde{W}\right|^2\right) - C \sum_{j=1}^{k+1} \left(\bar{\delta} D_{-j}^1 + \eta \e^{\frac{1}{2}} D_{-j}^2\right) \left|\p_{y_1}^{k+1-j} \tilde{\Psi}_1\right|^2.
		\end{align}
		Thus, we have
		\begin{align}
			\int_{\R} \p_{y_1} \left(\bar{B}^{(k)} -  \p_{y_1}^k B^t\right) A_4 \p_{y_1}^{(k+1)} B dy_1 \le C (\chi + \bar{\delta}^{\frac{1}{2}}) (\e^{-2} + \tau)^{-1} E_k + C (\chi + \bar{\delta}^{\frac{1}{2}}) K_k.
		\end{align}
		Now, we estimate the right side of \eqref{eq;diag}. A direct calculation yields
		\begin{align}\label{I_1+I_2}
			I_1 +I_2 =& \int_{\R} \bar{B}^{(k)} \p_{y_1}A_4 \p_{y_1}^k \p_{y_1} B dy_1 + \int_{\R} \left[\left(\frac{v_1^n}{2}\right)_{\tau} \left|\p_{y_1}^k b_1\right|^2 + \left(\frac{v_1^{-n}}{2}\right)_{\tau} \left|\p_{y_1}^k b_3\right|^2\right] dy_1 \notag \\
			\le&C (\chi + \bar{\delta}^{\frac{1}{2}}) (\e^{-2} + \tau)^{-1} E_k + C (\chi + \bar{\delta}^{\frac{1}{2}}) K_k.
		\end{align}
		Next, we will estimate \(I_3\) 
		\begin{align*}
			&\int_{\R} \bar{B}^{(k)} \hat{M}_k^{(2)} \le C \sum_{j=1}^{k+2}\left(\bar{\delta} \norm{D_{-\frac{j}{2}}^1 \p_{y_1}^{k-j+2} B}_{L^2(\R)} + \eta \e^{\frac{1}{2}} \norm{D_{-\frac{j}{2}}^2 \p_{y_1}^{k-j+2} B}_{L^2(\R)} \right) \norm{\bar{B}^{(k)}}_{L^2(\R)} \\
			\le& C \bar{\delta} \sum_{j=1}^{k+2} D_{-\frac{j}{2}}^1 \norm{\bar{B}^{(k)}}_{L^2(\R)} \norm{\p_{y_1}^{k-j+2} B}_{L^2(\R)} + C\sqrt{\eta}\e^{\frac{1}{2}} \sum_{j=1}^{k+2} D_{-\frac{j}{2}}^2 \norm{\bar{B}^{(k)}}_{L^2(\R)} \norm{\p_{y_1}^{k-j+2} B}_{L^2(\R)} \\
			\le& C \bar{\eta} \sum_{j=0}^{k+1} (\e^{-2} + \tau)^{-j} \norm{\p_{y_1}^{k-j+1} B}_{L^2(\R)}^2.
		\end{align*}
		For \(\int_{\R} \bar{B}^{(k)} \hat{M}_k^{(1)} dy_1\), we consider the case that \(k=0\) first
		\begin{align}
			\int_{\R} \bar{B}^{(0)} \hat{M}_0^{(1)} dy_1 = \int_{\R} \bar{B}^{(0)} \sum_{j=1}^{k+1}\left(|\p_{y_1} \bar{\rho}| |b_1|, 0, |\p_{y_1} \bar{\rho}| |b_3|\right)^t dy_1 \le C \int_{\R} \p_{y_1} \bar{\rho} (|b_1|^2 + |b_3|^2) dy_1.
		\end{align}
		As \(1 \le k \le 2\), notice that
		\begin{align*}
			\left|\p_{y_1}^2 \bar{\rho}\right| \lesssim \left[\frac{|y_1|}{\e^{-2} + \tau} +\delta (\e^{-2} + \tau)^{-\frac12}\right]\p_{y_1} \bar{\rho}, \qquad 	\left|\p_{y_1}^3 \bar{\rho}\right| \lesssim \frac{|y_1|^2}{(\e^{-2} + \tau)^2} \p_{y_1} \bar{\rho} + (\e^{-2} + \tau)^{-1} \p_{y_1} \bar{\rho}.
		\end{align*}
		We have
		\begin{align}
			\int_{\R} \bar{B}^{(1)} \hat{M}_1^{(1)} =& \int_{\R} \left(v_1^n \p_{y_1} b_1, \p_{y_1} b_2, v_1^{-n} \p_{y_1}b_3\right) \begin{pmatrix}
				|\p_{y_1} \bar{\rho}| |\p_{y_1} b_1| \\
				0\\
				|\p_{y_1} \bar{\rho}| |\p_{y_1} b_3|
			\end{pmatrix} dy_1 \notag\\
			&+ \int_{\R} \left(v_1^n \p_{y_1} b_1, \p_{y_1} b_2, v_1^{-n} \p_{y_1}b_3\right) \begin{pmatrix}
				|\p_{y_1}^2 \bar{\rho}| |b_1| \\
				0\\
				|\p_{y_1}^2 \bar{\rho}| |b_3|
			\end{pmatrix} dy_1 \notag\\
			\le& \sum_{i=1,3}\int_{\R} C \frac{|y_1|}{\e^{-2} + \tau} \p_{y_1} \bar{\rho} |b_i| |\p_{y_1} b_i| dy_1 + \sum_{i=1,3} \int_{\R} C \p_{y_1} \bar{\rho} |\p_{y_1} b_i|^2 dy_1 \label{int of B^1 M_1^1}\\
			\le& \sum_{i=1,3} C \left( \int_{\R} \frac{|y_1|}{\e^{-2} + \tau} \p_{y_1} \bar{\rho} |\p_{y_1}^2 b_i|^2 dy_1 +  \int_{\R} \frac{1}{\e^{-2} + \tau} \p_{y_1} \bar{\rho} | b_i|^2 dy_1\right) \notag\\
			&+ \sum_{i=1,3} \int_{\R} C \p_{y_1} \bar{\rho} |b_i| |\p_{y_1}^2 b_i|^2 dy_1\notag\\
			\le& C\bar{\delta} \sum_{i=1,3} \int_{\R} \tilde D_{-\frac12}^1  \left|\p_{y_1}^2 b_i\right|^2 dy_1 + \sum_{j=0}^{1} \sum_{i=1,3} C (\e^{-2} + \tau)^{-j} \int_{\R} \p_{y_1}\bar{\rho} \left|\p_{y_1}^{1-j} b_i\right|^2 dy_1, \notag\\
			\int_{\R} \bar{B}^{(2)} \hat{M}_2^{(1)} \le& \sum_{i=1,3} \int_{\R} C \frac{|y_1|^2}{(\e^{-2} + \tau)^2} \p_{y_1} \bar{\rho} |b_i| |\p_{y_1}^2 b_i| dy_1 + \sum_{i=1,3} \int_{\R} C \frac{|y_1|}{\e^{-2} + \tau} \p_{y_1} \bar{\rho} |\p_{y_1} b_i| |\p_{y_1}^2 b_i| dy_1 \notag\\
			&+ C(\e^{-2} + \tau)^{-1} \sum_{i=1,3} \int_{\R} \p_{y_1} \bar{\rho} |b_i| |\p_{y_1}^2 b_i| dy_1 + \sum_{i=1,3} C \int_{\R} \p_{y_1} \bar{\rho} |\p_{y_1}^2 b_i|^2 dy_1 \label{int of B^2 M_2^1}\\
			\le& C\bar{\delta} \sum_{i=1,3} \int_{\R} \tilde D_{-\frac12}^1  \left|\p_{y_1}^2 b_i\right|^2 dy_1 + \sum_{j=0}^{2} \sum_{i=1,3} C (\e^{-2} + \tau)^{-j} \int_{\R} \p_{y_1}\bar{\rho} \left|\p_{y_1}^{2-j} b_i\right|^2 dy_1, \notag
		\end{align}
		where we have used the fact that
		\begin{align*}
			\int_{\R} \frac{|y_1|}{\e^{-2} + \tau} \p_{y_1} \bar{\rho} |b_i| |\p_{y_1} b_i| dy_1 \lesssim&  \int_{\R} \frac{|y_1|^2}{(\e^{-2} + \tau)^2} \p_{y_1} \bar{\rho} |\p_{y_1} b_i|^2 dy_1 + \int_{\R} \frac{1}{\e^{-2} + \tau} \p_{y_1}\bar{\rho} |b_i|^2 dy_1 \\
			\lesssim& \bar{\delta} \int_{\R} \tilde D_{-\frac12}^1  \left|\p_{y_1} b_i\right|^2 dy_1 + (\e^{-2} + \tau)^{-1} \int_{\R} \p_{y_1}\bar{\rho} \left|b_i\right|^2 dy_1, \\
			\int_{\R} \frac{|y_1|^2}{(\e^{-2} + \tau)^2} \p_{y_1} \bar{\rho} |b_i| |\p_{y_1}^2 b_i| dy_1 \lesssim& \int_{\R} \frac{|y_1|^4}{(\e^{-2} + \tau)^2} \p_{y_1} \bar{\rho} |\p_{y_1}^2 b_i|^2 dy_1 + \int_{\R} \frac{1}{(\e^{-2} + \tau)^2} \p_{y_1} \bar{\rho} |b_i|^2 dy_1 \\
			\lesssim& \bar{\delta} \int_{\R} \tilde D_{-\frac12}^1  \left|\p_{y_1}^2 b_i\right|^2 dy_1 + (\e^{-2} + \tau)^{-2} \int_{\R} \p_{y_1} \bar{\rho} \left|b_i\right|^2 dy_1.
		\end{align*}
		Then, we will estimate terms involving \(\hat{M}_k^{(3)}\).
		For \(k=0\), it holds that
		\begin{align*}
			&\norm{\bar{B}^{(0)} \hat{M}_0^{(3)}}_{L^1(\R)} = \int_{\R} \left(v_1^n b_1, b_2, v_1^{-n} b_3\right) \left(\tilde{L} A_3\right) dy_1 \\
			\le& C \int_{\R} \left|(Y_1, Y_{21}, Y_3, \mathcal{S}_1, \mathcal{S}_{21}, \mathcal{S}_3)\right| \left|(b_1, b_3)\right|dy_1 +  C \int_{\R} \left|(Y_1, Y_3, \mathcal{S}_1, \mathcal{S}_3)\right| |b_2| dy_1\\
			\le& C \int_{\R} (\bar{\delta} D_{-1}^1 + \eta \e^{\frac{1}{2}} D_{-1}^2) \left|(b_1, b_3)\right| dy_1 + C \int_{\R} \sum_{l=0}^{1} \left( \mathcal{D}^{(l)} + \mathcal{L}^{(l)} +  \mathbf{D}_0\mathcal{I}^{(l)}\right) \left|(b_1, b_2, b_3)\right| dy_1 + C \int_{\R} (\bar{\delta} D^1_{-\frac{3}{2}}  + \eta \e^{\frac{1}{2}} D^2_{-1}) |b_2| dy_1 \\
			:=& \sum_{j=1}^{10} I_4^{(j)}.
		\end{align*}
		By a straightforward calculation, one observes that
		\begin{align}
			I_4^{(1)} \le& C\bar{\delta} \int_{\R} D_{-1}^1 \left|(b_1, b_3)\right| dy_1 \le \bar{\delta} \int_{\R} \p_{y_1} \bar{\rho} |(b_1, b_3)|^2 dy_1 + C \bar{\delta} \left(\e^{-2} + \tau\right)^{-1},  \label{I_4^1}\\
			I_4^{(2)} \le& C \sqrt{\eta} \e^{\frac{1}{2}} \int_{\R} D_{-1}^2 \left|(b_1, b_3)\right| dy_1 \le \sqrt{\eta}  (\e^{-2} + \tau)^{-1} \norm{V}_{L^2(\R)}^2 + \sqrt{\eta} \e \left(\e^{-2} + \tau\right)^{-\frac{1}{2}}, \label{I_4^2}\\
			I_4^{(3)} \le& C \int_{\R} \mathcal{D}^{(0)} \left|\bar{B}^{(0)}\right| dy_1 \le C \bar{\delta} \int_{\R} \left(D_{-\frac{1}{2}} \left|\p_{y_1} V\right| + D_{-1} \left|V\right|\right) \left|\bar{B}^{(0)}\right| dy_1 \notag \\
			\le& C\bar{\delta} \left(\e^{-2} + \tau\right)^{-1} \norm{V}_{L^2(\R)}^2 + C \bar{\delta} \norm{\p_{y_1} V}_{L^2(\R)}^2,\label{I_4^3} \\
			I_4^{(4)} \le& C \int_{\R} \mathcal{D}^{(1)} \left|\bar{B}^{(0)}\right| dy_1 \le C \bar{\delta} \int_{\R} \left(D _{-\frac{1}{2}} \left|\p_{y_1}^2 V\right| + D_{-1} \left|\p_{y_1}V\right| + D_{-\frac{3}{2}} \left|V\right|\right)\left|\bar{B}^{(0)}\right| dy_1 \notag \\
			\le& C\bar{\delta} \left(\e^{-2} + \tau\right)^{-1} \norm{V}_{L^2(\R)}^2 + C \bar{\delta} \norm{\p_{y_1} V}_{H^1(\R)}^2,\label{I_4^4} \\
			I_4^{(5)} \le& C \int_{\R} \mathcal{L}^{(0)} \left|\bar{B}^{(0)}\right| dy_1 \le C \int_{\R} \left(\left|\p_{y_1} V\right|^2 + \bar{\delta} D_{-\frac{1}{2}} \left|\p_{y_1}V\right|\right)\left|\bar{B}^{(0)}\right| dy_1 \notag \\
			\le& C\bar{\delta} \left(\e^{-2} + \tau\right)^{-1} \norm{V}_{L^2(\R)}^2 + C \left(\bar{\delta} + \chi\right) \norm{\p_{y_1} V}_{L^2(\R)}^2, \label{I_4^5}\\
			I_4^{(6)} \le& C \int_{\R} \mathcal{L}^{(1)} \left|\bar{B}^{(0)}\right| dy_1 \le C \int_{\R} \left(\left|\p_{y_1}^2 V\right| \left|\p_{y_1}V\right| + \bar{\delta}\left(D_{-\frac{1}{2}} \left|\p_{y_1}V\right|^2 + D_{-1} \left|\p_{y_1}V\right| + D_{\frac{1}{2}} \left|\p_{y_1}^2 V\right|\right)\right) \left|\bar{B}^{(0)}\right| dy_1 \notag \\
			\le& C\bar{\delta} \left(\e^{-2} + \tau\right)^{-1} \norm{V}_{L^2(\R)}^2 + C \left(\bar{\delta} + \chi\right) \norm{\p_{y_1} V}_{H^1(\R)}^2, \label{I_4^6}\\
			I_4^{(7)} \le& C \int_{\R} \mathbf{D}_0 \mathcal{I}^{(0)} \left|\bar{B}^{(0)}\right| dy_1
			\le C \e \left(\e^{-2} + \tau\right)^{-1} \norm{V}_{L^2(\R)}^2 + C \e^{-\frac{2N+3}{4} + 1} e^{-\frac{c_0\e^{-2}}{4}} (\e^{-2} + \tau)^{-\frac{N}{8} + 1}, \label{I_4^7}\\
			I_4^{(8)} \le& C \int_{\R}  \mathbf{D}_0 \mathcal{I}^{(1)} \left|\bar{B}^{(0)}\right| dy_1
			\le C \e \left(\e^{-2} + \tau\right)^{-1} \norm{V}_{L^2(\R)}^2 + C \e^{-\frac{2N+3}{4}+1} e^{-\frac{c_0\e^{-2}}{4}} (\e^{-2} + \tau)^{-\frac{N}{8} + 1}, \label{I_4^8} \\
			I_4^{(9)} \le& C \int_{\R} \bar{\delta} D^1_{-\frac{3}{2}} |b_2| dy_1 \le C \bar{\delta} \left(\e^{-2} + \tau\right)^{-1} \norm{V}_{L^2(\R)}^2 + C \bar{\delta} \left(\e^{-2} + \tau\right)^{-\frac{3}{2}},  \label{I_4^9} \\
			I_4^{(10)} \le& C \int_{\R} \sqrt{\eta} \e^{\frac{1}{2}} D^2_{-1} |b_2| dy_1 \le C \sqrt{\eta} \left(\e^{-2} + \tau\right)^{-1} \norm{V}_{L^2(\R)}^2 +\sqrt{\eta} \e \left(\e^{-2} + \tau\right)^{-\frac{1}{2}}. \label{I_4^10}
		\end{align}
		For \(k=1\), we have
		\begin{align}\label{int of B^1 M_1^3}
			&\int_{\R} \bar{B}^{(1)} \hat{M}_1^{(3)} dy_1 = \int_{\R} \bar{B}^{(1)} \p_{y_1} \left(\tilde{L} A_3\right) dy_1\notag \\
			\le& C \int_{\R} \sum_{l=1}^{2}\left( (\bar{\delta}  D_{-\frac{3}{2}}^1 + \sqrt{\eta} \e^{\frac{1}{2}} D_{-\frac{3}{2}}^2) \mathcal{D}^{(l)}+ \mathcal{L}^{(l)} +  \mathbf{D}_0\mathcal{I}^{(l)}\right) |\p_{y_1}(b_1, b_3)| dy_1 \notag\\
			&+ C\int_{\R}\bar{\delta} \left(\e^{-2} + \tau\right)^{-\frac{1}{2}} \sum_{l=0}^{1}\left[\mathcal{D}^{(l)} + \mathcal{L}^{(l)} +  \mathbf{D}_0 \mathcal{I}^{(l)} + \bar{\delta} D_{-1} \right] |\p_{y_1}(b_1, b_3)| dy_1 \notag\\
			&+ C\int_{\R} \sum_{l=0}^{1} \left(\bar{\delta} D^1_{-\frac{3}{2}}  + \sqrt{\eta} \e^{\frac{1}{2}} D^2_{-1} + \mathcal{D}^{(l)} + \mathcal{L}^{(l)} + \mathbf{D}_0 \mathcal{I}^{(l)}\right) \left|\p_{y_1}^2 b_2\right| dy_1.
		\end{align}
		For the error terms in (\ref{int of B^1 M_1^3}), we have
		\begin{align*}
			&\int_{\R} \bar{\delta} D_{-\frac{3}{2}}^1 |\p_{y_1}(b_1, b_3)| dy_1 \le \bar{\delta} \int_{\R} \p_{y_1}\bar{\rho} |\p_{y_1}(b_1, b_3)|^2 dy_1 + C \bar{\delta} \left(\e^{-2} + \tau\right)^{-2},\\
			&\int_{\R} \sqrt{\eta} \e^{\frac{1}{2}} D_{-\frac{3}{2}}^2 |\p_{y_1}(b_1, b_3)| dy_1 \le \sqrt{\eta} (\e^{-2} + \tau)^{-1} \norm{\p_{y_1} V}_{L^2(\R)}^2 + \sqrt{\eta} \e \left(\e^{-2} + \tau\right)^{-\frac{3}{2}},\\
			&\int_{\R} \bar{\delta} D^1_{-\frac{3}{2}}  \left|\p_{y_1}^2 b_2\right| dy_1 \le C\bar{\delta} \norm{\p_{y_1}^2 B}_{L^2(\R)}^2 + C \bar{\delta} \left(\e^{-2} + \tau\right)^{-\frac{5}{2}}, \\
			&\int_{\R} \sqrt{\eta} \e^{\frac{1}{2}} D_{-1}^2 \left|\p_{y_1}^2 b_2\right| dy_1 \le C \sqrt{\eta} \norm{\p_{y_1}^2 B}_{L^2(\R)}^2 + \sqrt{\eta} \e \left(\e^{-2} + \tau\right)^{-\frac{3}{2}},
		\end{align*}
		and the other terms in (\ref{int of B^1 M_1^3}) can be estimated in the same way as (\ref{I_4^2}) - (\ref{I_4^10}). Overall, it holds that
		\begin{align*}
			\int_{\R} \bar{B}^{(1)} \hat{M}_1^{(3)} dy_1 \le C \bar{\eta} \sum_{i=0}^{2} \left(\e^{-2} + \tau\right)^{-i} \norm{\p_{y_1}^{2-i} B}_{L^2(\R)}^2 + C \bar{\delta} \left(\e^{-2} + \tau\right)^{-2} + \sqrt{\eta} \e \left(\e^{-2} + \tau\right)^{-\frac{3}{2}}.
		\end{align*}
		For \(k=2\), we have
		\begin{equation}\label{int of B^2 M_2^3}
			\begin{aligned}
				&\int_{\R} \bar{B}^{(2)} \hat{M}_2^{(3)} dy_1 = \int_{\R} \bar{B}^{(2)} \p_{y_1}^2 \left(\tilde{L} A_3\right) dy_1  \\
				\le& C \int_{\R} (\bar{\delta} D_{-2}^1 + \sqrt{\eta} \e^{\frac{1}{2}} D_{-2}^2) |\p_{y_1}^2 (b_1, b_3)| dy_1 + C \int_{\R} (\bar{\delta} D^1_{-2}  + \sqrt{\eta} \e^{\frac{1}{2}} D^2_{-\frac{3}{2}}) |\p_{y_1}^3 b_2| dy_1 \\
				&+ C \int_{\R} \sum_{l=0}^{1} \left( \mathcal{D}^{(l)} + \mathcal{L}^{(l)} +  \mathbf{D}_0 \mathcal{I}^{(l)}\right) \left(\bar{\delta} D_{-\frac{1}{2}} \left|\p_{y_1}^3 B\right| + \bar{\delta}^{-\frac{3}{2}} D_{-1} \left|\p_{y_1}^2 B\right|\right) dy_1 \\
				&+ C \int_{\R} \sum_{l=1}^{2} \left( \mathcal{D}^{(l)} + \mathcal{L}^{(l)} +  \mathbf{D}_0 \mathcal{I}^{(l)}\right) \left(\left|\p_{y_1}^3 B\right| + \bar{\delta}^{\frac{1}{2}} D_{-\frac{1}{2}} \left|\p_{y_1}^2 B\right|\right)dy_1, \\
			\end{aligned}
		\end{equation}
		where the error terms satisfy
		\begin{align}
			&\int_{\R} \bar{\delta} D_{-2}^1 |\p_{y_1}^2 (b_1, b_3)| dy_1  \le \bar{\delta} \int_{\R} \p_{y_1}\bar{\rho} |\p_{y_1}^2 (b_1, b_3)|^2 dy_1 + C \bar{\delta} \left(\e^{-2} + \tau\right)^{-3},\notag\\
			&\int_{\R} \sqrt{\eta} \e^{\frac{1}{2}} D_{-2}^2 |\p_{y_1}^2 (b_1, b_3)| dy_1 \le C \sqrt{\eta} (\e^{-2} + \tau)^{-1} \norm{\p_{y_1}^2 V}_{L^2(\R)}^2 + \sqrt{\eta} \e \left(\e^{-2} + \tau\right)^{-\frac{5}{2}},\notag\\
			&\int_{\R} \bar{\delta} D^1_{-2}  \left|\p_{y_1}^3 b_2\right| dy_1 \le C\bar{\delta} \norm{\p_{y_1}^3 B}_{L^2(\R)}^2 + C \bar{\delta} \left(\e^{-2} + \tau\right)^{-\frac{7}{2}},\notag \\
			&\int_{\R} \sqrt{\eta} \e^{\frac{1}{2}} D_{-\frac{3}{2}}^2 \left|\p_{y_1}^3 b_2\right| dy_1 \le C \sqrt{\eta} \norm{\p_{y_1}^3 B}_{L^2(\R)}^2 + \sqrt{\eta} \e \left(\e^{-2} + \tau\right)^{-\frac{5}{2}}.
		\end{align}
		Next, we concentrate on the third-order derivative terms in (\ref{int of B^2 M_2^3}). It follows from (\ref{a priori assumption}) that
		\begin{align*}
			&\int_{\R} \mathcal{D}^{(1)} \left|\p_{y_1}^3 B\right| dy_1 = \int_{\R} \bar{\delta} \left(D_{-\frac{1}{2}}\left|\p_{y_1}^2 V\right| + D_{-1} \left|\p_{y_1} V\right| + D_{-\frac{3}{2}} \left|V\right|\right) \left|\p_{y_1}^3B\right| dy_1 \\
			\le&  C \bar{\delta} \norm{\p_{y_1}^3 B}_{L^2(\R)}^2 + C\bar{\delta} \left(\e^{-2} + \tau\right)^{-1} \norm{\p_{y_1}^2 V}_{L^2(\R)}^2 + C\bar{\delta} \left(\e^{-2} + \tau\right)^{-2} \norm{\p_{y_1} V}_{L^2(\R)}^2 + C\bar{\delta} \left(\e^{-2} + \tau\right)^{-3} \norm{V}_{L^2(\R)}^2 \\
			\le& C\bar{\delta} \sum_{j=0}^{3} \left(\e^{-2} + \tau\right)^{-j} \norm{\p_{y_1}^{3-j} V}_{L^2(\R)}^2, \\
			&\int_{\R} \mathcal{D}^{(2)} \left|\p_{y_1}^3 B\right| dy_1 = \int_{\R} \bar{\delta} \left(D_{-\frac{1}{2}} \left|\p_{y_1}^3V\right| + D_{-1} \left|\p_{y_1}^2 V\right| + D_{-\frac{3}{2}} \left|\p_{y_1} V\right| + D_{-2} \left|V\right|\right) \left|\p_{y_1}^3 B\right| dy_1 \\
			\le& C\bar{\delta} \sum_{j=0}^{3} \left(\e^{-2} + \tau\right)^{-j} \norm{\p_{y_1}^{3-j} V}_{L^2(\R)}^2, \\
			&\int_{\R} \mathcal{L}^{(1)} \left|\p_{y_1}^3 B\right|dy_1 \le \int_{\R} \left[\left|\p_{y_1}^2 V\right| \left|\p_{y_1} V\right| + C\bar{\delta} \left(D_{-\frac{1}{2}} \left|\p_{y_1} V\right|^2 + D_{-1} \left|\p_{y_1} V\right| + D_{\frac{1}{2}}\left|\p_{y_1}^2 V\right|\right)\right] \left|\p_{y_1}^3 B\right| dy_1 \\
			\le& C(\bar{\delta} + \chi) \sum_{j=0}^{3} \left(\e^{-2} + \tau\right)^{-j} \norm{\p_{y_1}^{3-j} V}_{L^2(\R)}^2, \\
			&\int_{\R} \mathcal{L}^{(2)} \left|\p_{y_1}^3 B\right|dy_1 \le \int_{\R} \left(\left|\p_{y_1}^3 V\right| \left|\p_{y_1} V\right| + \left|\p_{y_1}^2 V\right|^2 \right) \left|\p_{y_1}^3 B\right| + C \bar{\delta} D_{\frac{1}{2}} \left(\left|\p_{y_1} V\right| \left|\p_{y_1}^2 V\right| + \left|\p_{y_1}^3 V\right|\right) \left|\p_{y_1}^3 B\right| dy_1 \\
			&+ \int_{\R} \bar{\delta} D_{-1} \left(\left|\p_{y_1}^2 V\right| + \left|\p_{y_1} V\right|^2\right) \left|\p_{y_1}^3 B\right| + \bar{\delta} D_{-\frac{3}{2}} \left|\p_{y_1}V\right| \left|\p_{y_1}^3 B\right| dy_1 \\
			\le& C (\bar{\delta} + \chi) \sum_{j=0}^{3} \left(\e^{-2} + \tau\right)^{-j} \norm{\p_{y_1}^{3-j} V}_{L^2(\R)}^2 + C\norm{\p_{y_1}^2 V}_{L^4(\R)}^4,
		\end{align*}
		where we have used the fact that
		\begin{align*}
			\int_{\R} \left|\p_{y_1}^2 V\right| \left|\p_{y_1}V\right| \left|\p_{y_1}^3 B\right| dy_1 \le& \norm{\p_{y_1}^2 V}_{L^\infty(\R)} \norm{\p_{y_1}V}_{L^2(\R)} \norm{\p_{y_1}^3 V}_{L^2(\R)} \\
			\le& \norm{\p_{y_1}^2 V}_{L^2(\R)}^{\frac{1}{2}} \norm{\p_{y_1}^3 V}_{L^2(\R)}^{\frac{1}{2}} \norm{\p_{y_1}V}_{L^2(\R)} \norm{\p_{y_1}^3 V}_{L^2(\R)} \\
			\le& \chi\norm{\p_{y_1}^3 V}_{L^2 (\R)}^2 + C \chi^{-1}\norm{\p_{y_1}^2 V}_{L^2(\R)}^2 \norm{\p_{y_1} V}_{L^2(\R)}^4 \\
			\le& C(\bar{\delta} + \chi) \sum_{j=0}^{3} \left(\e^{-2} + \tau\right)^{-j} \norm{\p_{y_1}^{3-j} V}_{L^2(\R)}^2, \\
			\int_{\R} \left|\p_{y_1}^3 V\right| \left|\p_{y_1} V\right| \left|\p_{y_1}^3 B\right| dy_1 \le&\norm{\p_{y_1} V}_{L^\infty(\R)} \norm{\p_{y_1}^3 V}_{L^2(\R)}^2 \le \norm{\p_{y_1} V}_{L^2(\R)}^{\frac{1}{2}} \norm{\p_{y_1}^2 V}_{L^2(\R)}^{\frac{1}{2}} \norm{\p_{y_1}^3 V}_{L^2(\R)}^2 \\
			\le& C (\bar{\delta} + \chi) \norm{\p_{y_1}^3 V}_{L^2(\R)}^2, \\
			\int_{\R} \left|\p_{y_1}^2 V\right|^2 \left|\p_{y_1}^3 B\right| \le& C\bar{\delta} \norm{\p_{y_1}^3 V}_{L^2(\R)}^2 + C\norm{|\p_{y_1}^2 V|^2}_{L^2(\R)}^2 \le C\chi \norm{\p_{y_1}^3 V}_{L^2(\R)}^2 + C \chi^{-1} \norm{\p_{y_1}^2 V}_{L^4(\R)}^4.
		\end{align*}
		Using the Gagliardo-Nirenberg inequality, we find that
		\begin{align*}
			\norm{\p_{y_1}^2 V}_{L^4(\R)}^4 \le C \norm{\p_{y_1}^3 V}_{L^2(\R)}^\frac{5}{2} \norm{\p_{y_1} V}_{L^2(\R)}^\frac{3}{2} \le C\chi^2 \norm{\p_{y_1}^3 V}_{L^2(\R)}^2.
		\end{align*}
		The terms containing \(\mathcal{I}^{(l)}\) can be treated in the same way as (\ref{I_4^7}) and (\ref{I_4^8}).
		In summary, we obtain that
		\begin{align*}
			\int_{\R} \bar{B}^{(k)} \hat{M}_k^{(3)} dy_1 \le C\bar{\delta} \left(\e^{-2} + \tau\right)^{-1 -k} + \sqrt{\eta} \e \left(\e^{-2} + \tau\right)^{-\frac{1}{2} -k} + C \bar{\eta} \sum_{j=0}^{k+1} \left(\e^{-2} + \tau\right)^{-j} \norm{\p_{y_1}^{k+1-j} V}_{L^2(\R)}^2.
		\end{align*}
		
		\subsubsection*{Step 2. Estimates on \(\int_{\R} \tilde D_{-\frac12}^1  \left|\p_{y_1}^k b_i\right|^2 dy_1\)}

		\(~\)
		
		We still need to estimate \(\int_{\R} \tilde D_{-\frac12}^1  \left|\p_{y_1}^k b_i\right|^2 dy_1\) in the preceding subsection, because \(\int_{\R} \tilde D_{-\frac12}^1  \left|\p_{y_1}^k b_i\right|^2 dy_1\) can not be controlled by \(G_k\), \(k=0,1,2\). Let \(\tilde{\Gamma}(y,\tau) = \int_{-\infty}^{y_1} \tilde D_{-\frac12}^1  (z, \tau) dz\), then it is easy to check that
		\[\|\tilde{\Gamma}\|_{L^\infty} \le C \quad \text{and} \quad 4\tilde{c}_0 \p_\tau \tilde{\Gamma} = \p_{y_1} \tilde D_{-\frac12}^1 .\]
		Multiplying \((\ref{B equ})_1\) by \(\tilde{\Gamma} \p_{y_1}^k b_1\) yields that
		\begin{align}\label{b_1 cdot (h b_1)}
			&\frac{1}{2} \frac{d}{d \tau} \left(\tilde{\Gamma} \p_{y_1}^k b_1\right)^2 - \frac{1}{2} \p_\tau \tilde{\Gamma} \left(\p_{y_1}^k b_1\right)^2 -\frac{1}{2} \p_{y_1} \left(\tilde{\Gamma} \tilde{\lambda}_1\right) \left(\p_{y_1}^k b_1\right)^2 + \frac{1}{2} \p_{y_1} \left[\tilde{\Gamma} \tilde{\lambda}_1 \left(\p_{y_1}^k b_1\right)^2\right]\notag \\
			=& \left(\hat{M}_k^{(1)}\right)_1 \tilde{\Gamma} \p_{y_1}^k b_1 + \left[\left(\tilde{L} A_2 \tilde{R} \p_{y_1}^{k+2} B\right)_1 + \sum_{i=2}^{3} \left(\hat{M}^{(i)}\right)_1\right] \tilde{\Gamma} \p_{y_1}^k b_1.
		\end{align}
		A direct calculation yields that
		\begin{align*}
			\int_{\R} \left(\hat{M}_k^{(1)}\right)_1 \tilde{\Gamma} \p_{y_1}^k b_1 dy_1 \lesssim \bar{\delta} \int_{\R} \tilde D_{-\frac12}^1  \left|\p_{y_1}^k b_1\right|^2 dy_1 + \sum_{i=0}^{k} \left(\e^{-2} + \tau\right)^{-i} \int_{\R}  \p_{y_1} \bar{\rho} \left|\p_{y_1}^{k-1} b_1\right|^2 dy_1.
		\end{align*}
		Using the Cauchy inequality, it is straightforward to get that
		\begin{align*}
			&\int_{\R} \left[\left(\tilde{L} A_2 \tilde{R} \p_{y_1}^{k+2} B\right)_1 + \sum_{i=2}^{3} \left(\hat{M}^{(i)}\right)_1\right] \tilde{\Gamma} \p_{y_1}^k b_1 dy_1 \\
			\lesssim& \bar{\eta} \sum_{j=0}^{k+1} \left(\e^{-2} + \tau\right)^{j-k-1} \norm{\p_{y_1}^j V}_{L^2(\R)}^2 + \sqrt{\eta} \e \left(\e^{-2} +\tau\right)^{-\frac{1}{2} - k} + \bar{\delta}  \left(\e^{-2} +\tau\right)^{-1 - k}.
		\end{align*}
		Integrating (\ref{b_1 cdot (h b_1)}) with respect to \(y_1\), we get
		\begin{align*}
			\int_{\R} \tilde D_{-\frac12}^1  \left|\p_{y_1}^k b_1\right|^2 dy_1 + \p_\tau \int_{\R} \tilde{\Gamma} \left|\p_{y_1}^k b_1\right|^2 dy_1 \lesssim& \bar{\eta} \sum_{j=0}^{k+1} \left(\e^{-2} + \tau\right)^{j-k-1} \norm{\p_{y_1}^j V}_{L^2(\R)}^2 + \sqrt{\eta} \e \left(\e^{-2} +\tau\right)^{-\frac{1}{2} - k} + \bar{\delta}  \left(\e^{-2} +\tau\right)^{-1 - k} \\
			&+ \sum_{j=0}^{k} \left(\e^{-2} + \tau\right)^{-j} \int_{\R} \p_{y_1}\bar{\rho} \left|\p_{y_1}^{k-j} b_1\right|^2 dy_1.
		\end{align*}
		By a similar argument, we find that
		\begin{align*}
			\int_{\R} \tilde D_{-\frac12}^1  \left|\p_{y_1}^k b_3\right|^2 dy_1 - \p_\tau \int_{\R} \tilde{\Gamma} \left|\p_{y_1}^k b_3\right|^2 dy_1 \lesssim& \bar{\eta} \sum_{j=0}^{k+1} \left(\e^{-2} + \tau\right)^{j-k-1} \norm{\p_{y_1}^j V}_{L^2(\R)}^2 + \sqrt{\eta} \e \left(\e^{-2} +\tau\right)^{-\frac{1}{2} - k} + \bar{\delta}  \left(\e^{-2} +\tau\right)^{-1 - k} \\
			&+ \sum_{j=0}^{k} \left(\e^{-2} + \tau\right)^{-j} \int_{\R} \p_{y_1}\bar{\rho} \left|\p_{y_1}^{k-j} b_3\right|^2 dy_1.
		\end{align*}
		Combining the above inequalities yields that
		\begin{align}
			&\frac{d}{d \tau} \left(\sum_{i=0}^{2} \tilde{E}_i\right) + \sum_{i=0}^{2} \left(\tilde{H}_i + G_i\right) \le C \bar{\eta} \left[\left(\e^{-2} + \tau\right)^{-1} \left(\sum_{i=0}^{2} E_i\right) + \left(\norm{\p_{y_1} \tilde{\Phi}}_{H^2(\R)}^2 + \norm{\p_{y_1} \tilde{\Psi}_2}_{H^2(\R)}^2 + \norm{\p_{y_1} \tilde{\Psi}_3}_{H^2(\R)}^2\right)\right] \notag \\
			&\qquad\qquad\qquad\qquad\qquad\qquad\qquad+ C \bar{\delta} \left(\e^{-2} + \tau\right)^{-1} + C\sqrt{\eta} \e \left(\e^{-2} + \tau\right)^{-\frac{1}{2}}, \label{tilde E_0 est}\\
			&\frac{d}{d \tau} \left(\sum_{i=1}^{2} \tilde{E}_i\right) + \sum_{i=1}^{2} \left(\tilde{H}_i + G_i\right) \le C\bar{\eta} \left(\e^{-2} + \tau\right)^{-1} \left(\sum_{i=1}^{2} E_i + G_0\right) + C\bar{\delta} \left(\e^{-2} + \tau\right)^{-2} + C \sqrt{\eta} \e\left(\e^{-2} + \tau\right)^{-\frac{3}{2}} \notag\\
			& \qquad\qquad\qquad\qquad\qquad\qquad\qquad+ C \bar{\eta} \left(\norm{\p_{y_1}^2 \tilde{\Phi}}_{H^2(\R)}^2 + \norm{\p_{y_1}^2 \tilde{\Psi}_2}_{H^2(\R)}^2 + \norm{\p_{y_1}^2 \tilde{\Psi}_3}_{H^2(\R)}^2\right)+C\bar{\eta} \left(\e^{-2} + \tau\right)^{-2} E_0, \label{tilde E_1 est}\\
			&\frac{d}{d \tau} \tilde{E}_2 + \tilde{H}_2 + G_2 \le C\bar{\eta} \sum_{i=0}^{2} \left(\e^{-2} + \tau\right)^{-(3-i)} E_i + C\bar{\eta} \left(\norm{\p_{y_1}^3 \tilde{\Phi}}_{H^2(\R)}^2 + \norm{\p_{y_1}^3 \tilde{\Psi}_2}_{H^2(\R)}^2 + \norm{\p_{y_1}^3 \tilde{\Psi}_3}_{H^2(\R)}^2\right) \notag \\
			& \qquad\qquad\qquad\qquad\qquad C\bar{\eta} \left[\left(\e^{-2} + \tau\right)^{-2} G_0 + \left(\e^{-2} + \tau\right)^{-1} G_1\right] + C \bar{\delta} \left(\e^{-2} + \tau\right)^{-3} + C\sqrt{\eta} \e \left(\e^{-2} + \tau\right)^{-\frac{5}{2}}. \label{tilde E_2 est}
		\end{align}
		
		\subsubsection*{Step 3. Estimates on \(\norm{\p_{y_1}^{k+1} \tilde{\Phi}}_{L^2(\R)}\)}

		\(~\)
		
		For \(k = 0, 1, 2\), considering \(\int_{\R} \p_{y_1}^k (\ref{the equ of tilde Phi, Psi, W})_2 \p_{y_1}^{k+1} \tilde{\Phi} dy_1\), we have
		\begin{align}
			\p_\tau \int_{\R} \p_{y_1}^{k+1} \tilde{\Phi} \p_{y_1}^k \tilde{\Psi}_1 dy_1 + \norm{\p_{y_1}^{k+1} \tilde{\Phi}}_{L^2(\R)}^2 = \int_{\R} \frac{\lambda + 2\mu}{\tilde{\rho}} \p_{y_1}^{k+2}\tilde{\Psi}_1 \p_{y_1}^{k+1}\tilde{\Phi} dy_1 + I_5 + I_6,
		\end{align}
		where
		\begin{align*}
			I_5 =& \int_{\R} \p_{y_1}^{k+1}\tilde{\Phi} \p_{y_1}^k Y_{21}^{(2)} dy_1, \\
			I_6 =& \int_{\R} \p_{y_1}^{k+1}\tilde{\Phi} \left[\p_{y_1}^k \mathcal{S}_{21} + \p_{y_1}^{k+1} \tilde{W} + \p_{y_1}^k Y_{21}^{(1)} + \sum_{i=0}^{k-1} \p_{y_1}^{k-i} \left(\frac{\lambda + 2\mu}{\tilde{\rho}}\right) \p_{y_1}^{i+2} \tilde{\Psi}_1 \right] dy_1  \\
			&- \int_{\R} \p_{y_1}^{k+1} \tilde{\Psi}_1 \p_{y_1}^k \left(Y_1 - \tilde{\T} \mathcal{S}_1 - \tilde{\T} \p_{y_1}\tilde{\Psi}_1\right) dy_1.
		\end{align*}
		On the other hand, considering \(\int_{\R} \frac{\lambda + 2\mu}{\tilde{\T} \tilde{\rho}} \p_{y_1}^{k+1} (\ref{the equ of tilde Phi, Psi, W})_1 \p_{y_1}^{k+1} \tilde{\Phi} dy_1\), we obtain
		\begin{align}
			\p_\tau \int_{\R} \frac{\lambda + 2\mu}{2 \tilde{\T} \tilde{\rho}}\left|\p_{y_1}^{k+1} \tilde{\Phi}\right|^2 dy_1 = -\int_{\R} \frac{\lambda + 2\mu}{\tilde{\rho}} \p_{y_1}^{k+2} \tilde{\Psi}_1 \p_{y_1}^{k+1} \tilde{\Phi} dy_1 + I_7 + I_8,
		\end{align}
		where
		\begin{align*}
			I_7 =& \int_{\R} \frac{\lambda + 2\mu}{ \tilde{\T} \tilde{\rho}} \p_{y_1}^{k+1} Y_1 \p_{y_1}^{k+1} \tilde{\Phi} dy_1, \\
			I_8 =& \int_{\R} \p_\tau \left(\frac{\lambda + 2\mu}{2 \tilde{\T} \tilde{\rho}}\right) \left|\p_{y_1}^{k+1} \tilde{\Phi}\right|^2 + (\lambda + 2\mu) \p_{y_1}^{k+1} \tilde{\Phi} \left[\frac{\p_{y_1}^{k+1} (\tilde{\T} \mathcal{S}_1)}{\tilde{\T} \tilde{\rho}} - \sum_{i=0}^{k+1} \frac{\p_{y_1}^{k+1-i} \tilde{\T} \p_{y_1}^i \tilde{\Psi}_1}{\tilde{\T} \tilde{\rho}}\right] dy_1.
		\end{align*}
		It follows from Lemma \ref{Z est} - Lemma \ref{J est} and (\ref{a priori assumption}) that
		\begin{align}\label{I_6 + I_8}
			\left|I_6\right| + \left|I_8\right| \le& \frac{1}{160} \norm{\p_{y_1}^{k+1} \tilde{\Phi}}_{L^2(\R)}^2 + C \left(\sum_{i=1}^{3} \norm{\p_{y_1}^{k+1} \tilde{\Psi}_i}_{L^2(\R)}^2 + \norm{\p_{y_1}^{k+1} \tilde{W}}_{L^2(\R)}^2\right) \notag \\
			&+ C\bar{\eta} \sum_{j=1}^{k} \left(\e^{-2} + \tau\right)^{-j} \norm{\p_{y_1}^{k+1-j} V}_{L^2(\R)}^2 + C\bar{\delta} \left(\e^{-2} + \tau\right)^{-\frac{3}{2} - k}.
		\end{align}
		As for \(I_5\), particular attention should be paid to the highest-order term. In fact,
		\begin{align}\label{I_5}
			&I_5 = \int_{\R} \p_{y_1}^{k+1} \tilde{\Phi} \p_{y_1}^k \left[(\lambda + 2\mu) \p_{y_1} \left(\e \mathring{u}_1 - \e \tilde{u}_1 \right) - \frac{\lambda+2\mu}{\tilde{\rho}} \p_{y_1}^2 \tilde{\Psi}\right] dy_1 \notag\\
			=& (\lambda + 2\mu) \int_{\R} \p_{y_1}^k \left[\p_{y_1} \mathbf{D}_0 \left(\frac{\e m_1}{\rho} - \frac{\e \mathring{m}_1}{\mathring{\rho}}\right) + \p_{y_1} \mathbf{D}_0 \left(\frac{\e \mathring{m}_1}{\mathring{\rho}} - \frac{\e \tilde{m}_1}{\tilde{\rho}}\right) - \p_{y_1} \left(\frac{1}{\tilde{\rho}} \p_{y_1} \tilde{\Psi}_1\right) + \p_{y_1} \left(\frac{1}{\tilde{\rho}}\right) \p_{y_1} \tilde{\Psi}_1\right] \p_{y_1}^{k+1} \tilde{\Phi} dy_1 \notag\\
			=& (\lambda + 2\mu) \int_{\R} \p_{y_1}^k \left[\p_{y_1} \mathbf{D}_0 \left(\frac{\e m_1}{\rho} - \frac{\e \mathring{m}_1}{\mathring{\rho}}\right) + \p_{y_1} \mathbf{D}_0 \left(\frac{\e \mathring{m}_1 \tilde{\rho} - \e \tilde{m}_1 \tilde{\rho} + \e \tilde{m}_1 \tilde{\rho} - \e \tilde{m}_1 \mathring{\rho}}{\mathring{\rho} \tilde{\rho}}\right)\right. \notag\\
			&\left. - \p_{y_1} \left(\frac{1}{\tilde{\rho}} \p_{y_1} \tilde{\Psi}_1\right) + \p_{y_1} \left(\frac{1}{\tilde{\rho}}\right) \p_{y_1} \tilde{\Psi}_1 \right] \p_{y_1}^{k+1} \tilde{\Phi} dy_1 \notag\\
			=& (\lambda + 2\mu) \int_{\R} \p_{y_1}^k \left[\p_{y_1} \mathbf{D}_0 \left(\frac{\e m_1}{\rho} - \frac{\e \mathring{m}_1}{\mathring{\rho}}\right) + \p_{y_1} \mathbf{D}_0 \left(\frac{\p_{y_1} \Psi_1}{\mathring{\rho}}\right) - \p_{y_1} \left(\frac{\p_{y_1} \tilde{\Psi}_1}{\tilde{\rho}}\right) \right. \notag\\
			&\left. - \p_{y_1} \left(\frac{\e \tilde{u}_1 \p_{y_1} \Phi}{\mathring{\rho}}\right) + \p_{y_1}\left(\frac{1}{\tilde{\rho}}\right) \p_{y_1} \tilde{\Psi}_1\right] \p_{y_1}^{k+1} \tilde{\Phi} dy_1 \notag\\
			=& (\lambda + 2\mu) \int_{\R} \p_{y_1}^k \left[\p_{y_1} \left(\frac{1}{\tilde{\rho}}\right) \p_{y_1} \Psi_1 - \p_{y_1} \left(\frac{\p_{y_1} \Phi \p_{y_1} \Psi_1}{\mathring{\rho} \tilde{\rho}}\right) + \frac{1}{\tilde{\rho}} \p_{y_1} \mathcal{C}_{1r}- \p_{y_1} \left(\frac{\e \tilde{u}_1 \p_{y_1} \Phi}{\mathring{\rho}}\right) \right. \notag\\
			&\left.+ \p_{y_1} \mathbf{D}_0 \left(\frac{\e m_1}{\rho} - \frac{\e \tilde{m}_1}{\tilde{\rho}}\right)\right] \p_{y_1}^{k+1} \tilde{\Phi} dy_1.
		\end{align}
		We focus on the term which could contain the (k + 2)-th order derivative:
		\[\p_{y_1}^k \left[- \p_{y_1} \left(\frac{\p_{y_1} \Phi \p_{y_1} \Psi_1}{\mathring{\rho} \tilde{\rho}}\right) + \frac{1}{\tilde{\rho}} \p_{y_1} \mathcal{C}_{1r}- \p_{y_1} \left(\frac{\e \tilde{u}_1 \p_{y_1} \Phi}{\mathring{\rho}}\right)\right].\]
		Notice that
		\begin{align*}
			\p_{y_1}^{k+2} \Phi =& \p_{y_1}^{k+2} \left(\frac{\tilde{\Phi}}{\tilde{\T}}\right) = \frac{1}{\tilde{\T}} \p_{y_1}^{k+2} \tilde{\Phi} + [\p_{y_1}^{k+2}, \frac{1}{\tilde{\T}}] \tilde{\Phi} 
			=\frac{1}{\tilde{\T}} \p_{y_1}^{k+2} \tilde{\Phi} + \sum_{i=0}^{k+1} \p_{y_1}^{k+2-i} \left(\frac{1}{\tilde{\T}}\right) \p_{y_1}^i \tilde{\Phi},
		\end{align*}
		then,  integration by parts can be applied to the terms containing  \(\p_{y_1}^{k+2} \Phi\) to transfer the derivatives. In what follows, we shall only provide a detailed treatment of the term \(-\frac{1}{\mathring{\rho} \tilde{\rho}} \p_{y_1} \Phi \p_{y_1}^{k+2} \Psi_1\) in \(-\p_{y_1}^{k+1} \left(\frac{\p_{y_1}\Phi \p_{y_1}\Psi_1}{\mathring{\rho} \tilde{\rho}}\right)\). 
		By \((\ref{the anti equs})_1\), we have
		\begin{equation}\label{k+2 order key term 1}
			\begin{aligned}
				&-\int_{\R} \frac{\p_{y_1} \Phi \p_{y_1}^{k+2} \Psi_1}{\mathring{\rho} \tilde{\rho}} \p_{y_1}^{k+1} \tilde{\Phi} dy_1 = - \int_{\R} \frac{\p_{y_1} \Phi \p_{y_1}^{k+1} (\mathcal{S}_1 - \p_\tau \Phi)}{\mathring{\rho} \tilde{\rho}} \p_{y_1}^{k+1} \tilde{\Phi} dy_1 \\
				=&\int_{\R} \frac{\p_{y_1}\Phi \p_\tau \p_{y_1}^{k+1} \Phi - \p_{y_1} \Phi \p_{y_1}^{k+1} \mathcal{S}_1}{\mathring{\rho} \tilde{\rho}} \p_{y_1}^{k+1} \tilde{\Phi} dy_1 \\
				=&\int_{\R} \frac{\p_{y_1}\Phi \p_\tau \left[\frac{1}{\tilde{\T}} \p_{y_1}^{k+1} \tilde{\Phi} + [\p_{y_1}^{k+1}, \frac{1}{\tilde{\T}}] \tilde{\Phi}\right] - \p_{y_1}\Phi \p_{y_1}^{k+1} \mathcal{S}_1}{\mathring{\rho} \tilde{\rho}} \p_{y_1}^{k+1} \tilde{\Phi} dy_1 \\
				=& \int_{\R} \left[\frac{\p_{y_1} \Phi \p_\tau \p_{y_1}^{k+1} \tilde{\Phi}}{\tilde{\T}\mathring{\rho} \tilde{\rho}} + \frac{\p_{y_1}\Phi \p_{y_1}^{k+1} \tilde{\Phi} \p_\tau \left(\frac{1}{\tilde{\T}}\right)}{\mathring{\rho} \tilde{\rho}} + \frac{\p_{y_1}\Phi [\p_{y_1}^{k+1}, \frac{1}{\tilde{\T}}]\tilde{\Phi} - \p_{y_1}\Phi \p_{y_1}^{k+1} \mathcal{S}_1}{\mathring{\rho} \tilde{\rho}}\right] \p_{y_1}^{k+1} \tilde{\Phi} dy_1 \\
				=& \p_\tau \int_{\R} \frac{\p_{y_1} \Phi (\p_{y_1}^{k+1} \tilde{\Phi})^2}{2\tilde{\T}\mathring{\rho} \tilde{\rho}} dy_1 - \int_{\R} \left[\p_\tau \left(\frac{\p_{y_1} \Phi}{2\tilde{\T}\mathring{\rho} \tilde{\rho}}\right) - \frac{\p_{y_1}\Phi \p_\tau (\frac{1}{\tilde{\T}})}{\mathring{\rho} \tilde{\rho}}\right] \left|\p_{y_1}^{k+1} \tilde{\Phi}\right|^2 dy_1 \\
				&- \int_{\R} \frac{\p_{y_1} \Phi \p_{y_1}^{k+1} \mathcal{S}_1 \p_{y_1}^{k+1} \tilde{\Phi}}{\mathring{\rho} \tilde{\rho}}	dy_1 + \int_{\R} \frac{\p_{y_1}\Phi}{\mathring{\rho} \tilde{\rho}} \p_\tau \left[\sum_{i=0}^{k} \p_{y_1}^{k+1-i} \left(\frac{1}{\tilde{\T}}\right) \p_{y_1}^i \tilde{\Phi}\right] \p_{y_1}^{k+1} \tilde{\Phi} dy_1.
			\end{aligned}
		\end{equation}
		Combining (\ref{I_5}), (\ref{k+2 order key term 1}), Lemma \ref{Z est}, Lemma \ref{J est}, Lemma \ref{Q est} and the a priori assumptions \ref{a priori assumption}, we obtain
		\begin{align}
			I_5 - (\lambda + 2\mu) \p_\tau \int_{\R} \frac{\p_{y_1} \Phi (\p_{y_1}^{k+1} \tilde{\Phi})^2}{2\tilde{\T}\mathring{\rho} \tilde{\rho}} dy_1 \le C\bar{\eta} \sum_{j=0}^{k} \left(\e^{-2} + \tau\right)^{-j} \norm{\p_{y_1}^{k+1-j} V}_{L^2(\R)}^2 +C \e^2 \norm{\mathcal{I}^{(k+1)}}_{L^2(\R)}^2.
		\end{align}
		For \(I_7\), we shall also be concerned with the highest-order term:
		\begin{equation}\label{k+2 order key term 2}
			\begin{aligned}
				I_7 =& \int_{\R} (\lambda + 2\mu) \frac{\p_{y_1}^{k+1} (\p_\tau \tilde{\T} \Phi - \tilde{\T} \mathcal{C}_{1r})}{\tilde{\T} \tilde{\rho}} \p_{y_1}^{k+1} \tilde{\Phi} dy_1 \\
				=& \int_{\R} (\lambda + 2\mu) \frac{\p_{y_1}^{k+1} \left(\p_\tau \tilde{\T} \Phi - \e \tilde{u}_1 \tilde{\T} \p_{y_1} \Phi - \tilde{\T} \p_{y_1}(\e \tilde{u}_1) \Phi\right)}{\tilde{\T} \tilde{\rho}} \p_{y_1}^{k+1} \tilde{\Phi} dy_1 \\
				=&\int_{\R} (\lambda + 2\mu) \frac{\p_{y_1}^{k+1} \left(\p_\tau \tilde{\T} \Phi - \e \tilde{u}_1 \p_{y_1} \tilde{\Phi} + \e \tilde{u}_1 \p_{y_1} \tilde{\T} \Phi - \p_{y_1}(\e \tilde{u}_1) \tilde{\Phi}\right)}{\tilde{\T} \tilde{\rho}} \p_{y_1}^{k+1} \tilde{\Phi} dy_1 \\
				=& \int_{\R} (\lambda + 2\mu) \frac{\p_{y_1}^{k+1} \left(\p_\tau \tilde{\T} \Phi + \e \tilde{u}_1 \p_{y_1} \tilde{\T} \Phi - \p_{y_1}(\e \tilde{u}_1) \tilde{\Phi}\right) - \sum_{i=0}^{k} \p_{y_1}^{k+1-i} (\e \tilde{u}_1) \p_{y_1}^{i+1} \tilde{\Phi}}{\tilde{\T} \tilde{\rho}} \p_{y_1}^{k+1} \tilde{\Phi} dy_1 \\
				&+ \int_{\R} (\lambda + 2\mu) \p_{y_1} \left(\frac{\e \tilde{u}_1}{2\tilde{\T} \tilde{\rho}}\right) \left|\p_{y_1}^{k+1} \tilde{\Phi}\right|^2 dy_1 \\
				\le&C\bar{\eta} \sum_{j=0}^{k}  \left(\e^{-2} + \tau\right)^{-j} \norm{\p_{y_1}^{k+1-j} V}_{L^2(\R)}^2.
			\end{aligned}
		\end{equation}
		Combining the above inequalities, we get
		\begin{equation}\label{tilde Phi est}
			\begin{aligned}
				&\frac{d}{d \tau} \int_{\R} \left(\p_{y_1}^{k+1} \tilde{\Phi} \p_{y_1}^k \tilde{\Psi}_1 + \frac{\lambda + 2\mu}{2\tilde{\rho} \tilde{\T}} \left|\p_{y_1}^{k+1} \tilde{\Phi}\right|^2 - (\lambda + 2\mu) \frac{\p_{y_1} \Phi |\p_{y_1}^{k+1} \tilde{\Phi}|^2}{2\mathring{\rho} \tilde{\rho}}\right) dy_1 + \frac{1}{2} \norm{\p_{y_1}^{k+1} \tilde{\Phi}}_{L^2(\R)}^2 \\
				\le& C\bar{\eta} \sum_{j=1}^{k} \left(\e^{-2} + \tau\right)^{-j} \norm{\p_{y_1}^{k+1-j} V}_{L^2(\R)}^2 + C \bar{\delta} \left(\e^{-2} + \tau\right)^{-\frac{3}{2}-k} + C \left(\sum_{i=1}^{3} \norm{\p_{y_1}^{k+1} \tilde{\Psi}_i}_{L^2(\R)}^2 + \norm{\p_{y_1}^{k+1} \tilde{W}}_{L^2(\R)}^2\right).
			\end{aligned}
		\end{equation}
		
		\subsubsection*{Step 4. Estimates on \(\tilde{\Psi}_i, i=2,3\)}

		\(~\)
		
		For \(k=0,1,2\) and \(i=2,3\), considering \(\int_{\R} \p_{y_1}^k (\ref{the equ of tilde Phi, Psi, W})_3 \p_{y_1}^k \tilde{\Psi}_i dy_1\), we have
		\begin{align*}
			\frac{d}{d \tau} \norm{\p_{y_1}^k \tilde{\Psi}_i}_{L^2(\R)}^2 + \int_{\R} \frac{\mu}{\tilde{\rho}} \left|\p_{y_1}^{k+1} \tilde{\Psi}_i\right|^2 dy_1 \le& C \int_{\R} \left|\p_{y_1}^2 \left(\frac{\mu}{\tilde{\rho}}\right)\right| \left|\p_{y_1}^{k} \tilde{\Psi}_i\right|^2 dy_1 + \int_{\R} \left|\p_{y_1}^k (J_{2i} + \mathcal{S}_i) \right|\p_{y_1}^{k} \tilde{\Psi}_i dy_1 \\
			:=& I_{\psi}^{k1} + I_{\psi}^{k2}.
		\end{align*}
		Immediately, \(I_{\psi}^{k1}\) can be estimated as 
		\begin{align*}
			I_{\psi}^{k1} \le& C \int_{\R} \left|\frac{|\p_{y_1} \tilde{\rho}|^2}{\tilde{\rho}^2} - \frac{\p_{y_1}^2 \tilde{\rho}}{\tilde{\rho}}\right| |\p_{y_1}^2 \tilde{\Psi}_i| dy_1 
			\le C\bar{\eta} \sum_{j=0}^{k+1} \left(\e^{-2} + \tau\right)^{-j} \norm{\p_{y_1}^{k+1-j} V}_{L^2(\R)}^2,
		\end{align*}
		and \(I_{\psi}^{k2}\) can be treated in a same way as (\ref{I_4^1}) - (\ref{I_4^10}). Finally, it reaches
		\begin{equation}\label{tilde Psi_i est}
			\begin{aligned}
				&\sum_{i=2}^{3} \left[\frac{d}{d \tau} \norm{\p_{y_1}^k \tilde{\Psi}_i}_{L^2(\R)}^2 + \int_{\R} \frac{\mu}{\tilde{\rho}} \left|\p_{y_1}^{k+1} \tilde{\Psi}_i\right|^2 dy_1\right] \\
				\le& C \bar{\eta} \left[\norm{\p_{y_1}^{k+1} \tilde{\Phi}}_{L^2(\R)}^2 + \norm{\p_{y_1}^{k+1} \tilde{\Psi}_1}_{L^2(\R)}^2 + \norm{\p_{y_1}^{k+1} \tilde{W}}_{L^2(\R)}^2 + \sum_{j=1}^{k+1} \left(\e^{-2} + \tau\right)^{-j} \norm{\p_{y_1}^{k+1-j} V}_{L^2(\R)}^2\right] \\
				&+ C\bar{\delta} \left(\e^{-2} + \tau\right)^{-1 - k} + C \sqrt{\eta} \e \left(\e^{-2} + \tau\right)^{-\frac{1}{2} - k}.
			\end{aligned}
		\end{equation}
		Combining (\ref{tilde Phi est}), (\ref{tilde Psi_i est}) and (\ref{tilde E_0 est}) with (\ref{tilde E_2 est}), we have proved the lemma \ref{main lemma}.
	\end{proof}
	
	\subsection*{Decay rate for the zero modes}

	\(~\)
	
	\begin{Lem}\label{zero mode est}
		Under the same assumptions as Proposition \ref{a priori estimates}, it holds that
		\begin{align*}
			&\norm{\p_{y_1}^k (\mathring{\phi}, \mathring{\psi}, \mathring{\zeta})}_{L^2(\R)}^2 \le C \e (\eta+\bar\delta) \left(\e^{-2} + \tau\right)^{-\frac{2k+1}{2}}, \quad k=0,1, \\
			&\norm{(\mathring{\phi}, \mathring{\psi}, \mathring{\zeta})}_{L^\infty(\R)} \le C  \e^{\frac{1}{2}}(\eta+\bar\delta)\left(\e^{-2} + \tau\right)^{-\frac{1}{2}}.
		\end{align*}
	\end{Lem}

	%First, since \(\e^{-2} + \tau \ge \e^{-2}\), we have \(\left(\e^{-2} + \tau\right)^{-\frac{1}{2}} \le \e\). Consequently, \(\bar{\delta} \left(\e^{-2} + \tau\right)^{-1} \le \bar{\delta} \e \left(\e^{-2} + \tau\right)^{-\frac{1}{2}}\). By a same way, we obtain that
	%\[\bar{\delta} \left(\e^{-2} + \tau\right)^{-2} \le \bar{\delta} \e \left(\e^{-2} + \tau\right)^{-\frac{3}{2}}, \qquad \bar{\delta} \left(\e^{-2} + \tau\right)^{-3} \le \bar{\delta} \e \left(\e^{-2} + \tau\right)^{-\frac{5}{2}}.\]
	%Hence, provided that \(\e\) tends to 0, the source terms in (\ref{E_0 est})-(\ref{E_2 est}) may be bounded respectively by
	%\[C \delta \e \left(\e^{-2} + \tau\right)^{-\frac{1}{2}}, \qquad C \delta \e \left(\e^{-2} + \tau\right)^{-\frac{3}{2}}, \qquad C \delta \e \left(\e^{-2} + \tau\right)^{-\frac{5}{2}}.\]
	\begin{proof}
		According to (\ref{tilde E_k norm}), we find
		\begin{align*}
			E_1 \le C K_0 + C \bar{\eta} \left(\e^{-2} + \tau\right)^{-1} E_0, \qquad E_2 \le C K_1 + C\bar{\eta} \left(\e^{-2} + \tau\right)^{-1} K_0 + C \bar{\eta} \left(\e^{-2} + \tau\right)^{-2} E_0.
		\end{align*}
		By the Grönwall's inequality and (\ref{E_0 est}), one gets
		\begin{align*}
			\sum_{i=0}^{2} E_i \le C  (\sqrt{\eta}+\bar \delta)  \e \left(\e^{-2} + \tau\right)^{\frac{1}{2}}, \quad \int_{0}^{\tau} \sum_{i=0}^{2} (K_i + G_i) ds \le C  (\sqrt{\eta}+\bar \delta) \e \left(\e^{-2} + \tau\right)^{\frac{1}{2}}.
		\end{align*}
		Considering \(\int_{0}^{\tau} (\ref{E_1 est}) \times (\e^{-2} + s) ds\), we have
		\begin{align}
			(\e^{-2} + \tau)\sum_{i=1}^{2} E_i + \int_{0}^{\tau} (\e^{-2} + s) \sum_{i=1}^{2} (K_i + G_i) ds \le C  (\e^2 + \sqrt{\eta}+ \bar \delta) \e \left(\e^{-2} + \tau\right)^{\frac{1}{2}},
		\end{align}
		which leads to 
		\begin{align}
			\sum_{i=1}^{2} E_i \le C  (\sqrt{\eta}+\bar \delta) \e \left(\e^{-2} + \tau\right)^{-\frac{1}{2}}, \quad \int_{0}^{\tau} (\e^{-2} + s) \sum_{i=1}^{2} (K_i + G_i) ds \le C  (\e^2 + \sqrt{\eta} + \bar \delta) \e \left(\e^{-2} + \tau\right)^{\frac{1}{2}}.
		\end{align}
		Considering \(\int_{0}^{\tau} (\ref{E_2 est}) \times (\e^{-2} + s)^2 ds\), we have
		\begin{align}\label{lab1}
			(\e^{-2} + \tau)^2 E_2 + \int_{0}^{\tau} (\e^{-2} + s)^2 (K_2 + G_2) ds \le C  (\sqrt{\eta}+\bar \delta) \e  \left(\e^{-2} + \tau\right)^{\frac{1}{2}},
		\end{align}
		which leads to 
		\begin{align}
			E_2 \le C  (\sqrt{\eta}+\bar \delta) \e \left(\e^{-2} + \tau\right)^{-\frac{3}{2}}.
		\end{align}
		Then we have
		\begin{align}
			\norm{(\mathring{\phi}, \mathring{\psi}, \mathring{\zeta})}_{L^\infty} \le& \norm{(\mathring{\phi}, \mathring{\psi}, \mathring{\zeta})}_{L^2}^\frac{1}{2} \norm{\p_{y_1} (\mathring{\phi}, \mathring{\psi}, \mathring{\zeta})}_{L^2}^\frac{1}{2}
			\le C  (\sqrt{\eta}+\bar \delta)^{\frac{1}{2}} \e^{\frac{1}{2}} \left(\e^{-2} + \tau\right)^{-\frac{1}{2}}.
		\end{align}
		Therefore, we obtain the decay rate for the zero mode.
	\end{proof}

	\subsection{Estimates for the non-zero modes}
	
	In this subsection, we will give the energy estimate for the non-zero mode of the perturbation. Since \[\int_{(\frac{\mathbb{T}}{\e})^2} (\phi_{\neq}, \psi_{\neq}, \zeta_{\neq}) dy_3 = 0,\] then we have the following Poincar\'e inequality for the non-zero mode in \(\Omega_{\e}\):
	\begin{align}
		\norm{(\phi_{\neq}, \psi_{\neq}, \zeta_{\neq})}_{L^2(\Omega_\e)} \le C \e^{-k} \norm{\nabla^k (\phi_{\neq}, \psi_{\neq}, \zeta_{\neq})}_{L^2(\Omega_\e)}, \qquad \forall k \ge 1.
	\end{align} 
	Under the transformation \(y = \frac{x}{\e}\), the length in the periodic directions is scaled by \(\frac{1}{\e}\). This is the reason that the \(k\)th-order Poincar\'e inequality on \(\Omega_\e\) holds with \(\e^{-k}\).
	Applying \(\mathbf{D}_{\neq}\) to (\ref{the perturbation equs}),  we have
	\begin{equation}\label{the non-zero equs}
		\left\{
		\begin{aligned}
			&\p_\tau \phi_{\neq} + \e \mathring{\mathbf{u}} \cdot \nabla \phi_{\neq} + \mathring{\rho} \dv \psi_{\neq} = \mathcal{A}_{0 \neq}, \\
			&\p_\tau \psi_{\neq} + \e \mathring{\mathbf{u}} \cdot \nabla \psi_{\neq} + \frac{1}{\tilde{\rho}} \left(\mathring{\T} \nabla \phi_{\neq} + \mathring{\rho} \nabla \zeta_{\neq}\right) = \frac{\mu}{\tilde{\rho}} \Delta \psi_{\neq} + \frac{\lambda+\mu}{\tilde{\rho}} \nabla \dv \psi_{\neq} + \mathcal{A}_{\neq}, \\
			&\p_\tau \zeta_{\neq} + \e \mathring{\mathbf{u}} \cdot \nabla \zeta_{\neq} + \mathring{\T} \dv \psi_{\neq} = \frac{\kappa}{\tilde{\rho}} \Delta \zeta_{\neq} + \mathcal{A}_{4 \neq},
		\end{aligned}
		\right.
	\end{equation}
	where
	\begin{align*}
		\mathcal{A}_{0 \neq} =& \e \mathring{\mathbf{u}} \cdot \nabla \phi_{\neq} - \mathbf{D}_{\neq} \left(\e \mathbf{u} \cdot \nabla \phi \right) + \mathring{\rho} \dv \psi_{\neq} - \mathbf{D}_{\neq} \left(\rho \dv \psi \right) - \psi_{\neq} \cdot \nabla \tilde{\rho} - \phi_{\neq} \dv (\e \tilde{\mathbf{u}}) , \\
		\mathcal{A}_{\neq} =& \e \mathring{\mathbf{u}} \cdot \nabla \psi_{\neq} - \mathbf{D}_{\neq} \left(\e \mathbf{u} \cdot \nabla \psi \right) + \frac{1}{\tilde{\rho}} \left(\mathring{\T} \nabla \phi_{\neq} + \mathring{\rho} \nabla \zeta_{\neq}\right) - \frac{\nabla p_{\neq}}{\tilde{\rho}} - \psi_{\neq} \cdot \nabla (\e \tilde{\mathbf{u}}) + \frac{1}{\tilde{\rho}} \mathbf{D}_{\neq} \left(\frac{\phi \nabla p}{\rho}\right) \\
		&+ \mathbf{D}_{\neq} \left[\left(\frac{1}{\rho} - \frac{1}{\tilde{\rho}}\right) (\mu \Delta (\e \mathbf{u})) + (\lambda+\mu) \nabla \dv (\e \mathbf{u})\right], \\
		\mathcal{A}_{4 \neq} =& \e \mathring{\mathbf{u}} \cdot \nabla \zeta_{\neq} - \mathbf{D}_{\neq} \left(\e \mathbf{u} \cdot \nabla \zeta \right) + \mathring{\T} \dv \psi_{\neq} - \mathbf{D}_{\neq} (\T \dv \psi) - \psi_{\neq} \cdot \nabla \tilde{\T} - \zeta_{\neq} \dv (\e \tilde{\mathbf{u}}) + \kappa \mathbf{D}_{\neq} \left[\left(\frac{1}{\rho} - \frac{1}{\tilde{\rho}}\right) \Delta \T\right] \\
		&+ \mathbf{D}_{\neq} \left[\frac{1}{\rho} \left(\frac{\mu}{2}\left|\nabla (\e \mathbf{u}) + \nabla (\e \mathbf{u})^t \right|^2 + \lambda |\dv (\e \mathbf{u})|^2\right) - \frac{1}{\tilde{\rho}} \left(\frac{\mu}{2}\left|\nabla (\e \tilde{\mathbf{u}}) + \nabla (\e \tilde{\mathbf{u}})^t \right|^2 + \lambda |\dv (\e \tilde{\mathbf{u}})|^2\right) \right].
	\end{align*}
	Then we conclude that
	\begin{Lem}\label{non zero est}
		Under the same assumptions as Proposition \ref{a priori estimates}, it holds that
		\begin{equation}\label{non zero inequ}
			\begin{aligned}
				&\frac{d}{d \tau} \norm{(\phi_{\neq}, \psi_{\neq}, \zeta_{\neq})}_{H^1(\Omega_\e)}^2 + c \norm{(\nabla \phi_{\neq}, \nabla \psi_{\neq}, \nabla \zeta_{\neq}, \nabla^2 \psi_{\neq}, \nabla^2 \zeta_{\neq})}_{L^2(\Omega_\e)}^2 \\
				\le &C_1 \left(\e^{-2} + \tau\right)^{-\frac{1}{2}} \norm{( \phi_{\neq}, \psi_{\neq}, \zeta_{\neq})}_{L^2(\Omega_\e)}^2.
			\end{aligned}
		\end{equation}
	\end{Lem}
	\begin{proof}
		The proof of this lemma is divided into the following three steps.
		\subsubsection*{Step 1}
		Multiplying \((\ref{the non-zero equs})_1\) by \(\frac{\mathring{\T}}{\mathring{\rho} \tilde{\rho}} \phi_{\neq}\), \((\ref{the non-zero equs})_2\) by \(\psi_{\neq}\) and \((\ref{the non-zero equs})_3\) by \(\frac{\mathring{\rho}}{\mathring{\T} \tilde{\rho}} \zeta_{\neq}\), one has
		\begin{align}\label{non zero est1}
			\p_\tau \left(\frac{\mathring{\T}}{2\mathring{\rho} \tilde{\rho}} \phi_{\neq}^2 + \frac{|\psi_{\neq}|^2}{2} + \frac{\mathring{\rho}}{2\mathring{\T} \tilde{\rho}} \zeta_{\neq}^2\right) + \frac{\mu}{\tilde{\rho}} \left|\nabla \psi_{\neq}\right|^2 + \frac{\lambda+ \mu}{\tilde{\rho}}  \left|\dv \psi_{\neq}\right|^2 + \frac{\kappa \mathring{\rho}}{\mathring{\T} \tilde{\rho}^2} \left|\dv \zeta_{\neq}\right|^2 = \hat{J}_0 + \dv(...),
		\end{align}
		where
		\begin{align*}
			\hat{J}_0 :=& \left[\left(\frac{\mathring{\T}}{\mathring{\rho} \tilde{\rho}}\right)_{\tau} + \dv \left(\frac{\e \mathring{\mathbf{u}} \mathring{\T}}{\mathring{\rho} \tilde{\rho}}\right)\right] \frac{\phi_{\neq}^2}{2} + \left[\left(\frac{\mathring{\rho}}{\mathring{\T} \tilde{\rho}}\right)_{\tau} + \dv \left(\frac{\e \mathring{\mathbf{u}} \mathring{\rho}}{\mathring{\T} \tilde{\rho}}\right)\right] \frac{\zeta_{\neq}^2}{2} + \frac{\p_{y_1}(\e \mathring{u}_1)}{2} |\psi_{\neq}|^2 + \frac{\mathring{\T}}{\mathring{\rho} \tilde{\rho}} \phi_{\neq} \mathcal{A}_{0\neq} + \frac{\mathring{\rho}}{\mathring{\T} \tilde{\rho}} \zeta_{\neq} \mathcal{A}_{4\neq} \\
			&- \p_{y_1} \left(\frac{1}{\tilde{\rho}}\right) \left(\mu \sum_{j=1}^{3} \p_{y_1} \psi_{j\neq} \psi_{j\neq} +(\lambda+\mu) \dv \psi_{\neq} \psi_{j\neq}\right) - \p_{y_1} \left(\frac{\kappa \mathring{\rho}}{\mathring{\T} \tilde{\rho}^2}\right) \dv \zeta_{\neq} \zeta_{\neq} + \psi_{\neq} \cdot \mathcal{A}_{\neq}.
		\end{align*}
		We need to estimate \(\int_{\Omega_\e} \hat{J}_0 dy\).
		By the Gagliardo-Nirenberg inequality and (\ref{relation1}), it yields that
		\begin{align}\label{non zero est11}
			&\int_{\Omega_\e} \frac{\p_{y_1}(\e \mathring{u}_1)}{2} |\psi_{\neq}|^2 dy = \int_{\Omega_\e} \frac{\p_{y_1} \mathring{\psi}_1 + \p_{y_1}(\e \tilde{u}_1)}{2} |\psi_{\neq}|^2 dy \notag\\
			\le& C\int_{\Omega_\e} \left[\e^{\frac{1}{2}} \left(\delta \e (\e^{-2} + \tau)^{-\frac{3}{2}}\right)^{\frac{1}{4}} \left(\chi^2 \e^{-1} (\e^{-2} + \tau)^{-\frac{3}{2}}\right)^{\frac{1}{4}} + \bar{\delta}(\e^{-2} + \tau)^{-1}\right] |\psi_{\neq}|^2 dy \\
			\le& C \e^{\frac{1}{2}} (\e^{-2} + \tau)^{-\frac{3}{4}}\norm{\psi_{\neq}}_{L^2(\Omega_\e)}^2. \notag
		\end{align}
		Using the Cauchy inequality, the Gagliardo-Nirenberg inequality and (\ref{relation1}), one obtains
		\begin{equation}
			\begin{aligned}
				&\int_{\Omega_\e} - \p_{y_1} \left(\frac{1}{\tilde{\rho}}\right) \left(\mu \sum_{j=1}^{3} \p_{y_1} \psi_{j\neq} \psi_{j\neq} +(\lambda+\mu) \dv \psi_{\neq} \psi_{j\neq}\right) dy \\
				\le& C\delta (\e^{-2} + \tau)^{-\frac{1}{2}} \norm{\psi_{\neq}}_{L^2(\Omega_\e)} \norm{\nabla \psi_{\neq}}_{L^2(\Omega_\e)} \\
				\le& \delta \norm{\nabla \psi_{\neq}}_{L^2(\Omega_\e)}^2+ C \delta (\e^{-2} + \tau)^{-1} \norm{ \psi_{\neq}}_{L^2(\Omega_\e)}^2,
			\end{aligned}
		\end{equation}
		
		and
		\begin{equation}
			\begin{aligned}
				&\int_{\Omega_\e} \p_{y_1} \left(\frac{\kappa \mathring{\rho}}{\mathring{\T} \tilde{\rho}^2}\right) \dv \zeta_{\neq} \zeta_{\neq} dy \\
				\le& \left(\norm{\p_{y_1} (\mathring{\phi}, \mathring{\zeta})}_{L^\infty(\R)} + \norm{\p_{y_1} (\tilde{\rho}, \tilde{\T})}_{L^\infty(\R)}\right) \norm{\zeta_{\neq}}_{L^2(\Omega_\e)} \norm{\nabla \zeta_{\neq}}_{L^2(\Omega_\e)} \\
				\le& \bar{\delta} \norm{\nabla \zeta_{\neq}}_{L^2(\Omega_\e)}^2+ C \bar{\delta} (\e^{-2} + \tau)^{-1} \norm{ \zeta_{\neq}}_{L^2(\Omega_\e)}^2.
			\end{aligned}
		\end{equation}
		
		As for the term \(\int_{\Omega_\e} \frac{\mathring{\T}}{\mathring{\rho} \tilde{\rho}} \phi_{\neq} \mathcal{A}_{0\neq} dy\), we need to make an estimation for \(\mathcal{A}_{0\neq}\).
		Using the Gagliardo-Nirenberg inequality and (\ref{relation1}), one has the following.
		\begin{align*}
			&\norm{\e \mathring{\mathbf{u}} \cdot \nabla \phi_{\neq} - \mathbf{D}_{\neq} (\e \mathbf{u} \cdot \nabla \phi)}_{L^2(\Omega_\e)} \\
			\le& \norm{\nabla \mathring{\phi}}_{L^2(\R)}^{\frac{1}{2}} \norm{\nabla^2 \mathring{\phi}}_{L^2(\R)}^{\frac{1}{2}} \norm{\psi_{\neq}}_{L^2(\Omega_\e)} + \norm{\psi_{\neq}}_{L^2(\Omega_\e)}^{\frac{1}{4}} \norm{\nabla^2 \psi_{\neq}}_{L^2(\Omega_\e)}^{\frac{3}{4}} \norm{\nabla \phi_{\neq}}_{L^2(\Omega_\e)} \\
			\le&  \e^{\frac{1}{2}} \norm{\nabla \mathring{\phi}}_{L^2(\R)}^{\frac{1}{2}} \norm{\nabla^2 \phi}_{L^2(\Omega_\e)}^{\frac{1}{2}} \norm{\psi_{\neq}}_{L^2(\Omega_\e)} + \norm{\psi_{\neq}}_{L^2(\Omega_\e)}^{\frac{1}{4}} \norm{\nabla^2 \psi_{\neq}}_{L^2(\Omega_\e)}^{\frac{3}{4}} \norm{\nabla \phi_{\neq}}_{L^2(\Omega_\e)}.
		\end{align*}
		According to the a priori assumptions (\ref{a priori assumption}), \(\norm{\psi_{\neq}}_{L^2(\Omega_\e)}^{\frac{1}{4}} \norm{\nabla^2 \psi_{\neq}}_{L^2(\Omega_\e)}^{\frac{3}{4}}\) is sufficiently small and enjoys an arbitrarily high polynomial decay rate in time. Thus, our main attention is devoted to the estimate of the first term on the right-hand side of the above inequality. Using Lemma \ref{zero mode est} and the a priori assumptions (\ref{a priori assumption}), one gets
		\begin{align*}
			&\e^{\frac{1}{2}} \norm{\nabla \mathring{\phi}}_{L^2(\R)}^{\frac{1}{2}} \norm{\nabla^2 \phi}_{L^2(\Omega_\e)}^{\frac{1}{2}} \norm{\psi_{\neq}}_{L^2(\Omega_\e)} \\
			\le& C\left[\e^{\frac{1}{2}} \left(\delta\e (\e^{-2} + \tau)^{-\frac{3}{2}}\right)^{\frac{1}{4}} \left(\chi^2 \e^{-1} (\e^{-2} + \tau)^{-\frac{3}{2}}\right)^{\frac{1}{4}}\right] \norm{\psi_{\neq}}_{L^2(\Omega_\e)} \\
			\le& C \e^{\frac{1}{2}} (\e^{-2} + \tau)^{-\frac{3}{4}}  \norm{\psi_{\neq}}_{L^2(\Omega_\e)}.
		\end{align*}
		By a same argument, we find
		\begin{align*}
			&\norm{\mathring{\rho} \nabla \cdot \psi_{\neq} - \mathbf{D}_{\neq} (\rho \nabla \cdot \psi)}_{L^2(\Omega_\e)} \\
			\le& \norm{\nabla \mathring{\psi}}_{L^2(\R)}^{\frac{1}{2}} \norm{\nabla^2 \mathring{\psi}}_{L^2(\R)}^{\frac{1}{2}} \norm{\phi_{\neq}}_{L^2(\Omega_\e)} + \norm{\phi_{\neq}}_{L^2(\Omega_\e)}^{\frac{1}{4}} \norm{\nabla^2 \phi_{\neq}}_{L^2(\Omega_\e)}^{\frac{3}{4}} \norm{\nabla \psi_{\neq}}_{L^2(\Omega_\e)} \\
			\le& C \e^{\frac{1}{2}} (\e^{-2} + \tau)^{-\frac{3}{4}}  \norm{\phi_{\neq}}_{L^2(\Omega_\e)}  
			+ \norm{\phi_{\neq}}_{L^2(\Omega_\e)}^{\frac{1}{4}} \norm{\nabla^2 \phi_{\neq}}_{L^2(\Omega_\e)}^{\frac{3}{4}} \norm{\nabla \psi_{\neq}}_{L^2(\Omega_\e)}.
		\end{align*}
		Moreover, a direct calculation gives 
		\begin{align*}
			&\norm{\psi_{\neq} \cdot \nabla \tilde{\rho}}_{L^2(\Omega_\e)} 
			\le C\bar{\delta} (\e^{-2} + \tau)^{-\frac{1}{2}} \norm{\psi_{\neq}}_{L^2(\Omega_\e)}, \\
			&\norm{\phi_{\neq} \dv (\e \tilde{\mathbf{u}})}_{L^2(\Omega_\e)} \le \norm{\dv (\e \tilde{\mathbf{u}})}_{L^\infty(\R)} \norm{\phi_{\neq}}_{L^2(\Omega_\e)} 
			\le C \bar{\delta} (\e^{-2} + \tau)^{-1} \norm{\phi_{\neq}}_{L^2(\Omega_\e)}.
		\end{align*}
		Combining the above estimates, we obtain
		\begin{align*}
			\norm{\mathcal{A}_{0\neq}}_{L^2(\Omega_\e)} \le 
			C \bar{\delta} (\e^{-2} + \tau)^{-\frac{1}{2}} \norm{(\phi_{\neq}, \psi_{\neq})}_{L^2(\Omega_\e)} + C\norm{(\phi_{\neq}, \psi_{\neq})}_{L^2(\Omega_\e)}^{\frac{1}{4}} \norm{(\nabla \phi_{\neq}, \nabla \psi_{\neq})}_{L^2(\Omega_\e)},
		\end{align*}
		which leads to
		\begin{equation}\label{F_0 term}
			\begin{aligned}
				&\int_{\Omega_\e} \frac{\mathring{\T}}{\mathring{\rho} \tilde{\rho}} \phi_{\neq} \mathcal{A}_{0\neq} dy 
				\le \norm{\phi_{\neq}}_{L^2(\Omega_\e)} \norm{\mathcal{A}_{0\neq}}_{L^2(\Omega_\e)} \\
				\le& C \bar{\delta} (\e^{-2} + \tau)^{-\frac{1}{2}} \norm{(\phi_{\neq}, \psi_{\neq})}_{L^2(\Omega_\e)}^2 + C(\e^{-2} + \tau)^{\frac{1}{2}} \norm{(\phi_{\neq}, \psi_{\neq})}_{L^2(\Omega_\e)}^{\frac{1}{2}} \norm{(\nabla \phi_{\neq}, \nabla \psi_{\neq})}_{L^2(\Omega_\e)}^2.
			\end{aligned}
		\end{equation}
		To estimate the term \(\int_{\Omega_\e} \frac{\mathring{\rho}}{\mathring{\T} \tilde{\rho}} \zeta_{\neq} \mathcal{A}_{4\neq} dy\), it remains to estimate \(\mathcal{A}_{4\neq}\). We carry out the calculation of the term including \(\kappa \mathbf{D}_{\neq} \left[\left(\frac{1}{\rho} - \frac{1}{\tilde{\rho}}\right) \Delta \T\right]\) as
		\begin{align*}
			&\int_{\Omega_\e} \kappa\frac{\mathring{\rho}}{\mathring{\T} \tilde{\rho}} \zeta_{\neq} \mathbf{D}_{\neq} \left[\left(\frac{1}{\rho} - \frac{1}{\tilde{\rho}}\right) \Delta \T\right] dy  \\
			=& - \int_{\Omega_\e} \kappa\frac{\mathring{\rho}}{\mathring{\T} \tilde{\rho}} \zeta_{\neq}
			\mathbf{D}_{\neq} \left[\nabla \left(\frac{1}{\rho} - \frac{1}{\tilde{\rho}}\right) \cdot \nabla \T\right] dy - \int_{\Omega_\e} \kappa\frac{\mathring{\rho}}{\mathring{\T} \tilde{\rho}} \nabla \zeta_{\neq} \cdot
			\mathbf{D}_{\neq} \left[\left(\frac{1}{\rho} - \frac{1}{\tilde{\rho}}\right) \nabla \T\right] dy \\
			& - \int_{\Omega_\e} \kappa \zeta_{\neq} \nabla \left(\frac{\mathring{\rho}}{\mathring{\T} \tilde{\rho}}\right) \cdot \mathbf{D}_{\neq} \left[\left(\frac{1}{\rho} - \frac{1}{\tilde{\rho}}\right) \nabla \T\right] dy \\
			\le& \bar{\delta} \norm{(\nabla \phi_{\neq}, \nabla \zeta_{\neq})}_{L^2(\Omega_\e)}^2 + 
			C  (\e^{-2} + \tau)^{-\frac{1}{2}} \norm{(\phi_{\neq}, \zeta_{\neq})}_{L^2(\Omega_\e)}^2.
		\end{align*}
		The calculation of the remaining terms in \(\mathcal{A}_{4\neq}\) is similar to the calculation of the terms in \(\mathcal{A}_{0\neq}\). Then, we have
		\begin{align}\label{F_4 term}
			\int_{\Omega_\e} \frac{\mathring{\rho}}{\mathring{\T} \tilde{\rho}} \zeta_{\neq} \mathcal{A}_{4\neq} dy 
			\le C (\e^{-2} + \tau)^{-\frac{1}{2}} \norm{(\phi_{\neq}, \psi_{\neq}, \zeta_{\neq})}_{L^2(\Omega_\e)}^2 +\bar{\delta} \norm{(\nabla \phi_{\neq}, \nabla \psi_{\neq}, \nabla \zeta_{\neq})}_{L^2(\Omega_\e)}^2.
		\end{align}
		For the term \(\int_{\Omega_\e} \psi_{\neq} \cdot \mathcal{A}_{\neq} dy\), we first derive an estimate for the term containing \(\mathbf{D}_{\neq} \left[\left(\frac{1}{\rho} - \frac{1}{\tilde{\rho}}\right) (\mu \Delta (\e \mathbf{u})) + (\lambda+\mu) \nabla \dv (\e \mathbf{u})\right]\). It follows from integration by parts that
		\begin{align*}
			&\int_{\Omega_\e} \mathbf{D}_{\neq} \left[\left(\frac{1}{\rho} - \frac{1}{\tilde{\rho}}\right) (\mu \Delta (\e \mathbf{u})) + (\lambda+\mu) \nabla \dv (\e \mathbf{u})\right] \cdot \psi_{\neq} dy \\
			=& \int_{\Omega_\e} \mathbf{D}_{\neq} \left(\frac{1}{\rho} - \frac{1}{\tilde{\rho}}\right) (\mu \Delta \left(\e \mathring{\mathbf{u}}) + (\mu+\lambda) \nabla \dv (\e \mathring{\mathbf{u}})\right) \cdot \psi_{\neq} dy 
			+ \int_{\Omega_\e} \mathbf{D}_{0} \left(\frac{1}{\rho} - \frac{1}{\tilde{\rho}}\right) \left(\mu \Delta (\e \mathbf{u}_{\neq}) + (\mu + \lambda) \nabla \dv (\e \mathbf{u}_{\neq})\right) \cdot \psi_{\neq} dy \\
			&+ \int_{\Omega_\e} \mathbf{D}_{\neq} \left(\frac{1}{\rho} - \frac{1}{\tilde{\rho}}\right) 
			\left(\mu \Delta (\e \mathbf{u}_{\neq}) + (\mu + \lambda) \nabla \dv (\e \mathbf{u}_{\neq})\right) \cdot \psi_{\neq} dy \\
			\le& \norm{\nabla (\e \mathring{\mathbf{u}})}_{L^\infty(\R)} \norm{\mathbf{D}_{\neq} \nabla \left(\frac{1}{\rho} - \frac{1}{\tilde{\rho}}\right)}_{L^2(\Omega_\e)} \norm{\psi_{\neq}}_{L^2(\Omega_\e)} 
			+ \norm{\nabla (\e \mathring{\mathbf{u}})}_{L^\infty(\R)} \norm{\mathbf{D}_{\neq} \left(\frac{1}{\rho} - \frac{1}{\tilde{\rho}}\right)}_{L^2(\Omega_\e)} \norm{\nabla \psi_{\neq}}_{L^2(\Omega_\e)} \\
			&+ \norm{\nabla \mathbf{D}_0 \left(\frac{1}{\rho} - \frac{1}{\tilde{\rho}}\right)}_{L^\infty(\R)} \norm{\nabla \psi_{\neq}}_{L^2(\Omega_\e)} \norm{\psi_{\neq}}_{L^2(\Omega_\e)} 
			+ \norm{\mathbf{D}_0 \left(\frac{1}{\rho} - \frac{1}{\tilde{\rho}}\right)}_{L^\infty(\R)} \norm{\nabla \psi_{\neq}}_{L^2(\Omega_\e)}^2 \\
			&+ \norm{\psi_{\neq}}_{L^\infty(\Omega_\e)} \norm{\nabla \psi_{\neq}}_{L^2(\Omega_\e)} \norm{\mathbf{D}_{\neq} \nabla \left(\frac{1}{\rho} - \frac{1}{\tilde{\rho}}\right)}_{L^2(\Omega_\e)} + \norm{\mathbf{D}_{\neq} \left(\frac{1}{\rho} - \frac{1}{\tilde{\rho}}\right)} \norm{\nabla \psi_{\neq}}_{L^2(\Omega_\e)}^2 \\
			\le& C (\e^{-2} + \tau)^{-\frac{1}{2}} \norm{( \phi_{\neq}, \psi_{\neq})}_{L^2(\Omega_\e)}^2 + \bar{\eta} \norm{(\nabla \phi_{\neq}, \nabla \psi_{\neq})}_{L^2(\Omega_\e)}^2.
		\end{align*}
		For the other terms in \(\mathcal{A}_{\neq}\), we have
		\begin{align*}
			&\norm{\frac{1}{\tilde{\rho}} (\mathring{\T} \nabla \phi_{\neq} + \mathring{\rho} \nabla \zeta_{\neq}) - \frac{\nabla p_{\neq}}{\tilde{\rho}}}_{L^2(\Omega_\e)} \\
			\le& \norm{\frac{1}{\tilde{\rho}} (\p_{y_1} \mathring{\phi} \zeta_{\neq} + \p_{y_1} \tilde{\rho} \zeta_{\neq} + \p_{y_1} \mathring{\zeta} \phi_{\neq} + \p_{y_1} \tilde{\T} \phi_{\neq})}_{L^2(\Omega_\e)} + \sum_{i=1}^{3} \norm{\frac{1}{\tilde{\rho}} (\zeta_{\neq} \p_{y_i} \phi_{\neq} + \phi_{\neq} \p_{y_i} \zeta_{\neq})}_{L^2(\Omega_\e)} \\
			\le& C (\e^{-2} + \tau)^{-\frac{1}{2}} \norm{(\phi_{\neq}, \zeta_{\neq})}_{L^2(\Omega_\e)} + \norm{(\phi_{\neq}, \zeta_{\neq})}_{L^2(\Omega_\e)}^{\frac{1}{4}} \norm{(\nabla^2 \phi_{\neq}, \nabla^2 \zeta_{\neq})}_{L^2(\Omega_\e)}^{\frac{3}{4}} \norm{(\nabla \phi_{\neq}, \nabla \zeta_{\neq})}_{L^2(\Omega_\e)}.
		\end{align*}
		The remaining terms can be treated in a same way as in \(\mathcal{A}_{0\neq}\). Then one obtains
		\begin{align}\label{F term}
			\int_{\Omega_\e} \psi_{\neq} \cdot \mathcal{A}_{\neq} dy \le C (\e^{-2} + \tau)^{-\frac{1}{2}} \norm{(\phi_{\neq}, \psi_{\neq}, \zeta_{\neq})}_{L^2(\Omega_\e)}^2 + \bar{\eta} \norm{(\nabla \phi_{\neq}, \nabla \zeta_{\neq})}_{L^2(\Omega_\e)}^2.
		\end{align}
		Then, we will deal with the term \(\int_{\Omega_\e} \left[\left(\frac{\mathring{\T}}{\mathring{\rho} \tilde{\rho}}\right)_{\tau} + \dv \left(\frac{\e \mathring{\mathbf{u}} \mathring{\T}}{\mathring{\rho} \tilde{\rho}}\right)\right] \frac{\phi_{\neq}^2}{2} dy\). A direct calculation gives that
		\begin{align*}
			&\int_{\Omega_\e} \left(\frac{\mathring{\T}}{\mathring{\rho} \tilde{\rho}}\right)_{\tau} \frac{\phi_{\neq}^2}{2} dy \le C \int_{\Omega_\e} (|\mathring{\T}_{\tau}| + |\mathring{\rho}_{\tau}| + |\tilde{\rho}_\tau|) \frac{\phi_{\neq}^2}{2} dy.
		\end{align*}
		According to \(\mathbf{D}_0 (\ref{scaled NS Eqs})_3\) and the a priori assumptions (\ref{a priori assumption}), it holds that
		\begin{align*}
			&\int_{\Omega_\e} |\mathring{\T}_\tau| \frac{\phi_{\neq}^2}{2} dy = \int_{\Omega_\e} \left|\kappa\mathbf{D}_0 (\frac{\Delta \T}{\rho}) - \mathbf{D}_0 (\e \mathbf{u} \cdot \nabla \T) - \mathbf{D}_0 (\T \dv(\e \mathbf{u}))+\Do\left(\frac{\mathbb{S(\e\mathbf u):\nabla(\e\mathbf u)}}{\rho}\right)\right| \frac{\phi_{\neq}^2}{2} dy \\
			\le& C (\norm{\p_{y_1}^2 \mathring{\T}}_{L^\infty} + \norm{\mathbf{D}_0 \left(\Delta \T_{\neq} (\frac{1}{\rho})_{\neq}\right)}_{L^\infty} + \norm{\mathbf{D}_0 (\e \mathbf{u} \cdot \nabla \T)}_{L^\infty} + \norm{\mathbf{D}_0 (\T \dv(\e \mathbf{u}))}_{L^\infty}+\e^2\norm{\nabla \mathbf u}_{L^\infty}^2) \frac{\phi_{\neq}^2}{2} dy \\
			\le& C \left(\norm{\nabla^2 \T}_{L^2}^{\frac{1}{2}} \norm{\nabla^3 \T}_{L^2}^{\frac{1}{2}} + \norm{(\e \mathbf{u}, \nabla \T)}_{L^2}^{\frac{1}{2}} \norm{(\nabla (\e \mathbf{u}), \nabla^2 \T)}_{L^2}^{\frac{1}{2}} + \e\norm{ \nabla \mathbf{u}}_{L^2}^{\frac{1}{2}} \norm{ \nabla^2\mathbf{u}}_{L^2}^{\frac{1}{2}}\right)\norm{\phi_{\neq}}_{L^2(\Omega_\e)}^2 \\
			\le& C  (\e^{-2} + \tau)^{-\frac{1}{2}} \norm{\phi_{\neq}}_{L^2(\Omega_\e)}^2,
		\end{align*}
		and the terms \(\int_{\Omega_\e} |\mathring{\rho}_\tau| \frac{\phi_{\neq}^2}{2} dy\), \(\int_{\Omega_\e} |\tilde{\rho}_\tau| \frac{\phi_{\neq}^2}{2} dy\) and \(\int_{\Omega_\e} \dv \left(\frac{\e \mathring{\mathbf{u}} \mathring{\T}}{\mathring{\rho} \tilde{\rho}}\right) \frac{\phi_{\neq}^2}{2} dy\) satisfy the same estimates.
		Then we obtain
		\begin{align}\label{J_0 1st}
			\int_{\Omega_\e} \left[\left(\frac{\mathring{\T}}{\mathring{\rho} \tilde{\rho}}\right)_{\tau} + \dv \left(\frac{\e \mathring{\mathbf{u}} \mathring{\T}}{\mathring{\rho} \tilde{\rho}}\right)\right] \frac{\phi_{\neq}^2}{2} dy \le C (\e^{-2} + \tau)^{-\frac{1}{2}} \norm{\phi_{\neq}}_{L^2(\Omega_\e)}^2.
		\end{align}
		By the similar argument, we have
		\begin{align}\label{J_0 2st}
			\int_{\Omega_\e} \left[\left(\frac{\mathring{\rho}}{\mathring{\T} \tilde{\rho}}\right)_{\tau} + \dv \left(\frac{\e \mathring{\mathbf{u}} \mathring{\rho}}{\mathring{\T} \tilde{\rho}}\right)\right] \frac{\zeta_{\neq}^2}{2} dy \le C (\e^{-2} + \tau)^{-\frac{1}{2}} \norm{\zeta_{\neq}}_{L^2(\Omega_\e)}^2.
		\end{align}
		Combining (\ref{non zero est11})-(\ref{J_0 2st}), one gets
		\begin{align}\label{hat J_0 est}
			\int_{\Omega_\e} \hat{J}_0 dy \le C (\e^{-2} + \tau)^{-\frac{1}{2}} \norm{(\phi_{\neq}, \psi_{\neq}, \zeta_{\neq})}_{L^2(\Omega_\e)}^2 + \bar{\eta} \norm{(\nabla \phi_{\neq},\nabla \psi_{\neq}, \nabla \zeta_{\neq})}_{L^2(\Omega_\e)}^2.
		\end{align}
		Integration (\ref{non zero est1}) with respect to \(y\) and using (\ref{hat J_0 est}) yield that
		\begin{align}\label{non zero low order est}
			\frac{d}{d \tau} \norm{(\phi_{\neq}, \psi_{\neq}, \zeta_{\neq})}_{L^2(\Omega_\e)}^2 + \norm{\nabla (\psi_{\neq}, \zeta_{\neq})}_{L^2(\Omega_\e)}^2 
			\le C (\e^{-2} + \tau)^{-\frac{1}{2}} \norm{(\phi_{\neq}, \psi_{\neq}, \zeta_{\neq})}_{L^2(\Omega_\e)}^2 + \bar{\eta} \norm{\nabla \phi_{\neq}}_{L^2(\Omega_\e)}^2.
		\end{align}

		\subsubsection*{Step 2}
		Next, we will estimate \(\norm{\nabla \phi_{\neq}}_{L^2(\Omega_\e)}\). Multiplying \((\ref{the non-zero equs})_2\) by \(\tilde{\rho} \nabla \phi_{\neq}\), \((\ref{the non-zero equs})_3\) by \(\frac{2\mu+ \lambda}{\tilde{\rho}} \nabla \phi_{\neq}\), one finds
		\begin{align}\label{nabla phi_neq equality}
			\p_\tau \left(\frac{2\mu + \lambda}{2\mathring{\rho}} |\nabla \phi_{\neq}|^2 + \tilde{\rho} \psi_{\neq} \cdot \nabla \phi_{\neq}\right) + \mathring{\T} |\nabla \phi_{\neq}|^2 = \hat{J}_1 + \dv(...),
		\end{align}
		where
		\begin{align*}
			\hat{J}_1 =& - \tilde{\rho} \nabla \phi_{\neq} \cdot (\e \mathring{\mathbf{u}}) \cdot \nabla \psi_{\neq} -\mathring{\rho} \nabla \phi_{\neq} \cdot \nabla \zeta_{\neq} + \p_\tau \tilde{\rho} \psi_{\neq} \cdot \nabla \phi_{\neq} +\dv(\tilde{\rho} \psi_{\neq}) (\e \mathring{\mathbf{u}}) \cdot \nabla \phi_{\neq}
			+ \dv (\tilde{\rho} \psi_{\neq}) \mathring{\rho} \dv \psi_{\neq} \\
			&+ \p_\tau \left(\frac{2\mu+\lambda}{2\mathring{\rho}}\right) |\nabla \phi_{\neq}|^2 - (2\mu+\lambda) \frac{\p_{y_1}(\e \mathring{\mathbf{u}})}{\mathring{\rho}} \cdot \nabla \phi_{\neq} \p_{y_1}\phi_{\neq} + (2\mu+\lambda) \p_{y_1} \left(\frac{\e \mathring{u}_1}{\mathring{\rho}}\right) |\nabla \phi_{\neq}|^2 + \tilde{\rho} \psi_{\neq} \cdot \nabla \mathcal{A}_{0\neq} \\
			&+ \frac{2\mu+\lambda}{\mathring{\rho}} \nabla \phi_{\neq} \cdot \nabla \mathcal{A}_{0\neq} - (2\mu+\lambda) \frac{\p_{y_1}\mathring{\rho} \dv \psi_{\neq}}{\mathring{\rho}} \p_{y_1} \phi_{\neq} + \tilde{\rho} \mathcal{A}_{\neq} \cdot \nabla \phi_{\neq}.
		\end{align*}
		Here we have used the fact that
		\begin{align*}
			&\tilde{\rho} \p_\tau \psi_{\neq} \cdot \nabla \phi_{\neq} = \p_\tau(\tilde{\rho} \psi_{\neq} \cdot \nabla \phi_{\neq}) - \p_\tau \tilde{\rho}\psi_{\neq} \cdot \nabla
			\phi_{\neq} - \tilde{\rho} \psi_{\neq} \cdot \nabla \p_\tau \phi_{\neq} \\
			=& \p_\tau(\tilde{\rho} \psi_{\neq} \cdot \nabla \phi_{\neq}) - \p_\tau \tilde{\rho} \psi_{\neq} \cdot \nabla \phi_{\neq} - \dv (\tilde{\rho} \psi_{\neq}) (\e \mathring{\mathbf{u}}) \cdot \nabla \phi_{\neq} - \dv (\tilde{\rho} \psi_{\neq}) \mathring{\rho} \dv \psi_{\neq} - \tilde{\rho} \psi_{\neq} \cdot \nabla \mathcal{A}_{0\neq} + \dv(...).
		\end{align*}
		Using the a priori assumptions (\ref{a priori assumption}) and the Cauchy inequality, we get
		\begin{align}\label{int of J_1}
			\int_{\Omega_\e} \hat{J}_1 dy \le C \bar{\eta} \norm{\nabla \phi_{\neq}}_{L^2(\Omega_\e)}^2 + C \left(\norm{(\nabla \psi_{\neq}, \nabla \zeta_{\neq})}_{L^2(\Omega_\e)}^2 + \bar{\eta} \norm{(\nabla^2 \psi_{\neq}, \nabla^2 \zeta_{\neq})}_{L^2(\Omega_\e)}^2\right).
		\end{align}
		Considering \(\int_{\Omega_\e} (\ref{nabla phi_neq equality}) dy\) and using (\ref{int of J_1}), we have
		\begin{equation}\label{nabla phi_neq est}
			\begin{aligned}
				&\frac{d}{d \tau} \int_{\Omega_\e} \left(\frac{2\mu + \lambda}{2\mathring{\rho}} |\nabla \phi_{\neq}|^2 + \tilde{\rho} \psi_{\neq} \cdot \nabla \phi_{\neq}\right) dy + \int_{\Omega_\e} \frac{\mathring{\T}}{2} |\nabla \phi_{\neq}|^2 dy \\
				\le& C \left(\norm{(\nabla \psi_{\neq}, \nabla \zeta_{\neq})}_{L^2(\Omega_\e)}^2 + \bar{\eta} \norm{(\nabla^2 \psi_{\neq}, \nabla^2 \zeta_{\neq})}_{L^2(\Omega_\e)}^2\right).
			\end{aligned}
		\end{equation}
		
		\subsubsection*{Step 3}
		Then, we will estimate \(\norm{\nabla (\psi_{\neq}, \zeta_{\neq})}_{L^2(\Omega_\e)}\). Taking
		\(\nabla (\ref{the non-zero equs})_1 \times \frac{\mathring{\T}}{\mathring{\rho} \tilde{\rho}} \nabla \phi_{\neq} + \nabla (\ref{the non-zero equs})_2 \times \nabla \psi_{\neq} + \nabla (\ref{the non-zero equs})_3 \times \frac{\mathring{\rho}}{\mathring{\T} \tilde{\rho}} \nabla \zeta_{\neq}\), we find that
		\begin{equation}\label{non zero est2}
			\begin{aligned}
				&\p_\tau \left(\frac{\mathring{\T}}{2\mathring{\rho} \tilde{\rho}} |\nabla \phi_{\neq}|^2 + \frac{|\nabla\psi_{\neq}|^2}{2} +\frac{\mathring{\rho}}{2\mathring{\T} \tilde{\rho}} |\nabla \zeta_{\neq}|^2\right) + \frac{\mu}{\tilde{\rho}} |\Delta \psi_{\neq}|^2 + \frac{\mu+ \lambda}{\tilde{\rho}} |\nabla \dv \psi_{\neq}|^2 + \frac{\kappa \mathring{\rho}}{\tilde{\rho}^2 \mathring{\T}} |\Delta \zeta_{\neq}|^2 \\
				=& \hat{J}_2 + \dv(...),
			\end{aligned}
		\end{equation}
		where
		\begin{align*}
			\hat{J}_2 =& \left[\left(\frac{\mathring{\T}}{\mathring{\rho} \tilde{\rho}}\right)_{\tau} + \dv \left(\frac{\e \mathring{\mathbf{u}} \mathring{\T}}{\mathring{\rho} \tilde{\rho}}\right)\right] \frac{|\nabla \phi_{\neq}|^2}{2} + \left[\left(\frac{\mathring{\rho}}{\mathring{\T} \tilde{\rho}}\right)_{\tau} + \dv \left(\frac{\e \mathring{\mathbf{u}} \mathring{\rho}}{\mathring{\T} \tilde{\rho}}\right)\right] \frac{|\nabla \zeta_{\neq}|^2}{2} + \frac{\p_{y_1} (\e \mathring{u}_1)}{2} |\nabla \psi_{\neq}|^2 \\
			&- \frac{\mathring{\T} \p_{y_1}\mathring{\rho}}{\mathring{\rho} \tilde{\rho}} \dv \psi_{\neq} \p_{y_1} \phi_{\neq}
			- \frac{\mathring{\rho} \p_{y_1}\mathring{\T}}{\mathring{\T} \tilde{\rho}} \dv \psi_{\neq} \p_{y_1} \zeta_{\neq} 
			- \frac{\mathring{\T}}{\mathring{\rho} \tilde{\rho}} \p_{y_1} (\e \mathring{\mathbf{u}}) \cdot \nabla \phi_{\neq} \p_{y_1} \phi_{\neq} - \p_{y_1}(\e \mathring{\mathbf{u}}) \cdot \nabla \psi_{\neq} \cdot \p_{y_1} \psi_{\neq} \\
			&-\frac{\mathring{\rho}}{\mathring{\T} \tilde{\rho}} \p_{y_1} (\e \mathring{\mathbf{u}}) \cdot \nabla \zeta_{\neq} \p_{y_1}\zeta_{\neq} - \p_{y_1} \left(\frac{\mathring{\T}}{\tilde{\rho}}\right) \p_{y_1} \psi_{\neq} \cdot \nabla \phi_{\neq} + \p_{y_1} \left(\frac{\mathring{\T}}{\tilde{\rho}}\right) \nabla \psi_{1\neq} \cdot \nabla \phi_{\neq} + \frac{\mathring{\T}}{\mathring{\rho} \tilde{\rho}} \nabla \phi_{\neq} \cdot \nabla \mathcal{A}_{0\neq} \\
			&- \p_{y_1} \left(\frac{\mathring{\rho}}{\tilde{\rho}}\right) \p_{y_1} \psi_{1\neq} \cdot \nabla \zeta_{\neq} + \p_{y_1} \left(\frac{\mathring{\rho}}{\tilde{\rho}}\right) \nabla \psi_{1\neq} \cdot \nabla \zeta_{\neq} - \p_{y_1} \left(\frac{\mu+\lambda}{\tilde{\rho}}\right) \dv \psi \Delta \psi_{1\neq} + \nabla \psi_{\neq} \cdot \mathcal{A}_{\neq} \\
			&- \p_{y_1} \left(\frac{\mu+\lambda}{\tilde{\rho}}\right) \dv \psi_{\neq} \p_{y_1} \dv \psi_{\neq} - \frac{\mu}{\tilde{\rho}} \p_{y_1} \left(\frac{\mathring{\rho}}{\mathring{\T} \tilde{\rho}}\right) \p_{y_1} \zeta_{\neq} \Delta \zeta_{\neq} + \frac{\mathring{\rho}}{\mathring{\T} \tilde{\rho}} \nabla \zeta_{\neq} \cdot \nabla \mathcal{A}_{4\neq}.
		\end{align*}
		Using integration by parts, we have
		\begin{align}\label{int of J_2}
			\int_{\Omega_\e} \hat{J}_2 dy \le C \bar{\eta} \norm{(\nabla \phi_{\neq}, \nabla \psi_{\neq}, \nabla \zeta_{\neq})}_{L^2(\Omega_\e)}^2
			+ C \bar{\eta} \norm{(\nabla^2 \psi_{\neq}, \nabla^2 \zeta_{\neq})}_{L^2(\Omega_\e)}^2.
		\end{align}
		Tnen, integrating (\ref{non zero est2}) over \(\Omega_\e\) and applying (\ref{int of J_2}) yield that
		\begin{align}\label{non zero high order est}
			\frac{d}{d \tau} \norm{(\nabla \phi_{\neq}, \nabla \psi_{\neq}, \nabla \zeta_{\neq})}_{L^2(\Omega_\e)}^2  + \norm{(\nabla^2 \psi_{\neq}, \nabla^2 \zeta_{\neq})}_{L^2(\Omega_\e)}^2 \le C \bar{\eta} \norm{(\nabla \phi_{\neq}, \nabla \psi_{\neq}, \nabla \zeta_{\neq})}_{L^2(\Omega_\e)}^2.
		\end{align}
		Combining (\ref{non zero low order est}), (\ref{nabla phi_neq est}) and (\ref{non zero high order est}), we complete the proof of Lemma \ref{non zero est}.
	\end{proof}
	
	\subsection*{Decay rate for the non-zero modes}
	
	\begin{Lem}\label{non-zero est for well}
		Under the same assumptions as Proposition \ref{a priori estimates}, it holds that
		\begin{equation}
			\begin{aligned}
				& \norm{(\phi_{\neq}, \psi_{\neq}, \zeta_{\neq})}_{H^1}^2 
				\le C\e^{-3} e^{-\frac{c}{2} \e^2 \tau - c_0 \e^{-2}}, \qquad \qquad\qquad \qquad\qquad \quad~~\text{if}~~ \tau \ge \left(\frac{2C_1}{c \e^2}\right)^2 - \e^{-2}, \\
				&\norm{(\phi_{\neq}, \psi_{\neq}, \zeta_{\neq})}_{H^1}^2 \le C  \e^{-2N-3} (\e^{-2} + \tau)^{-\frac{N}{2}} e^{-c_0 \e^{-2}},~ \forall N  \in \mathbb{Z}^+, \qquad \text{if}~~ \tau < \left(\frac{2C_1}{c \e^2}\right)^2 - \e^{-2}.
			\end{aligned}
		\end{equation}
		Here, \(c_0>0\) is chosen sufficiently large.
	\end{Lem}
	
	\begin{proof}
		For notational simplicity, \(c_0\) denotes a positive constant independent of \(\varepsilon\) that may decrease from line to line.
		According to Lemma \ref{non zero est}, we find that
		\begin{align*}
			&\frac{d}{d \tau} \norm{(\phi_{\neq}, \psi_{\neq}, \zeta_{\neq})}_{H^1(\Omega_\e)}^2 
			+ c \norm{(\nabla \phi_{\neq}, \nabla \psi_{\neq}, \nabla \zeta_{\neq}, \nabla^2 \psi_{\neq}, \nabla^2 \zeta_{\neq})}_{L^2(\Omega_\e)}^2 \\
			\le& C_1 \left(\e^{-2} + \tau\right)^{-\frac{1}{2}} \norm{(\phi_{\neq},  \psi_{\neq}, \zeta_{\neq})}_{L^2(\Omega_\e)}^2 \\
			\le& C_1 \e^{-2} \left(\e^{-2} + \tau\right)^{-\frac{1}{2}} \norm{(\nabla \phi_{\neq}, \nabla \psi_{\neq}, \nabla \zeta_{\neq})}_{L^2(\Omega_\e)}^2.
		\end{align*}
		Set \(\tau_* = \left(\frac{2C_1 }{c \e^2}\right)^2 - \e^{-2}\), which is nonnegative for sufficiently small \(\varepsilon\). We first estimate the solution on \([0,\tau_*]\). According to 
		\begin{align*}
			\frac{d}{d \tau} \norm{(\phi_{\neq}, \psi_{\neq}, \zeta_{\neq})}_{H^1(\Omega_\e)}^2 
			\le C_1 \left(\e^{-2} + \tau\right)^{-\frac{1}{2}} \norm{(\phi_{\neq},  \psi_{\neq}, \zeta_{\neq})}_{L^2(\Omega_\e)}^2,
		\end{align*}
		it holds that 
		\begin{align*}
			\norm{(\phi_{\neq}, \psi_{\neq}, \zeta_{\neq})}_{H^1(\Omega_\e)}^2 \le& \norm{(\phi_{\neq}, \psi_{\neq}, \zeta_{\neq})(0)}_{H^1(\Omega_\e)}^2 e^{\frac{4C_1^2}{c} \e^{-2}} \\
			\le& C \e^{-3} e^{-\left(c_0 - \frac{4C_1^2}{c}\right) \e^{-2}}, \qquad \text{for}~~ 0 \le \tau \le \tau_*.
		\end{align*}
		When \(\tau \ge \tau_*\), it is easy to check that
		\(C_1 \e^{-2} \left(\e^{-2} + \tau\right)^{-\frac{1}{2}} \le \frac{c}{2}\). Then we have
		\begin{align*}
			\frac{d}{d \tau} \norm{(\phi_{\neq}, \psi_{\neq}, \zeta_{\neq})}_{H^1(\Omega_\e)}^2 + \frac{c}{2} \norm{(\nabla \phi_{\neq}, \nabla \psi_{\neq}, \nabla \zeta_{\neq}, \nabla^2 \psi_{\neq}, \nabla^2 \zeta_{\neq})}_{L^2(\Omega_\e)}^2 \le 0,
		\end{align*}
		which leads to
		\begin{align*}
			&\frac{d}{d \tau} \norm{(\phi_{\neq}, \psi_{\neq}, \zeta_{\neq})}_{H^1(\Omega_\e)}^2 + \frac{c}{2} \e^2 \norm{(\phi_{\neq}, \psi_{\neq}, \zeta_{\neq})}_{H^1(\Omega_\e)}^2 \le 0.
		\end{align*}
		Applying the Grönwall's inequality to the above inequality, we obtain
		\begin{align}
			\norm{(\phi_{\neq}, \psi_{\neq}, \zeta_{\neq})}_{H^1(\Omega_\e)}^2 \le& \norm{(\phi_{\neq}, \psi_{\neq}, \zeta_{\neq})(\tau_*)}_{H^1(\Omega_\e)}^2 e^{-\frac{c}{2}\e^2 (\tau- \tau_*)} \notag\\
			\le& C \e^{-3} e^{-\left(c_0 - \frac{4C_1^2}{c}\right) \e^{-2}} e^{-\frac{c}{2}\e^2 (\tau- \tau_*)} \\
			\le& C \e^{-3} e^{-\frac{c}{2} \e^2 \tau - \left(c_0 - \frac{6C_1^2}{c}\right)s^{-2}}. \notag
		\end{align}
		The constant \(c_0\) in the initial smallness assumption (\ref{initial data}) is chosen larger than \(6C_1^2/c\). After decreasing it if necessary, we continue to denote the resulting positive exponent constant by \(c_0\).
		Then we have
		\begin{align}
			\norm{(\phi_{\neq}, \psi_{\neq}, \zeta_{\neq})}_{H^1(\Omega_\e)}^2 \le& C \e^{-3} e^{-\frac{c}{2} \e^2 \tau - c_0 s^{-2}}.
		\end{align}
		
		When \(\tau < \left(\frac{2C_1}{c \e^2}\right)^2 - \e^{-2}\), we find \((\e^{-2} + \tau)^{\frac{1}{2}} < \frac{2C_1}{c \e^2}\).
		The Grönwall's inequality leads to 
		\begin{align*}
			&\norm{(\phi_{\neq}, \psi_{\neq}, \zeta_{\neq})(\tau)}_{H^1(\Omega_\e)}^2 + \int_{0}^{\tau} \norm{(\nabla \phi_{\neq}, \nabla \psi_{\neq}, \nabla \zeta_{\neq}, \nabla^2 \psi_{\neq}, \nabla^2 \zeta_{\neq})}_{L^2(\Omega_\e)}^2 ds \\
			\le& \norm{(\phi_{\neq}, \psi_{\neq}, \zeta_{\neq})(0)}_{H^1(\Omega_\e)}^2 e^{2C_1  ((\e^{-2} + \tau)^{\frac{1}{2}} - \e^{-1})} \\
			\le& \norm{(\phi_{\neq}, \psi_{\neq}, \zeta_{\neq})(0)}_{H^1(\Omega_\e)}^2  e^{\frac{4C_1^2}{c \e^2} - 2C_1 \bar{\eta}\e^{-1}}.
		\end{align*}
		%Suppose \(\frac{4C_1^2}{c} (\chi+ \sqrt{\delta})^2 < \frac{1}{2}\), then we have \(\frac{4C_1^2}{c} \bar{\eta}^2 <\frac{3}{4}\) when \(\e\) and \(\eta\) are sufficiently small. 
		According to (\ref{initial data}), it reaches that
		\begin{align}\label{non zero decay0}
			\norm{(\phi_{\neq}, \psi_{\neq}, \zeta_{\neq})(\tau)}_{H^1(\Omega_\e)}^2 \le C\e^{-3} e^{-c_0 \e^{-2}}, \quad \int_{0}^{\tau} \norm{(\nabla \phi_{\neq}, \nabla \psi_{\neq}, \nabla \zeta_{\neq}, \nabla^2 \psi_{\neq}, \nabla^2 \zeta_{\neq})}_{L^2(\Omega_\e)}^2 ds \le C \e^{-3}e^{-c_0 \e^{-2}}.
		\end{align}
		Here \(c_0 \) satisfies $c_0>\frac{4C_1^2}{c}$.
		%which may be decreased if necessary.
		Taking \((\e^{-2} + \tau)^{\frac{1}{2}} (\ref{non zero inequ})\) and using (\ref{non zero decay0}), we have
		\begin{align*}
			&\frac{d}{d \tau}\left[(\e^{-2} + \tau)^{\frac{1}{2}}\norm{(\phi_{\neq}, \psi_{\neq}, \zeta_{\neq})}_{H^1(\Omega_\e)}^2\right] + c(\e^{-2} + \tau)^{\frac{1}{2}} 
			\norm{(\nabla \phi_{\neq}, \nabla \psi_{\neq}, \nabla \zeta_{\neq}, \nabla^2 \psi_{\neq}, \nabla^2 \zeta_{\neq})}_{L^2(\Omega_\e)}^2 \\
			%\le&\frac{1}{2}  (\e^{-2} + \tau)^{-\frac{1}{2}} \norm{(\nabla \phi_{\neq}, \nabla \varphi_{\neq}, \nabla \zeta_{\neq})}_{H^1(\Omega_\e)}^2+ C_1 \bar{\eta} \e^{-2} \norm{(\nabla \phi_{\neq}, \nabla \varphi_{\neq}, \nabla \zeta_{\neq})}_{L^2(\Omega_\e)}^2 \\
			%\le& C (\e^{-2} + \tau)^{-\frac{1}{2}} e^{-c_0 \e^{-3}}
			%+ C_1 \bar{\eta} \e^{-2} \norm{(\nabla \phi_{\neq}, \nabla \varphi_{\neq}, \nabla \zeta_{\neq})}_{L^2(\Omega_\e)}^2 \\
			\le& C (\e^{-2} + \tau)^{-\frac{1}{2}} e^{-c_0 \e^{-2}}
			+ C_1 \bar{\eta} \e^{-2} \norm{(\nabla \phi_{\neq}, \nabla \psi_{\neq}, \nabla \zeta_{\neq})}_{L^2(\Omega_\e)}^2.
		\end{align*}
		Integrating the above inequality on \([0, \tau]\), we obtain
		\begin{align*}
			&(\e^{-2} + \tau)^{\frac{1}{2}}\norm{(\phi_{\neq}, \psi_{\neq}, \zeta_{\neq})}_{H^1(\Omega_\e)}^2 + c \int_{0}^{\tau} (\e^{-2} + s)^{\frac{1}{2}} 
			\norm{(\nabla \phi_{\neq}, \nabla \psi_{\neq}, \nabla \zeta_{\neq}, \nabla^2 \psi_{\neq}, \nabla^2 \zeta_{\neq})}_{L^2(\Omega_\e)}^2 ds \\
			\le&\e^{-1} \norm{(\phi_{\neq}, \psi_{\neq}, \zeta_{\neq})(0)}_{H^1(\Omega_\e)}^2 + 2C e^{-c_0 \e^{-2}} \left((\e^{-2} + \tau)^{\frac{1}{2}} - \e^{-1}\right) + C_1 \e^{-2} e^{-c_0 \e^{-2}} \\
			\le& \e^{-1} \norm{(\phi_{\neq}, \psi_{\neq}, \zeta_{\neq})(0)}_{H^1(\Omega_\e)}^2 + C \e^{-2} e^{-c_0 \e^{-2}} \\
			\le& C \e^{-5} e^{-c_0 \e^{-2}}.
		\end{align*}
		This leads to 
		\begin{equation}\label{non zero decay1}
			\begin{aligned}
				&\norm{(\phi_{\neq}, \psi_{\neq}, \zeta_{\neq})}_{H^1(\Omega_\e)}^2 \le C \e^{-2} (\e^{-2} + \tau)^{-\frac{1}{2}} e^{-c_0 \e^{-2}},\\
				&\int_{0}^{\tau} (\e^{-2} + s)^{\frac{1}{2}} 
				\norm{(\nabla \phi_{\neq}, \nabla^2 \psi_{\neq}, \nabla^2 \zeta_{\neq})}_{L^2(\Omega_\e)}^2 ds \le C \e^{-2} e^{-c_0 \e^{-2}}.
			\end{aligned}
		\end{equation}
		Taking \((\e^{-2} + \tau) (\ref{non zero inequ})\) and using (\ref{non zero decay1}), we have
		\begin{align*}
			&\frac{d}{d \tau} \left[(\e^{-2} + \tau) \norm{(\phi_{\neq}, \psi_{\neq}, \zeta_{\neq})}_{H^1(\Omega_\e)}^2\right] + c(\e^{-2} + \tau) \norm{(\nabla \phi_{\neq}, \nabla \psi_{\neq}, \nabla \zeta_{\neq}, \nabla^2 \psi_{\neq}, \nabla^2 \zeta_{\neq})}_{L^2(\Omega_\e)}^2 \\ 
			\le& \norm{(\phi_{\neq}, \psi_{\neq}, \zeta_{\neq})}_{H^1(\Omega_\e)}^2 + C_1 \e^{-2} \left(\e^{-2} + \tau\right)^{\frac{1}{2}} \norm{(\nabla \phi_{\neq}, \nabla \psi_{\neq}, \nabla \zeta_{\neq})}_{L^2(\Omega_\e)}^2.
		\end{align*}
		Integrating the above inequality on \([0, \tau]\), we find
		\begin{align*}
			& (\e^{-2} + \tau) \norm{(\phi_{\neq}, \psi_{\neq}, \zeta_{\neq})}_{H^1(\Omega_\e)}^2 + c \int_{0}^{\tau} (\e^{-2} + s) 
			\norm{(\nabla \phi_{\neq}, \nabla \psi_{\neq}, \nabla \zeta_{\neq}, \nabla^2 \psi_{\neq}, \nabla^2 \zeta_{\neq})}_{L^2(\Omega_\e)}^2 ds\\
			\le& \e^{-2} \norm{(\phi_{\neq}, \psi_{\neq}, \zeta_{\neq})(0)}_{H^1(\Omega_\e)}^2 + \int_{0}^{\tau}  \norm{(\phi_{\neq}, \psi_{\neq}, \zeta_{\neq})}_{H^1(\Omega_\e)}^2  ds \\
			&+ C_1  \e^{-2} \int_{0}^{s} \left(\e^{-2} + s\right)^{\frac{1}{2}} \norm{(\nabla \phi_{\neq}, \nabla \psi_{\neq}, \nabla \zeta_{\neq})}_{L^2(\Omega_\e)}^2 ds \\
			\le& C\e^{-7} e^{-c_0 \e^{-2}},
		\end{align*}
		which leads to 
		\begin{equation}
			\begin{aligned}
				&\norm{(\phi_{\neq}, \psi_{\neq}, \zeta_{\neq})}_{H^1(\Omega_\e)}^2 \le C \e^{-7} (\e^{-2} + \tau)^{-1} e^{-c_0 \e^{-2}}, \\
				&\int_{0}^{\tau} (\e^{-2} + s) 
				\norm{(\nabla \phi_{\neq}, \nabla \psi_{\neq}, \nabla \zeta_{\neq}, \nabla^2 \psi_{\neq}, \nabla^2 \zeta_{\neq})}_{L^2(\Omega_\e)}^2 ds \le C_N \e^{-7} e^{-c_0 \e^{-2}}.
			\end{aligned}
		\end{equation}
		Continuing in this manner, we can prove the following conclusion:
		\begin{align}
			\norm{(\phi_{\neq}, \psi_{\neq}, \zeta_{\neq})}_{H^1(\Omega_\e)}^2 \le C \e^{-2N-3} (\e^{-2} + \tau)^{-\frac{N}{2}} e^{-c_0 \e^{-2}}, \qquad \forall N \in \mathbb{Z}^+,
		\end{align}
		where \(C_N\) is independent of \(\varepsilon\) and \(\tau\).  Thus, we complete the proof of Lemma \ref{non-zero est for well}.
	\end{proof}
	
	\subsection{Higher order estimates}
	
	The main purpose of this section is to establish energy estimates for the original perturbation equations \eqref{the perturbation equs} and to justify the a priori assumptions. Recall the system of the perturbation variables \eqref{the perturbation equs}
	\begin{equation}\label{the perturbation equs1}
		\left\{
		\begin{aligned}
			&\p_\tau \phi + \e \mathbf{u} \cdot \nabla \phi + \rho \dv \psi = \mathcal{R}_1, \\
			&\rho \p_\tau \psi + \e \rho \mathbf{u} \cdot \nabla \psi + (\T \nabla \phi + \rho \nabla \zeta) = \mu \Delta \psi + (\mu + \lambda) \nabla \dv \psi + \mathcal{R}_2, \\
			&\rho \p_\tau \zeta + \e \rho \mathbf{u} \cdot \nabla \zeta + \rho \T \dv \psi = \kappa \Delta \zeta + \mathcal{R}_3,
		\end{aligned}
		\right.
	\end{equation}
	The energy estimates involving the zeroth- and first-order derivatives have been previously obtained. Consequently, the remaining task in verifying the a priori assumptions is to derive suitable estimates for the second- and third-order derivatives. In this subsection, we denote \(\p_{y_i}\) by \(\p_i\) for conveniece.
	
	For \(i, j =1,2,3\), applying \(\p_j \p_i\) to \((\ref{the perturbation equs1})_1\), \(\p_j\) to the i-th component of \((\ref{the perturbation equs1})_2\) and \(\p_j\) to \((\ref{the perturbation equs1})_3\), we find
	\begin{equation}\label{high-order equs}
		\left\{
		\begin{aligned}
			&\p_\tau \p_{ij} \phi + \e \mathbf{u} \cdot \nabla \p_{ij} \phi + \rho \dv \p_{ij} \psi = \p_{ij} \mathcal{R}_1 + \mathcal{Q}_1, \\
			&\rho \p_\tau \p_j \psi + \e \rho \mathbf{u} \cdot \nabla \p_j \psi + (\T \nabla \p_j \phi + \rho \nabla \p_j \zeta) = \mu \Delta \p_j \psi + (\mu + \lambda) \nabla \dv \p_j \psi + \p_j \mathcal{R}_2 + \mathcal{Q}_2,\\
			&\rho \p_\tau \p_j \zeta + \e\rho \mathbf{u} \cdot \nabla \p_j \zeta + \rho \T \dv \p_j \psi = \kappa \Delta \p_j \zeta + \p_j \mathcal{R}_3 + \mathcal{Q}_3,
		\end{aligned}
		\right.
	\end{equation}
	where 
	\begin{align*}
		&\mathcal{Q}_1 := -\p_{ij} (\e \mathbf{u} \cdot \nabla \phi+\rho \dv \psi) + \e \mathbf{u} \cdot \nabla \p_{ij} \phi + \rho \dv \p_{ij} \psi, \\
		&\mathcal{Q}_2 := -\p_j \rho \p_\tau \psi - \p_j(\e \rho \mathbf{u}) \cdot \nabla \psi + \p_j \T \nabla \phi + \p_j \rho \nabla \zeta, \\
		&\mathcal{Q}_3 := -\p_j \rho \p_\tau \zeta -  \p_j(\e \rho \mathbf{u}) \cdot \nabla \zeta - \p_j(\rho \T) \dv \psi.
	\end{align*}
	We first derive an estimate for \(\norm{\mathcal{R}_i}_{L^2(\Omega_\e)}\), \(i=1,2,3\). By the a priori assumptions (\ref{a priori assumption}), we find that
	\begin{align}
		&\norm{\nabla^k \mathcal{R}_1}_{L^2(\Omega_\e)} \le C \eta \e^{-\frac{1}{2}} (\e^{-2} + \tau)^{-\frac{5}{4} - \frac{k}{2}} + C \bar{\eta} \sum_{l=0}^{k} (\e^{-2} + \tau)^{-\frac{k-l+1}{2}} \norm{\nabla^l (\phi, \psi, \zeta)}_{L^2(\Omega_\e)}, \label{R_1}\\
		&\norm{\nabla^k \mathcal{R}_2}_{L^2(\Omega_\e)} \le C \bar{\delta} \e^{-1}(\e^{-2} + \tau)^{-\frac{5}{4} - \frac{k}{2}} + C \bar{\eta} \sum_{l=0}^{k} (\e^{-2} + \tau)^{-\frac{k-l+1}{2}} \norm{\nabla^l (\phi, \psi, \zeta)}_{L^2(\Omega_\e)}, \label{R_2}\\
		&\norm{\nabla^k \mathcal{R}_3}_{L^2(\Omega_\e)} \le C \norm{\nabla^k (\mathcal{R}_1, \mathcal{R}_2)}_{L^2(\Omega_\e)} + C \chi (\e^{-2} + \tau)^{-\frac{3}{4}} \norm{\nabla \psi}_{H^k(\Omega_\e)} + C \bar{\eta} \sum_{l=0}^{k} (\e^{-2} + \tau)^{-\frac{k-l+1}{2}} \norm{\nabla^{l+1} \psi}_{L^2(\Omega_\e)}. \label{R_3}
	\end{align}
	Moreover, by the a priori assumptions (\ref{a priori assumption}), we have the result
	\begin{equation}\label{Q}
		\begin{aligned}
			&\norm{\nabla^k \mathcal{Q}_1}_{L^2(\Omega_\e)} \le C (\bar{\delta} + \chi + \eta) (\e^{-2} + \tau)^{-\frac{1}{2}} \norm{\nabla (\phi, \psi, \zeta)}_{H^{k+1}(\Omega_\e)}, \\
			&\norm{\nabla^k \mathcal{Q}_i}_{L^2(\Omega_\e)} \le C (\bar{\delta} + \chi + \eta) (\e^{-2} + \tau)^{-\frac{1}{2}} \norm{\nabla (\phi, \psi, \zeta)}_{H^{k+1}(\Omega_\e)} + C \e^{-1}\bar{\delta} (\e^{-2} + \tau)^{-\frac{7}{4} - \frac{k}{2}},
		\end{aligned}
	\end{equation}
	for \(i=2,3\), \(k= 0, 1\). 
	Then, we state the following result.
	\begin{Lem}\label{high order lemma1}
		Under the same assumptions as Proposition \ref{a priori estimates}, it holds that
		\begin{equation}
			\begin{aligned}
				\frac{d}{d \tau} E_3 + K_3 
				\le& C \bar\delta \e^{-1}(\e^{-2} + \tau)^{-\frac{5}{2}}
				+ C \bar{\eta} (\e^{-2} + \tau)^{-1} \norm{\nabla (\phi, \psi, \zeta)}_{L^2(\Omega_\e)}^2 
				+ C \bar{\eta} (\e^{-2} + \tau)^{-2} \norm{(\phi, \psi, \zeta)}_{L^2(\Omega_\e)}^2\\
				&+ C \bar{\eta} \e^{-2} \sum_{i=0}^{2} (\e^{-2} + \tau)^{-(3-i)} E_i + C \bar{\eta} \e^{-2} [(\e^{-2} + \tau) ^{-2} G_0 + (\e^{-2} + \tau) ^{-1} G_1],
			\end{aligned}
		\end{equation}
		where
		\begin{align*}
			E_3 =& \int_{\Omega_\e} \frac{2\mu+\lambda}{2\rho} \tilde{C}_1|\nabla^2 \phi|^2 + \tilde{C}_1\rho \nabla \phi \cdot \nabla^2 \phi +\frac{\rho}{2} |\nabla^2 (\psi, \zeta)|^2 dy + \frac{\tilde{C}_2}{\e^2} E_2 + \tilde{C}_2 \norm{(\phi_{\neq}, \psi_{\neq}, \zeta_{\neq})}_{H^1(\Omega_\e)}^2, \\
			K_3 =& \frac{1}{2} \int_{\Omega_\e} \left(\tilde{C}_1\T|\nabla^2 \phi|^2 + \mu |\nabla^3 \psi|^2 + (\mu+\lambda)|\nabla^2 \dv \psi|^2 + \kappa |\nabla^3 \zeta|^2\right) dy\\
			&+ \frac{\tilde{C}_2}{\e^2} (K_2 + G_2) + \tilde{C}_2 \norm{(\nabla \phi_{\neq}, \nabla^2 \psi_{\neq}, \nabla^2 \zeta_{\neq})}_{L^2(\Omega_\e)}^2.
		\end{align*}
		Here \(\tilde{C}_1\) is a suitably small constant used to ensure that \(E_3\) is positive, and \(\tilde{C}_2\) is a positive constant determined later.
	\end{Lem}
	\begin{proof}
		\subsubsection*{Step 1. Estimates on \(\|\nabla^2 \phi\|_{L^2(\Omega_\e)}\)}
		
		Multiplying \((\ref{high-order equs})_2\) by \(\nabla \p_j \phi\) and integrating with respect to \(y\), one has
		\begin{equation}\label{nabla^2 phi equ1}
			\begin{aligned}
				&\frac{d}{d \tau} \int_{\Omega_\e} \rho \p_j \psi_i \p_j \p_i \phi dy +\int_{\Omega_\e} \T |\p_j \p_i \phi|^2 dy \\
				=& \int_{\Omega_\e} \p_\tau \rho \p_j \psi_i \p_j \p_i \phi dy
				+ \int_{\Omega_\e} \p_i (\rho \p_j \psi_i) \p_j (\e \mathbf{u} \cdot \nabla \phi + \rho \dv \psi - \mathcal{R}_1) dy \\
				&- \int_{\Omega_\e} [\e \rho \mathbf{u} \cdot \nabla \p_j \psi +\rho \nabla \p_j \zeta - \p_j \mathcal{R}_2 - \mathcal{Q}_2] \cdot \nabla \p_j \phi dy + (2\mu+\lambda) \int_{\Omega_\e} \p_j \p_i \phi \p_j \p_i \dv \psi dy.
			\end{aligned}
		\end{equation}
		On the other hand, multiplying \((\ref{high-order equs})_1\) by \(\frac{1}{\rho}\p_j \p_i \phi\) and integrating the result on \(\Omega_\e\) yield that
		\begin{equation}\label{nabla^2 phi equ2}
			\begin{aligned}
				\frac{d}{d \tau} \int_{\Omega_\e} \frac{|\p_{ij} \phi|^2}{2\rho} dy = -\int_{\Omega_\e}\p_{ij} \phi \p_{ijk} \psi_k dy + \int_{\Omega_\e} \left(\p_\tau \left(\frac{1}{2\rho}\right) + \dv \left(\frac{\e \mathbf{u}}{2\rho}\right)\right) |\p_{ij} \phi|^2 dy + \int_{\Omega_\e} (\p_{ijk} \mathcal{R}_1 + \mathcal{Q}_1) \frac{\p_{ij} \phi}{\rho} dy.
			\end{aligned}
		\end{equation}
		Taking \(\sum_{i,j=1}^{3} (\ref{nabla^2 phi equ1}) + (2\mu + \lambda)(\ref{nabla^2 phi equ2})\) and using the Cauchy inequality, \eqref{R_1} and \eqref{Q}, we get
		\begin{align}\label{high1}
			&\frac{d}{d \tau} \left[\int_{\Omega_\e} \frac{2\mu+\lambda}{2\rho} |\nabla^2 \phi|^2 + \rho \nabla \phi \cdot \nabla^2 \phi dy\right] + \int_{\Omega_\e} \T|\nabla^2 \phi|^2 dy \notag\\
			\le& \frac{1}{120} \int_{\Omega_\e} \T |\nabla^2 \phi|^2 dy + C \bar{\delta} \e^{-2}(\e^{-2} + \tau)^{-\frac{7}{2}}  
			+ C \bar{\eta} (\e^{-2} + \tau)^{-1} \norm{\nabla (\phi, \psi, \zeta)}_{L^2(\Omega_\e)}^2 \notag\\
			&+ C \bar{\eta} (\e^{-2} + \tau)^{-2} \norm{(\phi, \psi, \zeta)}_{L^2(\Omega_\e)}^2 
			+ \frac{\tilde{C}}{120} \norm{\nabla^2 (\psi, \zeta)}_{L^2(\Omega_\e)}^2,
		\end{align}
		where \(\tilde{C} > 0\) is a positive constant.
		
		\subsubsection*{Step 2. Estimates on \(\|\nabla^2 (\psi, \zeta)\|_{L^2(\Omega_\e)}\)}
		
		Next, we will estimate \(\|\nabla^2 (\psi, \zeta)\|_{L^2(\Omega_\e)}\). Multiplying \((\ref{high-order equs})_2\) by \(-\Delta \p_i \psi\), \((\ref{high-order equs})_3\) by \(-\Delta \p_i \zeta\), summing the result equations up and using \eqref{high-order equs}-\eqref{Q}, we find
		\begin{align}\label{high2}
			&\frac{d}{d \tau} \int_{\Omega_\e} \frac{\rho}{2} |\nabla^2 (\psi, \zeta)|^2 dy + \int_{\Omega_\e} \left(\mu |\nabla^3 \psi|^2 + (\mu+\lambda)|\nabla^2 \dv \psi|^2 + \kappa |\nabla^3 \zeta|^2\right) dy \notag\\
			\le& \int_{\Omega_\e} |\e \rho \mathbf{u} \cdot \dv \p_j \psi + \T \nabla \p_j \phi + \rho \nabla \p_j \zeta + \p_j \mathcal{R}_2 + \mathcal{Q}_2| |\Delta \p_j \psi| + \frac{|\p_\tau \rho|}{2} |\nabla^2 \psi|^2 + |\nabla \rho| |\p_\tau \nabla \psi| |\nabla^2 \psi| dy \notag\\
			&+ \int_{\Omega_\e} |\e \rho \mathbf{u} \cdot \dv \p_j \zeta + \rho \T \dv \p_j \psi + \p_j \mathcal{R}_3 + \mathcal{Q}_3| |\Delta \p_i \zeta| + \frac{|\p_\tau \rho|}{2} |\nabla^2 \zeta|^2 + |\nabla \rho| |\p_\tau \nabla \zeta| |\nabla^2 \zeta| dy \notag\\
			\le& \frac{\min\{\mu, \kappa, \mu+\lambda\}}{120} \norm{\nabla^3 (\psi, \zeta)}_{L^2(\Omega_\e)}^2 + \frac{\tilde{C}_1}{120} \int_{\Omega_\e} \T |\nabla^2 \phi|^2 dy + C \bar{\delta} \e^{-2} (\e^{-2} + \tau)^{-\frac{7}{2}}  \notag\\
			&+ C \bar{\eta} (\e^{-2} + \tau)^{-1} \norm{\nabla (\phi, \psi, \zeta)}_{L^2(\Omega_\e)}^2 
			+ C \bar{\eta} (\e^{-2} + \tau)^{-2} \norm{(\phi, \psi, \zeta)}_{L^2(\Omega_\e)}^2 
			+ \frac{\tilde{C}_1}{120} \norm{\nabla^2 (\psi, \zeta)}_{L^2(\Omega_\e)}^2,
		\end{align}
		where \(\tilde{C}_1 >0\) is a positive constant.
		Combining \eqref{high1} and \eqref{high2}, one obtains 
		\begin{equation}\label{nabla^2 phi est}
			\begin{aligned}
				&\frac{d}{d \tau} \left[\int_{\Omega_\e} \frac{2\mu+\lambda}{2\rho} \tilde{C}_1|\nabla^2 \phi|^2 + \tilde{C}_1\rho \nabla \phi \cdot \nabla^2 \phi +\frac{\rho}{2} |\nabla^2 (\psi, \zeta)|^2 dy\right] \\
				&+\frac{1}{2} \int_{\Omega_\e} \left(\tilde{C}_1\T|\nabla^2 \phi|^2 + \mu |\nabla^3 \psi|^2 + (\mu+\lambda)|\nabla^2 \dv \psi|^2 + \kappa |\nabla^3 \zeta|^2\right) dy \\
				\le& C \bar{\delta} \e^{-2} (\e^{-2} + \tau)^{-\frac{7}{2}} + C \bar{\eta} (\e^{-2} + \tau)^{-1} \norm{\nabla (\phi, \psi, \zeta)}_{L^2(\Omega_\e)}^2 
				+ C \bar{\eta} (\e^{-2} + \tau)^{-2} \norm{(\phi, \psi, \zeta)}_{L^2(\Omega_\e)}^2 \\
				&+ \frac{\tilde{C}_2}{120} \norm{\nabla^2 (\psi, \zeta)}_{L^2(\Omega_\e)}^2,
			\end{aligned}
		\end{equation}
		where \(\tilde{C}_2>0\) is a positive constant.
		According to (\ref{nabla^2 phi est}), (\ref{E_2 est}) and (\ref{non zero inequ}), we arrive at Lemma \ref{high order lemma1}. 
	\end{proof}
	
	To close the a priori estimate, we need to estimate the third order derivatives. Applying \(\p_k\), \(k = 1,2,3\) to \((\ref{high-order equs})_1\) and \(\p_i\), \(i=1,2,3\) to \((\ref{high-order equs})_2\) and \((\ref{high-order equs})_2\), one has
	\begin{equation}\label{highest-order equs}
		\left\{
		\begin{aligned}
			&\p_\tau \p_{ijk} \phi + \e \mathbf{u} \cdot \nabla \p_{ijk} \phi + \rho \dv \p_{ijk} \psi = \p_{ijk} \mathcal{R}_1 + \p_k \mathcal{Q}_1 + \mathcal{Q}_1^*, \\
			&\rho \p_\tau \p_{ij} \psi + \e \rho \mathbf{u} \cdot \nabla \p_{ij} \psi + \T \nabla \p_{ij} \phi + \rho \nabla \p_{ij} \zeta = \mu \Delta \p_{ij} \psi + (\mu+ \lambda)\nabla \dv \p_{ij} \psi + \p_{ij}\mathcal{R}_2 +  \p_i \mathcal{Q}_2 + \mathcal{Q}_2^*, \\
			&\rho \p_\tau \p_{ij} \zeta + \e \rho \mathbf{u} \cdot \nabla \p_{ij} \zeta + \rho \T \dv \p_{ij} \psi = \kappa \Delta \p_{ij} \zeta + \p_{ij} \mathcal{R}_3 + \p_i \mathcal{Q}_3 + \mathcal{Q}_3^*,
		\end{aligned}
		\right.
	\end{equation}
	where
	\begin{align*}
		&\mathcal{Q}_1^* := -\p_k (\e \mathbf{u}) \cdot \nabla \p_{ij} \phi - \p_k \rho \dv \p_{ij} \psi, \\
		&\mathcal{Q}_2^* := -\p_i \rho \p_\tau \p_j \psi - \p_i (\e \rho \mathbf{u}) \cdot \nabla \p_j \psi + \p_i \T \nabla \p_j \psi + \p_i \rho \nabla \p_j \zeta, \\
		&\mathcal{Q}_3^* := -\p_i \rho \p_\tau \p_j \zeta - \p_i (\e \rho \mathbf{u}) \cdot \nabla \p_j \zeta - \p_i (\rho \T) \dv \p_j \psi.
	\end{align*}
	A direct calculation gives that
	\begin{equation}\label{Q*}
		\begin{aligned}
			\norm{\mathcal{Q}_1^*}_{L^2(\Omega_\e)} \le& C (\bar{\delta} + \chi) (\e^{-2} + \tau)^{-\frac{1}{2}} \norm{\nabla^3 (\phi, \psi)}_{L^2(\Omega_\e)}, \\
			\norm{\mathcal{Q}_i^*}_{L^2(\Omega_\e)} \le& C (\bar{\delta} + \chi) (\e^{-2} + \tau)^{-\frac{1}{2}} \norm{\nabla^2 (\phi, \psi, \zeta)}_{H^1(\Omega_\e)}  + C (\bar{\delta} + \chi) (\e^{-2} + \tau)^{-1} \norm{\nabla (\phi, \psi, \zeta)}_{L^2(\Omega_\e)} \\
			&+ C (\bar{\delta} + \chi) (\e^{-2} + \tau)^{-\frac{3}{2}} \norm{(\phi, \psi, \zeta)}_{L^2(\Omega_\e)} + C (\bar{\delta} + \chi) \e^{-1} (\e^{-2} + \tau)^{-\frac{9}{4}},
		\end{aligned}
	\end{equation}
	for \(i = 2,3\).
	
	We summarize the estimates for the highest-order derivatives in the following result.
	\begin{Lem}\label{high order lemma2}
		Under the same assumptions as Proposition \ref{a priori estimates}, it holds that
		\begin{equation}
			\begin{aligned}
				\frac{d}{d \tau} E_4 + K_4 
				\le& C \bar{\delta} \e^{-1}(\e^{-2} + \tau)^{-\frac{5}{2}}
				+ C \bar{\eta} (\e^{-2} + \tau)^{-1} \norm{\nabla^2 (\phi, \psi, \zeta)}_{L^2(\Omega_\e)}^2
				+ C \bar{\eta} (\e^{-2} + \tau)^{-1} \norm{\nabla (\phi, \psi, \zeta)}_{L^2(\Omega_\e)}^2\\ 
				&+ C \bar{\eta} (\e^{-2} + \tau)^{-2} \norm{(\phi, \psi, \zeta)}_{L^2(\Omega_\e)}^2
				+ C \bar{\eta} \e^{-2} \sum_{i=0}^{2} (\e^{-2} + \tau)^{-(3-i)} E_i \\
				&+ C \bar{\eta} \e^{-2} [(\e^{-2} + \tau) ^{-2} G_0 + (\e^{-2} + \tau) ^{-1} G_1],
			\end{aligned}
		\end{equation}
		where
		\begin{align*}
			E_4 =& \int_{\Omega_\e} \frac{2\mu+\lambda}{2\rho} \tilde{C}_4|\nabla^3 \phi|^2 + \tilde{C}_4\rho \nabla^2 \phi \cdot \nabla^3 \phi +\frac{\rho}{2} |\nabla^3 (\psi, \zeta)|^2 dy + \tilde{C}_5 E_3, \\
			K_4 =& \frac{1}{2} \int_{\Omega_\e} \left(\tilde{C}_4\T|\nabla^3 \phi|^2 + \mu |\nabla^4 \psi|^2 + (\mu+\lambda)|\nabla^3 \dv \psi|^2 + \kappa |\nabla^4 \zeta|^2\right) dy + \tilde{C}_5 K_3.
		\end{align*}
		Here \(\tilde{C}_4\) is a suitably small constant used to ensure that \(E_3\) is positive, and \(\tilde{C}_5\) is a positive constant determined later.
	\end{Lem}
	
	\begin{proof}
		\subsubsection*{Step 1. Estimates on \(\|\nabla^3 \phi\|_{L^2(\Omega_\e)}\)}
		
		As in the proof of the previous lemma, we first estimate  \(\|\nabla^3 \phi\|_{L^2(\Omega_\e)}\). Multiplying \((\ref{highest-order equs})_2\) by \(\nabla \p_{ij} \phi\) yields that
		\begin{equation}\label{nabla^3 phi equ1}
			\begin{aligned}
				&\frac{d}{d \tau} \int_{\Omega_\e} \rho \p_{ij} \psi_k \p_{ijk} \phi dy + \int_{\Omega_\e} \T |\p_{ijk} \phi|^2 dy \\
				=&\int_{\Omega_\e} \p_\tau \rho \p_{ij} \psi_k \p_{ijk} \phi dy + \int_{\Omega_\e} \p_k (\rho \p_{ij} \psi_k) (\e \mathbf{u} \cdot \nabla \p_{ij} \phi + \rho \dv \p_{ij} \psi - \p_{ij} \mathcal{R}_1 - \mathcal{Q}_1) dy \\
				&- \int_{\Omega_\e} (\e \rho \mathbf{u} \cdot \nabla \p_{ij} \psi + \rho \nabla \p_{ij} \zeta - \p_{ij} \mathcal{R}_2 - \p_i \mathcal{Q}_2 - \mathcal{Q}_2^*) \cdot \nabla \p_{ij} \phi dy + (2\mu + \lambda) \int_{\Omega_\e} \p_{ijk} \phi \p_{ijk} \dv \psi dy.
			\end{aligned}
		\end{equation}
		On the other hand, multiplying \((\ref{highest-order equs})_1\) by \(\frac{\p_{ijk} \phi}{\rho}\) yields that
		\begin{equation}\label{nabla^3 phi equ2}
			\begin{aligned}
				\frac{d}{d \tau} \int_{\Omega_\e} \frac{|\p_{ijk} \phi|^2}{2\rho} dy =& - \int_{\Omega_\e} \p_{ijk} \phi \p_{ijk} \dv \psi dy + \int_{\Omega_\e} \left[\p_\tau \left(\frac{1}{2\rho}\right) + \dv \left(\frac{\e \mathbf{u}}{2\rho}\right)\right] |\p_{ijk} \phi|^2 dy \\
				& + \int_{\Omega_\e} (\p_{ijk} \mathcal{R}_1 + \p_k \mathcal{Q}_1 + \mathcal{Q}_1^*) \frac{\p_{ijk} \phi}{\rho} dy.
			\end{aligned}
		\end{equation}
		Taking \(\sum_{i,j,k=1}^{3} (\ref{nabla^3 phi equ1}) + (2\mu + \lambda)(\ref{nabla^3 phi equ2})\) and using (\ref{R_1}) - (\ref{Q}) and (\ref{Q*}), we find
		\begin{equation*}
			\begin{aligned}
				&\frac{d}{d \tau} \left[\int_{\Omega_\e} \frac{2\mu+\lambda}{2\rho} |\nabla^3 \phi|^2 + \rho \nabla^2 \phi \cdot \nabla^3 \phi dy\right] + \int_{\Omega_\e} \T|\nabla^3 \phi|^2 dy \\
				\le& \frac{1}{120} \int_{\Omega_\e} \T |\nabla^3 \phi|^2 dy + C \bar{\delta} \e^{-2} (\e^{-2} + \tau)^{-\frac{9}{2}} + \frac{\tilde{C}_3}{120} \norm{\nabla^3 (\psi, \zeta)}_{L^2(\Omega_\e)}^2 + C \bar{\eta} (\e^{-2} + \tau)^{-1} \norm{\nabla^2 (\phi, \psi, \zeta)}_{L^2(\Omega_\e)}^2 \\
				&+ C \bar{\eta} (\e^{-2} + \tau)^{-1} \norm{\nabla (\phi, \psi, \zeta)}_{L^2(\Omega_\e)}^2 + C \bar{\eta} (\e^{-2} + \tau)^{-3} \norm{(\phi, \psi, \zeta)}_{L^2(\Omega_\e)}^2,
			\end{aligned}
		\end{equation*}
		where \(\tilde{C}_3\) is a positive constant.
		
		\subsubsection*{Step 2. Estimates on \(\|\nabla^3 (\psi, \zeta)\|_{L^2(\Omega_\e)}\)}
		
		Next, we will estimate \(\|\nabla^3 (\psi, \zeta)\|_{L^2(\Omega_\e)}\). Multiplying  \((\ref{highest-order equs})_2\) by \(-\Delta \p_{ij} \psi\), \((\ref{highest-order equs})_3\) by \(-\Delta \p_{ij} \zeta\) and adding them together, one has
		\begin{align*}
			&\frac{d}{d \tau} \int_{\Omega_\e}\frac{\rho}{2} |\nabla^3 (\psi, \zeta)|^2 dy + \int_{\Omega_\e} \left(\mu |\nabla^4 \psi|^2 + (\mu+\lambda)|\nabla^3 \dv \psi|^2 + \kappa |\nabla^4 \zeta|^2\right) dy \\
			\le& \int_{\Omega_\e} |\e \rho \mathbf{u} \cdot \nabla \p_{ij} \psi + \T \nabla \p_{ij} \psi + \rho \nabla \p_{ij} \zeta + \p_{ij} \mathcal{R}_2 + \p_i \mathcal{Q}_2 + \mathcal{Q}_2^*| |\Delta \p_{ij} \psi| + \frac{|\p_\tau \rho|}{2} |\nabla^3 \psi|^2 + |\nabla \rho| |\p_\tau \nabla^2 \psi| |\nabla^3 \psi| dy \\
			&+ \int_{\Omega_\e} |\e \rho \mathbf{u} \cdot \nabla \p_{ij} \zeta + \rho \T \dv \p_{ij}\psi + \p_{ij} \mathcal{R}_3 + \p_i \mathcal{Q}_3 + \mathcal{Q}_3^*| |\Delta \p_{ij} \zeta| 
			+ \frac{|\p_\tau \rho|}{2} |\nabla^3 \zeta|^2 + |\nabla \rho| |\p_\tau \nabla^2 \zeta| |\nabla^3 \zeta| dy \\
			\le& \frac{\min\{\mu, \kappa, \mu+ \lambda\}}{120} \norm{\nabla^4 (\psi, \zeta)}_{L^2(\Omega_\e)}^2 + \frac{\tilde{C}_4}{120} \int_{\Omega_\e} \T |\nabla^3 \phi|^2 dy + \frac{\tilde{C}_4}{120} \norm{\nabla^3 (\psi, \zeta)}_{L^2(\Omega_\e)}^2
			+ C \bar{\delta} \e^{-2} (\e^{-2} + \tau)^{-\frac{9}{2}} \\
			&+ C \bar{\eta} (\e^{-2} + \tau)^{-1} \norm{\nabla^2 (\phi, \psi, \zeta)}_{L^2(\Omega_\e)}^2 
			+ C \bar{\eta} (\e^{-2} + \tau)^{-1} \norm{\nabla (\phi, \psi, \zeta)}_{L^2(\Omega_\e)}^2
			+ C \bar{\eta} (\e^{-2} + \tau)^{-2} \norm{(\phi, \psi, \zeta)}_{L^2(\Omega_\e)}^2,
		\end{align*}
		where \(\tilde{C}_4\) is a positive constant. Combining the above two inequalities leads to
		\begin{equation}
			\begin{aligned}
				&\frac{d}{d \tau} \left[\int_{\Omega_\e} \frac{2\mu+\lambda}{2\rho} \tilde{C}_4|\nabla^3 \phi|^2 + \tilde{C}_4\rho \nabla^2 \phi \cdot \nabla^3 \phi +\frac{\rho}{2} |\nabla^3 (\psi, \zeta)|^2 dy\right] \\
				&+\frac{1}{2} \int_{\Omega_\e} \left(\tilde{C}_4\T|\nabla^3 \phi|^2 + \mu |\nabla^4 \psi|^2 + (\mu+\lambda)|\nabla^3 \dv \psi|^2 + \kappa |\nabla^4 \zeta|^2\right) dy \\
				\le& C \bar{\delta} \e^{-2} (\e^{-2} + \tau)^{-\frac{9}{2}} + C \bar{\eta} (\e^{-2} + \tau)^{-1} \norm{\nabla^2 (\phi, \psi, \zeta)}_{L^2(\Omega_\e)}^2 
				+ C \bar{\eta} (\e^{-2} + \tau)^{-1} \norm{\nabla (\phi, \psi, \zeta)}_{L^2(\Omega_\e)}^2 \\
				&+ C \bar{\eta} (\e^{-2} + \tau)^{-2} \norm{(\phi, \psi, \zeta)}_{L^2(\Omega_\e)}^2 + + \frac{\tilde{C}_5}{120} \norm{\nabla^3 (\psi, \zeta)}_{L^2(\Omega_\e)}^2,
			\end{aligned}
		\end{equation}
		where \(\tilde{C}_5\) is a positive constant. Combining the above inequality and Lemma \ref{high order lemma1}, we obtain Lemma \ref{high order lemma2}.
	\end{proof}
	
	\subsection*{Decay rate for the high-order derivatives}
	
	\begin{Lem}\label{decay for high}
		Under the same assumptions as Proposition \ref{a priori estimates}, it holds that
		\begin{equation}
			\begin{aligned}
				\norm{\nabla^2 (\phi, \psi, \zeta)}_{H^1(\Omega_\e)}^2 \le C\bar{\delta} \e^{-1} (\e^{-2} + \tau)^{-\frac{3}{2}}.
			\end{aligned}
		\end{equation}
	\end{Lem}
	
	\begin{proof}
		According to  Lemma \ref{main lemma}, Lemma \ref{non zero est} and Lemma \ref{high order lemma1}, it can be seen that
		\begin{align}\label{d E_3 est}
			&\frac{d}{d \tau} E_3 + K_3 \le C \bar{\delta} \e^{-1} (\e^{-2} + \tau)^{-\frac{5}{2}}.
		\end{align}
		Multiplying (\ref{d E_3 est}) by \((\e^{-2} + \tau)^2\), integrating on \([0,\tau]\) and using the expression for \(E_3\) and (\ref{lab1}), one has
		\begin{align*}
			(\e^{-2} + \tau)^2 E_3 + \int_{0}^{\tau}(\e^{-2} + s)^2 K_3 ds &\le \e^{-4}E_3(0)+ C\int_{0}^{\tau}(\e^{-2} + s) E_3 ds + C \bar{\delta} \e^{-1} (\e^{-2} + \tau)^{\frac{1}{2}} \\
			&\le \e^{-4}E_3(0) + C\bar{\delta} \e^{-1} (\e^{-2} + \tau)^{\frac{1}{2}},
		\end{align*}
		which leads to 
		\begin{align}\label{E_3 est}
			E_3 \le C\bar{\delta} \e^{-1} (\e^{-2} + \tau)^{-\frac{3}{2}}, \qquad \int_{0}^\tau (\e^{-2} + s)^2 K_3 ds \le C\bar{\delta} \e^{-1} (\e^{-2} + \tau)^{\frac{1}{2}},
		\end{align}
		By a same argument and using (\ref{E_3 est}), we find 
		\begin{align}\label{E_4 est}
			E_4 \le C\bar{\delta} \e^{-1} (\e^{-2} + \tau)^{-\frac{3}{2}},
		\end{align}
		then we arrive at Lemma \ref{decay for high}. 
	\end{proof}
	\subsection*{Proof of Proposition \ref{a priori estimates}}
	Combining Lemma \ref{zero mode est}, Lemma \ref{non-zero est for well} and Lemma \ref{decay for high}, we have finished the proof of Proposition \ref{a priori estimates}.
	
	\section{Low Mach number limit for moderately ill-prepared initial data} \label{section4}
	
	%In this section, we rewrite (\ref{nondimensional NS Eqs}) as 
	%\begin{equation} \label{nondimensional NS Eqs for ill data}
	%\left\{
	%\begin{aligned}
	%&\partial_t \rho^{\e} + \text{div}(\rho^{\e} \mathbf{u}^{\e})=0, \\
	%& \rho^{\e} (\partial_t \mathbf{u}^{\e} + \mathbf{u}^\e \cdot \nabla \mathbf{u}^\e) + \frac{1}{\e^2} (\rho^\e \nabla \T^\e + \T^\e \nabla \rho^\e) = \mu \Delta \mathbf{u}^\e + (\mu + \lambda) \nabla \dv \mathbf{u}^\e,  \\
	%&\rho^{\e} \p_t \T^\e + \rho^\e \mathbf{u}^\e \cdot \nabla \T^\e + \rho^\e \T^\e \dv \mathbf{u}^\e = \kappa \Delta \T^\e, \\
	%&(\rho^\e , \mathbf{u}^\e, \T^\e)(x,0) = (\rho^\e_0, \mathbf{u}^\e_0, \T^\e_0).
	%\end{aligned}
	%\right.
	%\end{equation}

	\subsection{Estimates for the non-zero modes}
	In this subsection, We establish estimates for the non-zero mode variables. 
	%The key lemma can be summarized as follows:
	%\begin{Lem}
	%	Under the assumptions of Theorem \ref{Uniform estimates Thm for ill-prepared initial data} and (\ref{non-zero initial data for ill}), it holds that
	%	\begin{align}\label{non-zero mode est for ill}
		%		\sup_{0 \le t \le T}\norm{(\rho_{\neq}, \mathbf{u}_{\neq}, \T_{\neq})(t)}_{H^s}^2
		%		\le e^{-C \e^{-4}},
		%	\end{align}
	%	where \(C\) is a positive constant.
	%\end{Lem}
	%\begin{proof}
	For convenience, we omit the superscript \(\varepsilon\) on the variables. First, we establish the equivalence of the norms of the variables under the transformation in (\ref{transformation1 for ill}).
	According to (\ref{transformation1 for ill}), for \(1\le q \le \infty\), we can see that
	\begin{align*}
		&\|p_{\neq}\|_{L^q(\Omega)} = \|\mathbf{D}_{\neq} [e^{\e \underline{\mathring{p}}} (e^{\e \underline{p}_{\neq}}-1)]\|_{L^q(\Omega)} \le C \|\e \underline{p}_{\neq}\|_{L^q(\Omega)},
	\end{align*}
	and
	\begin{align*}
		&\|\e \underline{p}_{\neq}\|_{L^q(\Omega)} = \|\mathbf{D}_{\neq} (\ln{p})\|_{L^q(\Omega)} = \norm{\ln{\left( 1 + \frac{p_{\neq}}{\mathbf{D}_0 p}\right)} - \mathbf{D}_{0} \left[\ln{\left( 1 + \frac{p_{\neq}}{\mathbf{D}_0 p}\right)}\right]}_{L^q(\Omega)} \le C \|p_{\neq}\|_{L^q(\Omega)},
	\end{align*}
	which means that \(\|p_{\neq}\|_{L^q(\Omega)}\) can be equivalent to \(\|\e \underline{p}_{\neq}\|_{L^q(\Omega)}\). By a similar argument, we also find that \(\|\T_{\neq}\|_{L^q(\Omega)}\) is equivalent to \(\|\underline{\T}_{\neq}\|_{L^q(\Omega)}\), and \(\|\rho_{\neq}^\e\|_{L^q(\Omega)}\) is equivalent to \(\|\e \underline{p}_{\neq} - \underline{\T}_{\neq}\|_{L^q(\Omega)}\).
	
	Then, we will prove (\ref{non-zero mode est for ill}). 
	%Applying \(\mathbf{D}_{\neq}\) to (\ref{nondimensional NS Eqs for ill data}), we obtain the non-zero mode system as follows:
	%\begin{equation}\label{non-zero system}
	%	\left\{
	%	\begin{aligned}
		%	&\p_t \rho_{\neq} + \mathring{\mathbf{u}} \cdot \nabla \rho_{\neq} + \mathring{\rho} \dv \mathbf{u}_{\neq} = f_{1\neq}, \\
		%	&\mathring{\rho} \p_t \mathbf{u}_{\neq} + \mathbf{D}_0 (\rho \mathbf{u}) \cdot \nabla \mathbf{u}_{\neq} +\frac{1}{\e^2}\left(\mathbf{D}_{\neq} (\rho \nabla \T) + \mathbf{D}_{\neq} (\T \nabla \rho)\right) = \mu \Delta \mathbf{u}_{\neq} + (\lambda + \mu) \nabla \dv \mathbf{u}_{\neq} + \mathbf{f}_{2\neq}, \\
		%	&\mathring{\rho} \p_t \T_{\neq} + \mathbf{D}_0 (\rho \mathbf{u}) \cdot \nabla \T_{\neq} + \mathbf{D}_0 (\rho \T) \dv \mathbf{u}_{\neq} = \kappa \Delta \T_{\neq} + f_{3\neq},
		%	\end{aligned}
	%	\right.
	%\end{equation}
	%where
	%\begin{align*}
	%	&f_{1\neq} := \mathring{\mathbf{u}} \cdot \nabla \rho_{\neq} - \mathbf{D}_{\neq} (\mathbf{u} \cdot \nabla \rho) + \mathring{\rho} \dv \mathbf{u}_{\neq} - \mathbf{D}_{\neq} (\rho \dv \mathbf{u}), \\
	%	&\mathbf{f}_{2\neq} :=\mathring{\rho} \p_t \mathbf{u}_{\neq} - \mathbf{D}_{\neq} (\rho \p_t \mathbf{u}) + \mathbf{D}_0 (\rho \mathbf{u}) \cdot \nabla \mathbf{u}_{\neq} - \mathbf{D}_{\neq} (\rho \mathbf{u} \cdot \nabla \mathbf{u}), \\
	%	&f_{3\neq} := \mathring{\rho} \p_t \T_{\neq} - \mathbf{D}_{\neq} (\rho \p_t \T) + \mathbf{D}_0 (\rho \mathbf{u}) \cdot \nabla \T_{\neq} - \mathbf{D}_{\neq} (\rho \mathbf{u} \cdot \nabla \T) + \mathbf{D}_0 (\rho \T) \dv \mathbf{u}_{\neq} - \mathbf{D}_{\neq} (\rho \mathbf{u} \dv \mathbf{u}).
	%\end{align*}
	For \(s \ge 4\) and \(\p^\alpha = \p^{\alpha_1}_{x_1} \p^{\alpha_2}_{x_2} \p^{\alpha_3}_{x_3}\), we apply \(\p^\alpha, (|\alpha| = s)\) to (\ref{non-zero system}) to obtain that
	\begin{equation}\label{p^alpha non-zero system}
		\left\{
		\begin{aligned}
			&\partial_t\partial^\alpha\rho_{\neq}
			+\mathring{\mathbf u}\cdot\nabla\partial^\alpha\rho_{\neq}
			+\mathring{\rho}\operatorname{div}\partial^\alpha\mathbf u_{\neq}
			=F_{1,\alpha},\\
			&\mathring{\rho}\partial_t\partial^\alpha\mathbf u_{\neq}
			+\mathbf D_0(\rho\mathbf u)\cdot
			\nabla\partial^\alpha\mathbf u_{\neq}
			+\frac{1}{\varepsilon^2}\mathbf D_{\neq}
			\left(
			\rho\nabla\partial^\alpha\theta
			+\theta\nabla\partial^\alpha\rho
			\right)
			=\mu\Delta\partial^\alpha\mathbf u_{\neq}
			+(\lambda+\mu)\nabla\operatorname{div}
			\partial^\alpha\mathbf u_{\neq}
			+\mathbf F_{2,\alpha},
			\\
			&\mathring{\rho}\partial_t\partial^\alpha\theta_{\neq}
			+\mathbf D_0(\rho\mathbf u)\cdot
			\nabla\partial^\alpha\theta_{\neq}
			+\mathbf D_0(\rho\theta)\operatorname{div}
			\partial^\alpha\mathbf u_{\neq}
			=\kappa\Delta\partial^\alpha\theta_{\neq}
			+ \e^2 \nabla \mathring{\mathbf{u}} : \mathbb{S}(\p^\alpha \mathbf{u}_{\neq})
			+F_{3,\alpha}.
		\end{aligned}
		\right.
	\end{equation}
	where
	\begin{align*}
		&F_{1,\alpha}
		=\partial^\alpha f_{1\neq}
		-[\partial^\alpha,\mathring{\mathbf u}]
		\cdot\nabla\rho_{\neq}
		-[\partial^\alpha,\mathring{\rho}]
		\operatorname{div}\mathbf u_{\neq}, \\
		&\mathbf F_{2,\alpha}
		=\partial^\alpha\mathbf f_{2\neq}
		-[\partial^\alpha,\mathring{\rho}]
		\partial_t\mathbf u_{\neq}
		-[\partial^\alpha,\mathbf D_0(\rho\mathbf u)]
		\cdot\nabla\mathbf u_{\neq}-\frac{1}{\varepsilon^2}\mathbf D_{\neq}
		\left(
		[\partial^\alpha,\rho]\nabla\theta
		+[\partial^\alpha,\theta]\nabla\rho
		\right), \\
		&F_{3,\alpha}
		=\partial^\alpha f_{3\neq}
		-[\partial^\alpha,\mathring{\rho}]
		\partial_t\theta_{\neq}
		-[\partial^\alpha,\mathbf D_0(\rho\mathbf u)]
		\cdot\nabla\theta_{\neq}
		-[\partial^\alpha,\mathbf D_0(\rho\theta)]
		\operatorname{div}\mathbf u_{\neq} + \e^2 [\p^\alpha, \nabla \mathring{\mathbf{u}}] :\mathbb{S}(\mathbf{u}_{\neq}).
	\end{align*}
	Multiplying \((\ref{p^alpha non-zero system})_1\) by \(\partial^\alpha \rho_{\neq}\), \((\ref{p^alpha non-zero system})_2\) by \(\partial^\alpha \mathbf{u}{\neq}\), \((\ref{p^alpha non-zero system})_3\) by \(\partial^\alpha  \theta{\neq}\) and adding the result equations together, we find
	\begin{equation}\label{equ1 in ill}
		\begin{aligned}
			&\frac{1}{2}\partial_t
			\left[
			|\partial^\alpha\rho_{\neq}|^2
			+\mathring{\rho}\left(
			|\partial^\alpha\mathbf u_{\neq}|^2
			+|\partial^\alpha\theta_{\neq}|^2
			\right)
			\right]
			+\mu |\nabla \p^\alpha\mathbf{u}_{\neq}|^2 + (\mu+\lambda) |\dv \p^\alpha\mathbf{u}_{\neq}|^2 +\kappa |\nabla \p^\alpha\T_{\neq}|^2 \\
			&\qquad+\frac12\mathring{\mathbf u}\cdot\nabla
			|\partial^\alpha\rho_{\neq}|^2
			+\frac12\mathbf D_0(\rho\mathbf u)\cdot\nabla
			\left(
			|\partial^\alpha\mathbf u_{\neq}|^2
			+|\partial^\alpha\theta_{\neq}|^2
			\right)
			\\
			&\qquad
			+\left(
			\mathring{\rho}\,\partial^\alpha\rho_{\neq}
			+\mathbf D_0(\rho\theta)\,\partial^\alpha\theta_{\neq}
			\right)
			\operatorname{div}\partial^\alpha\mathbf u_{\neq}
			+\frac{1}{\varepsilon^2}
			\mathbf D_{\neq}\left(
			\rho\nabla\partial^\alpha\theta
			+\theta\nabla\partial^\alpha\rho
			\right)\cdot\partial^\alpha\mathbf u_{\neq}
			\\
			=& \frac12\partial_t\mathring{\rho}
			\left(
			|\partial^\alpha\mathbf u_{\neq}|^2
			+|\partial^\alpha\theta_{\neq}|^2
			\right)
			+ \e^2 \nabla \mathring{\mathbf{u}} : \mathbb{S}(\p^\alpha \mathbf{u}_{\neq}) \p^\alpha \T_{\neq}
			+F_{1,\alpha}\partial^\alpha\rho_{\neq}
			+\mathbf F_{2,\alpha}\cdot\partial^\alpha\mathbf u_{\neq}
			+F_{3,\alpha}\partial^\alpha\theta_{\neq} + \dv(...).
		\end{aligned}
	\end{equation}
	By integration by parts, one can verify that
	\begin{align}\label{est1 in ill}
		&\frac{1}{2} \int_{\Omega} \mathring{\mathbf u}\cdot\nabla
		|\partial^\alpha\rho_{\neq}|^2
		+\mathbf D_0(\rho\mathbf u)\cdot\nabla
		\left(
		|\partial^\alpha\mathbf u_{\neq}|^2
		+|\partial^\alpha\theta_{\neq}|^2
		\right) dx \le C \norm{(\p^\alpha \rho_{\neq}, \p^\alpha \mathbf{u}_{\neq}, \p^\alpha \T_{\neq})}_{L^2(\Omega)}^2.
	\end{align}
	Using the Cauchy inequality, we have
	\begin{align}
		&\int_{\Omega} \left(
		\mathring{\rho} \partial^\alpha\rho_{\neq}
		+\mathbf D_0(\rho\theta)\partial^\alpha\theta_{\neq} \right)
		\operatorname{div}\partial^\alpha\mathbf u_{\neq} dx \le \frac{\min{\{\mu, \kappa, \mu+\kappa\}}}{120} \norm{\nabla \p^\alpha \mathbf{u}_{\neq}}_{L^2(\Omega)}^2 + C\norm{(\p^\alpha \rho_{\neq}, \p^\alpha \T_{\neq})}_{L^2(\Omega)}^2, \\
		&\int_{\Omega} \e^2 \nabla \mathring{\mathbf{u}} : \mathbb{S}(\p^\alpha \mathbf{u}_{\neq}) \p^\alpha \T_{\neq} dx \le \frac{\min{\{\mu, \kappa, \mu+\kappa\}}}{120} \norm{\nabla \p^\alpha \mathbf{u}_{\neq}}_{L^2(\Omega)}^2 + C \norm{ \p^\alpha \T_{\neq}}_{L^2(\Omega)}^2.
	\end{align}
	Next, as for the term \(\frac{1}{\varepsilon^2} \int_{\Omega}
	\mathbf D_{\neq}\left(
	\rho\nabla\partial^\alpha\theta
	+\theta\nabla\partial^\alpha\rho
	\right)\cdot\partial^\alpha\mathbf u_{\neq} dx\), we have
	\begin{equation}
		\begin{aligned}
			&\frac{1}{\varepsilon^2} \int_{\Omega}
			\mathbf D_{\neq}\left(
			\rho\nabla\partial^\alpha\theta
			\right)\cdot\partial^\alpha\mathbf u_{\neq} dx = \frac{1}{\varepsilon^2} \int_{\Omega} 
			(\mathring{\rho} \nabla \p^\alpha \T_{\neq} + \rho_{\neq} \nabla \p^\alpha \T) \cdot \p^\alpha \mathbf{u}_{\neq} dx \\
			\le& \frac{1}{\e^2} \norm{\mathring{\rho}}_{L^\infty(\mathbb{R})} \norm{\nabla \p^\alpha \T_{\neq}}_{L^2(\Omega)} \norm{\p^\alpha \mathbf{u}_{\neq}}_{L^2(\Omega)} +  \frac{1}{\e^2} \norm{\rho_{\neq}}_{L^\infty(\Omega)} \norm{\nabla \p^\alpha \T}_{L^2(\Omega)} \norm{\p^\alpha \mathbf{u}_{\neq}}_{L^2(\Omega)} \\
			\le& \frac{\min{\{\mu, \kappa, \mu+\kappa\}}}{120} \norm{\nabla \p^\alpha \T_{\neq}}_{L^2(\Omega)}^2 + \frac{C}{\e^4} \norm{(\p^\alpha \rho_{\neq}, \p^\alpha \mathbf{u}_{\neq})}_{L^2(\Omega)}^2,
		\end{aligned}
	\end{equation}
	and
	\begin{equation}
		\begin{aligned}
			&\frac{1}{\varepsilon^2} \int_{\Omega}
			\mathbf D_{\neq}\left(
			\theta\nabla\partial^\alpha\rho
			\right)\cdot\partial^\alpha\mathbf u_{\neq} dx = \frac{1}{\varepsilon^2} \int_{\Omega}
			(\mathring{\T} \nabla \p^\alpha \rho_{\neq} + \T_{\neq} \nabla \p^\alpha \rho) \cdot \p^\alpha \mathbf{u}_{\neq} dx \\
			\le& \frac{1}{\varepsilon^2} \norm{\mathring{\T}}_{L^\infty(\mathbb{R})} \norm{\p^\alpha \rho_{\neq}}_{L^2(\Omega)} \norm{\nabla \p^\alpha \mathbf{u}_{\neq}}_{L^2(\Omega)} + \frac{1}{\varepsilon^2} \norm{\nabla \mathring{\T}}_{L^\infty(\mathbb{R})} \norm{\p^\alpha \rho_{\neq}}_{L^2(\Omega)} \norm{\p^\alpha \mathbf{u}_{\neq}}_{L^2(\Omega)} \\
			&+ \frac{1}{\varepsilon^2} \norm{\T_{\neq}}_{L^\infty(\Omega)} \norm{\p^\alpha \rho}_{L^2(\Omega)} \norm{\nabla \p^\alpha \mathbf{u}_{\neq}}_{L^2(\Omega)} + \frac{1}{\varepsilon^2} \norm{\nabla \T_{\neq}}_{L^\infty(\Omega)} \norm{\p^\alpha \rho}_{L^2(\Omega)} \norm{\p^\alpha \mathbf{u}_{\neq}}_{L^2(\Omega)} \\
			\le& \frac{\min{\{\mu, \kappa, \mu+\kappa\}}}{120} \norm{\nabla \p^\alpha \mathbf{u}_{\neq}}_{L^2(\Omega)}^2 + \frac{C}{\e^4} \norm{(\p^\alpha \rho_{\neq}, \p^\alpha \mathbf{u}_{\neq}, \p^\alpha \T_{\neq})}_{L^2(\Omega)}^2.
		\end{aligned}
	\end{equation}
	Using \(\mathbf{D}_0 (\ref{nondimensional NS Eqs})_1\) and the Gagliardo-Nirenberg inequality, one has
	\begin{equation}
		\begin{aligned}
			&\int_{\Omega} \partial_t\mathring{\rho} \left(|\partial^\alpha\mathbf u_{\neq}|^2 +|\partial^\alpha\theta_{\neq}|^2\right) dy \le \int_{\Omega} |\mathbf{D}_0 \dv (\rho \mathbf{u})| \left(|\partial^\alpha\mathbf u_{\neq}|^2 +|\partial^\alpha\theta_{\neq}|^2\right) dy \le C \norm{(\p^\alpha \mathbf{u}_{\neq}, \p^\alpha \T_{\neq})}_{L^2(\Omega)}^2.
		\end{aligned}
	\end{equation}
	For the term\(\int_{\Omega} F_{1,\alpha}\partial^\alpha\rho_{\neq} dx\), the key point is to derive an estimate for \(\int_{\Omega} \p^\alpha f_{1\neq} \p^\alpha\rho_{\neq} dx\). We take \(\int_{\Omega} (\mathring{\mathbf{u}} \cdot \nabla \rho_{\neq} - \mathbf{D}_{\neq} (\mathbf{u} \cdot \nabla \rho)) \p^\alpha\rho_{\neq} dx\) as an example, which is the part of the term \(\int_{\Omega} \p^\alpha f_{1\neq} \p^\alpha\rho_{\neq} dx\),
	\begin{align*}
		&\int_{\Omega} (\mathring{\mathbf{u}} \cdot \nabla \rho_{\neq} - \mathbf{D}_{\neq} (\mathbf{u} \cdot \nabla \rho)) \p^\alpha\rho_{\neq} dx = -\int_{\Omega} \p^\alpha (\mathbf{u}_{\neq} \cdot \nabla \mathring{\rho} + \mathbf{u}_{\neq} \cdot \nabla \rho_{\neq}) \p^\alpha \rho_{\neq} dx \\
		=& -\int_{\Omega} [\p^\alpha, \mathbf{u}_{\neq}] \cdot \nabla \mathring{\rho} \, \p^\alpha \rho_{\neq} dx - \int_{\Omega}\mathbf{u}_{\neq} \cdot \nabla \p^\alpha \mathring{\rho}\, \p^\alpha \rho_{\neq} dx - \int_{\Omega} [\p^\alpha, \mathbf{u}_{\neq}] \cdot \rho_{\neq} \, \p^\alpha \rho_{\neq}dx - \int_{\Omega}\mathbf{u}_{\neq} \cdot \nabla \p^\alpha \rho_{\neq} \, \p^\alpha \rho_{\neq} dx \\
		\le& \norm{\p^\alpha \mathring{\rho}}_{L^\infty(\mathbb{R})} \norm{(\p^\alpha \rho_{\neq}, \p^\alpha \mathbf{u}_{\neq})}_{L^2(\Omega)}^2 + \norm{\mathbf{u}_{\neq}}_{L^\infty(\Omega)} \norm{\p^\alpha \nabla \mathring{\rho}}_{L^2(\Omega)} \norm{\p^\alpha \rho_{\neq}}_{L^2(\Omega)} + C \norm{(\p^\alpha \rho_{\neq}, \p^\alpha \mathbf{u}_{\neq})}_{L^2(\Omega)}^2 \\
		\le& C \norm{(\p^\alpha \rho_{\neq}, \p^\alpha \mathbf{u}_{\neq})}_{L^2(\Omega)}^2.
	\end{align*}
	And the other terms can be treated in the same way.
	Then we have
	\begin{align}
		\int_{\Omega} F_{1,\alpha}\partial^\alpha\rho_{\neq} dx \le C \norm{(\p^\alpha \rho_{\neq}, \p^\alpha \mathbf{u}_{\neq})}_{L^2(\Omega)}^2.
	\end{align}
	For the term \(\int_{\Omega} \mathbf F_{2,\alpha}\cdot\partial^\alpha\mathbf u_{\neq} dx\), Using \(\eqref{p^alpha non-zero system}_2\), one obtains
	\begin{align*}
		&\int_{\Omega} [\p^\alpha, \mathring{\rho}] \p_t \mathbf{u}_{\neq} \cdot \p^\alpha \mathbf{u}_{\neq} dx \\
		\le&\frac{1}{\e^2} \norm{\p^\alpha \mathring{\rho}}_{L^\infty(\mathbb{\R})} \norm{\e^2 \p^{\alpha-1} \p_t \mathbf{u}_{\neq}}_{L^2(\Omega)} \norm{\p^\alpha \mathbf{u}_{\neq}}_{L^2(\Omega)} \\
		\le&\frac{1}{\e^2} \norm{\p^\alpha \mathring{\rho}}_{L^\infty(\mathbb{\R})} \norm{ \p^{\alpha-1} (\e^2\mathring{\rho}\p_t \mathbf{u}_{\neq}\frac{1}{\mathring{\rho}}) }_{L^2(\Omega)} \norm{\p^\alpha \mathbf{u}_{\neq}}_{L^2(\Omega)} \\
		\le&  \frac{\min{\{\mu, \kappa, \mu+\kappa\}}}{120} \norm{\nabla \p^\alpha \mathbf{u}_{\neq}}_{L^2(\Omega)}^2 + \frac{C}{\e^4} \norm{(\p^\alpha \rho_{\neq}, \p^\alpha \mathbf{u}_{\neq}, \p^\alpha \T_{\neq})}_{L^2(\Omega)}^2,
	\end{align*}
	\begin{align*}
		&\int_{\Omega}\left([\partial^\alpha,\mathbf D_0(\rho\mathbf u)]
		\cdot\nabla\mathbf u_{\neq}\right) \cdot \mathbf{u}_{\neq} dx
		\le \frac{1}{\e} \norm{\e \p^\alpha \mathbf{D}_0 (\rho \mathbf{u})}_{L^\infty(\R)} \norm{\p^\alpha \mathbf{u}_{\neq}}^2_{L^2(\Omega)}
		\le  \frac{C}{\e^4} \norm{\p^\alpha \mathbf{u}_{\neq}}_{L^2(\Omega)}^2,
	\end{align*}
	and
	\begin{align*}
		&\int_{\Omega} \frac{1}{\varepsilon^2}\mathbf D_{\neq}
		\left(
		[\partial^\alpha,\rho]\nabla\theta
		+[\partial^\alpha,\theta]\nabla\rho
		\right) \cdot \p^\alpha \mathbf{u}_{\neq} dx \le \frac{C}{\e^4} \norm{(\p^\alpha \rho_{\neq}, \p^\alpha \mathbf{u}_{\neq}, \p^\alpha \T_{\neq})}_{L^2(\Omega)}^2.
	\end{align*}
	Moreover, we have
	\begin{align*}
		&\int_{\Omega} \p^\alpha \mathbf f_{2\neq} \cdot \p^\alpha \mathbf{u}_{\neq} dx = \int_{\Omega} \p^\alpha \left(\mathring{\rho} \p_t \mathbf{u}_{\neq} - \mathbf{D}_{\neq} (\rho \p_t \mathbf{u})\right) \cdot \p^\alpha \mathbf{u}_{\neq} dx + \int_{\Omega} \p^\alpha \left(\mathbf{D}_0 (\rho \mathbf{u}) \cdot \nabla \mathbf{u}_{\neq} - \mathbf{D}_{\neq} (\rho \mathbf{u} \cdot \nabla \mathbf{u})\right) \cdot \p^\alpha \mathbf{u}_{\neq} dx \\
		=& -\int_{\Omega} \p^\alpha \left(\rho_{\neq} \p_t \mathbf{\mathring{u}}\right) \cdot \p^\alpha \mathbf{u}_{\neq} dx - \int_{\Omega} \p^\alpha \left(\rho_{\neq} \p_t \mathbf{u}_{\neq}\right) \cdot \p^\alpha \mathbf{u}_{\neq} dx - \int_{\Omega} \p^\alpha \left(\mathbf{D}_{\neq} (\rho \mathbf{u}) \cdot \nabla \mathring{\mathbf{u}}\right) \cdot \p^\alpha \mathbf{u}_{\neq} dx\\
		&- \int_{\Omega} \p^\alpha \left(\mathbf{D}_{\neq} (\rho \mathbf{u}) \cdot \nabla \mathbf{u}_{\neq}\right) \cdot \p^\alpha \mathbf{u}_{\neq} dx \\
		\le&-\frac{1}{2} \p_t \left(\rho_{\neq} |\p^\alpha \mathbf{u}_{\neq}|^2\right)
		+ \frac{\min{\{\mu, \kappa, \mu+\kappa\}}}{120} \norm{\nabla \p^\alpha \mathbf{u}_{\neq}}_{L^2(\Omega)}^2 + \frac{C}{\e^4} \norm{(\p^\alpha \rho_{\neq}, \p^\alpha \mathbf{u}_{\neq}, \p^\alpha \T_{\neq})}_{L^2(\Omega)}^2.
	\end{align*}
	Combining the above inequalities, we obtain
	\begin{align}
		&\int_{\Omega} \mathbf F_{2,\alpha}\cdot\partial^\alpha\mathbf u_{\neq} dx \le -\frac{1}{2} \p_t \left(\rho_{\neq} |\p^\alpha \mathbf{u}_{\neq}|^2\right)
		+ \frac{\min{\{\mu, \kappa, \mu+\kappa\}}}{120} \norm{\nabla \p^\alpha \mathbf{u}_{\neq}}_{L^2(\Omega)}^2 + \frac{C}{\e^4} \norm{(\p^\alpha \rho_{\neq}, \p^\alpha \mathbf{u}_{\neq}, \p^\alpha \T_{\neq})}_{L^2(\Omega)}^2.
	\end{align}
	By a similar calculation, we obtain
	\begin{align}\label{F_3 alpha est}
		&\int_{\Omega} F_{3,\alpha}\partial^\alpha\theta_{\neq} dx 
		\le -\frac{1}{2} \p_t \left(\rho_{\neq} |\p^\alpha \T_{\neq}|^2\right)
		+ \frac{\min{\{\mu, \kappa, \mu+\kappa\}}}{120} \norm{\nabla \p^\alpha \T_{\neq}}_{L^2(\Omega)}^2 + \frac{C}{\e^4} \norm{(\p^\alpha \rho_{\neq}, \p^\alpha \mathbf{u}_{\neq}, \p^\alpha \T_{\neq})}_{L^2(\Omega)}^2.
	\end{align}
	Finally, integrating (\ref{equ1 in ill}) with respect to \(x\) and combining with (\ref{est1 in ill}) - (\ref{F_3 alpha est}), we obtain
	\begin{align}\label{non-zero inequ}
		\frac{d}{dt} \norm{(\rho_{\neq}, \mathbf{u}_{\neq}, \T_{\neq})}_{H^s}^2 + \norm{(\nabla \mathbf{u}_{\neq}, \nabla \T_{\neq})}_{H^s(\Omega)}^2 \le \frac{C}{\e^4} \norm{( \rho_{\neq}, \mathbf{u}_{\neq}, \T_{\neq})}_{H^s(\Omega)}^2.
	\end{align}
	Let \(c_0 = 2C\). According to (\ref{non-zero inequ}), it yields that
	\begin{align*}
		\norm{(\rho_{\neq}, \mathbf{u}_{\neq}, \T_{\neq})}_{H^s}^2 \le& \norm{(\rho_{\neq}, \mathbf{u}_{\neq}, \T_{\neq})(0)}_{H^s}^2 e^{C\e^{-4}}
		\le e^{-C \e^{-4}}.
	\end{align*}
	This completes the energy estimates for the non-zero modes.
	%\end{proof}
	
	\subsection{Low Mach number limit}
	
	In this subsection, we will prove Theorem \ref{Convergence Thm for ill-prepared initial data} with a modified compactness argument, which was introduced by Métivier and Schochet in \cite{Métivier2001}.
	
	According to Theorem \ref{Uniform estimates Thm for ill-prepared initial data}, we find
	\begin{align}\label{zero mode est for ill}
		\norm{\mathbf{D}_0(\underline{p}^\e, \mathbf{u}^\e, \underline{\T}^\e - \underline{\Theta}^\e)}_{H^{s,\e}(\mathbb{R})}^2 \le 
		\norm{(\underline{p}^\e, \mathbf{u}^\e, \underline{\T}^\e - \underline{\Theta})}_{H^{s,\e}(\Omega)}^2.
	\end{align}
	Then extracting a subsequence, it holds that
	\begin{align*}
		&\mathbf{D}_0 (\underline{p}^\e, \underline{\mathbf{u}}^\e) \rightharpoonup (\bar{p}, \bar{\mathbf{u}}) \quad\text{as} \quad \e \to 0 \quad weakly-* \quad \text{in} \quad L^\infty (0, T_0; H^s(\mathbb{R})), \\
		&\mathbf{D}_0 (\underline{\T}^\e - \underline{\Theta}) \rightharpoonup \bar{\T} - \underline{\Theta} \quad\text{as} \quad \e \to 0 \quad weakly-* \quad \text{in} \quad L^\infty (0, T_0; H^{s+1}(\mathbb{R})).
	\end{align*}
	Applying \(\mathbf{D}_0\) to \((\ref{nondimensional NS Eqs})_3\), we have
	\begin{align}\label{equ of zero T}
		\mathring{\rho} \p_t \mathring{\T} + \mathring{\rho} \mathring{\mathbf{u}} \cdot \nabla \mathring{\T} + \mathring{\rho} \mathring{\T} \dv \mathring{\mathbf{u}} = \kappa \Delta \mathring{\T} + \e^2 \nabla \mathring{\mathbf{u}} : \mathbb{S}(\mathring{\mathbf{u}}) + f,
	\end{align}
	where
	\begin{align*}
		f =& - \mathbf{D}_0 (\rho_{\neq} \p_t \T_{\neq}) - \mathring{\rho} \mathbf{D}_0 (\mathbf{u}_{\neq} \cdot \nabla \T_{\neq}) - \mathring{\mathbf{u}} \cdot \mathbf{D}_0 (\mathbf{u}_{\neq} \cdot \nabla \T_{\neq}) - \mathbf{D}_0 (\rho_{\neq} \mathbf{u}_{\neq}) \cdot \nabla \mathring{\T} - \mathbf{D}_0 (\rho_{\neq} \mathbf{u}_{\neq} \cdot \nabla \T_{\neq}) \\
		&- \mathring{\rho} \mathbf{D}_0 (\T_{\neq} \dv \mathbf{u}_{\neq}) - \mathring{\T} \mathbf{D}_0 (\rho_{\neq} \dv \mathbf{u}_{\neq}) - \mathbf{D}_0(\rho_{\neq} \T_{\neq}) \dv \mathring{\mathbf{u}} - \mathbf{D}_0 (\rho_{\neq} \T_{\neq} \dv \mathbf{u}_{\neq}) \\
		&- \e^2 \nabla \mathring{\mathbf{u}} : \mathbb{S}(\mathring{\mathbf{u}}) + \mathbf{D}_0 [\e^2 \nabla \mathbf{u} : \mathbb{S}(\mathbf{u})].
	\end{align*}
	It follows from (\ref{transformation1 for ill}), (\ref{non-zero mode est for ill}), (\ref{zero mode est for ill}) and (\ref{equ of zero T}) that
	\begin{align*}
		\p_t \mathring{\T}^\e \in L^\infty(0, T_0; H^{s-2}(\mathbb{R})).
	\end{align*}
	Using (\ref{transformation1 for ill}), we observe that \(\p_t \underline{\mathring{\T}}^\e \in L^\infty(0, T_0; H^{s-2}(\mathbb{R}))\),
	which, together with Aubin–Lions lemma, yields that the functions \(\underline{\mathring{\T}}^\e\) converge to \(\bar{\T}\) strongly in \(C([0,T_0]; H^{s'+1}_{loc}(\mathbb{R}))\) for all \(s' < s\).
	Since \(\mathbf{D}_0 \begin{pmatrix}
		\p_{x_2} p \\
		\p_{x_3} p
	\end{pmatrix} = 0\),
	by a similar argument, we obtain \((\mathring{u}_2^\e, \mathring{u}_3^\e)\) converge to \((\bar{u}_2, \bar{u}_3)\) strongly in \(C([0,T_0]; H^{s'}_{loc}(\mathbb{R}))\) for all \(s' < s\).
	
	Then, we will show \(\mathring{\underline{p}}^\e\) and \(\p_{x_1} (2\mathring{u}_1 - \kappa e^{\mathring{\T}^\e} \underline{\mathring{\T}}^{\e}_{x_1})\) converge strongly to 0 as \(\e \to 0\). Applying \(\mathbf{D}_0\) to \((\ref{NS system in scaling})_1\) and \((\ref{NS system in scaling})_2\), we have
	\begin{align}
		&\e \p_t \mathring{\underline{p}}^\e + \p_{x_1} (2\mathring{u}_1^\e - \kappa e^{-\e \mathring{\underline{p}}^\e + \mathring{\underline{\T}}^\e} \p_{x_1} \mathring{\underline{\T}}^\e) = \e f_1^\e + g_1^\e, \label{equ1}\\
		&\e e^{-\underline{\mathring{\T}}^\e} \p_t \mathring{u}_1^\e + \p_{x_1} \underline{\mathring{p}}^\e = \e f_2^\e, \label{equ2}
	\end{align}
	where
	\begin{align*}
		&g_1^\e = -\kappa \p_{x_1} \left([\mathbf{D}_0(e^{-\e \underline{p}^\e + \underline{\T}^\e}) - e^{-\e \mathring{\underline{p}}^\e + \mathring{\underline{\T}}^\e}] \p_{x_1} \mathring{\underline{\T}}^\e \right)  - \kappa \p_{x_1} \mathbf{D}_0 \left(\mathbf{D}_{\neq} (e^{-\e \underline{p}^\e + \underline{\T}^\e}) \p_{x_1} \underline{\T}_{\neq}^\e \right).
	\end{align*}
	A direct calculation shows that
	\begin{align*}
		\mathbf{D}_0(e^{-\e \underline{p}^\e + \underline{\T}^\e}) - e^{-\e \mathring{\underline{p}}^\e + \mathring{\underline{\T}}^\e} = \mathbf{D}_0 \left(e^{-\e \underline{\mathring{p}}^\e+\underline{\mathring{\T}}^\e} (e^{-\e \underline{p}_{\neq}^\e + \underline{\T}_{\neq}^\e}-1)\right),
	\end{align*}
	then we have
	\begin{align*}
		\norm{g_1^\e}_{H^{s-1}(\Omega)}^2 \le e^{-C \e^{-4}}.
	\end{align*}
	Moreover, it follows from (\ref{zero mode est for ill}) and (\ref{non-zero mode est for ill}) that \(f_1^\e\) and \(f_2^\e\) are uniformly bounded in \(C([0,T_0]; H^{s-1}(\mathbb{R}))\).
	Passing the weak limit in (\ref{equ1}) and (\ref{equ2}) leads to \(\p_{x_1} \bar{p} = 0\) and \(\p_{x_1}(2\bar{u}_1 - \kappa e^{\bar{\T}}\p_{x_1}\bar{\T}) =0\) in the distribution sense.
	
	On the other hand, by taking the limit of \(\mathbf{D}_0(\ref{NS system in scaling})_3\), we can get that \(\bar{\T}\) satisfies
	\begin{align}\label{bar T equ}
		\bar{\T}_t = \frac{\kappa e^{\bar{\T}}}{2} \bar{\T}_{x_1 x_1},
	\end{align}
	with the initial data
	\begin{align}\label{bar T initial data}
		\bar{\T}(x_1, 0) = (\mathbf{D}_0 \underline{\T})_0 (x_1).
	\end{align}
	From the maximum principle and energy method, one can show the existence and uniqueness of smooth solution of (\ref{bar T equ}) and (\ref{bar T initial data}). Next, we will get the spatial decay of \(\bar{\T}\), as in \cite{Alazard2005LowMN}, we set
	\begin{align*}
		H = x_1^{1+\sigma} (\bar{\T} - \underline{\T}_+),
	\end{align*}
	which satisfies
	\begin{align*}
		H_t = \frac{\kappa e^{\bar{\T}}}{2} H_{x_1 x_1} - \frac{\kappa (1+\sigma) e^{\bar{\T}}}{x_1} H_{x_1} + \frac{\kappa (1+\sigma) (2+\sigma) e^{\bar{\T}}}{2x_1^2} H.
	\end{align*}
	By the energy estimates and (\ref{zero mode est for ill}), it holds that
	\begin{align*}
		\norm{H}_{L^\infty(0,T_1; H^1[1,\infty)} \le C [\norm{H(0)}_{H^1[1,\infty)} + \norm{\bar{\T} - \underline{\Theta}}_{H^2} + 1] \le C[1+ \Lambda(C_0)],
	\end{align*}
	then using Sobolev inequality, one gets that
	\begin{align}\label{bar T decay}
		|\bar{\T}(x_1, t) - \underline{\T}_+| \le Cx_1^{-1-\sigma}, \qquad \text{as}~ x_1 \in [1, +\infty). 
	\end{align}
	The following is to obtain the convergence of \(\mathring{u}_1^\e\) and \(\underline{\mathring{p}}^\e\).
	\begin{Prop}\label{Prop for ill}
		Under the assumptions of Theorem \ref{Convergence Thm for ill-prepared initial data}, suppose that the family of solutions satisfies (\ref{uniform estimates for ill data}) and (\ref{non-zero mode est for ill}) on \([0,T_1]\), and that the limiting temperature satisfies (\ref{bar T decay}). Then it holds that
		\begin{equation}
			\begin{aligned}
				&\underline{\mathring{p}}^\e \to 0 \qquad \text{strongly in}~ L^2(0,T_1; H^{s'}_{loc}(\mathbb{R})) \quad \text{as}~ \e \to 0, \\
				&(2\mathring{u}_1^\e - \kappa e^{-\e \mathring{\underline{p}}^\e + \mathring{\underline{\T}}^\e} \p_{x_1} \mathring{\underline{\T}}^\e)_{x_1} \to 0 \qquad \text{strongly in}~ L^2(0,T_1; H^{s'-1}_{loc}(\mathbb{R})) \quad \text{as}~ \e \to 0,
			\end{aligned}
		\end{equation}
		where \(s'<s\).
	\end{Prop}
	\begin{proof}
		Applying \(\frac{\e}{2} \p_t\) to (\ref{equ1}) and using (\ref{equ2}), we find
		\begin{align}\label{p_tt}
			\frac{\varepsilon^2}{2}
			\partial_t^2\mathring{\underline p}^{\varepsilon}
			-
			\partial_{x_1}\!\left(
			e^{\mathring{\underline\theta}^{\varepsilon}}
			\partial_{x_1}\mathring{\underline p}^{\varepsilon}
			\right)
			=\mathcal{R}^\e,
		\end{align}
		where
		\begin{align*}
			\mathcal{R}^\e=& \frac{\varepsilon^2}{2}\partial_t f_1^\varepsilon +\frac{\varepsilon}{2}\partial_t g_1^\varepsilon -\varepsilon\partial_{x_1}\!\left( e^{\mathring{\underline\theta}^{\varepsilon}}f_2^\varepsilon \right)
			+\frac{\kappa\varepsilon}{2} \partial_{x_1}\partial_t \left( e^{-\varepsilon\mathring{\underline p}^{\varepsilon} +\mathring{\underline\theta}^{\varepsilon}} \partial_{x_1}\mathring{\underline\theta}^{\varepsilon}\right).
		\end{align*}
		According to (\ref{uniform estimates for ill data}) and (\ref{equ of zero T}), we can check
		\begin{align}\label{1equ}
			&\|(\varepsilon\partial_t)f_1^\e\|_{L^2(0,T_1;L^2(\R))} +\|f_2^\e\|_{L^2(0,T_1;H^1(\R))}
			+\left\|\partial_t\left(e^{-\varepsilon\mathring{\underline p}^{\varepsilon} +\mathring{\underline\theta}^{\varepsilon}} \partial_{x_1}\mathring{\underline\theta}^{\varepsilon}\right) \right\|_{L^2(0,T_1;H^1(\R))} \le C.
		\end{align}
		It follows from (\ref{uniform estimates for ill data}) and (\ref{non-zero mode est for ill}) that
		\begin{align}\label{2equ}
			\|g_1^\e\|_{L^2 (0,T_1; L^2(\R))} + \|(\e \p_t)g_1^\e\|_{L^2 (0,T_1; L^2(\R))} \to 0 \qquad \text{as}~\e \to 0.
		\end{align}
		Combining (\ref{1equ}) and (\ref{2equ}), one gets 
		\begin{align*}
			\mathcal{R}^\e \to 0 \qquad \text{strongly}~\text{in}~ L^2 (0,T_1; L^2(\R)).
		\end{align*}
		Following the arguments in \cite{Métivier2001} and using (\ref{bar T decay}), we can obtain that 
		\begin{align*}
			\mathring{\underline p}^{\varepsilon}\longrightarrow 0\quad\text{strongly in } L^2(0,T_1;H^{s'}_{loc}(\mathbb R)), \qquad s'<s.
		\end{align*}
		On the other hand, we will establish the convergence of \((2\mathring{u}_1^\e - \kappa e^{-\e \mathring{\underline{p}}^\e + \mathring{\underline{\T}}^\e} \p_{x_1} \mathring{\underline{\T}}^\e)_{x_1}\). Applying \(\e \p_t\) to the equation (\ref{p_tt}), one obtains
		\begin{align*}
			\e^2 \p_t \left(\frac{1}{2}\p_t (\e \p_t \mathring{\underline p}^{\varepsilon}) \right) - \partial_{x_1} \left( e^{\mathring{\underline\theta}^{\varepsilon}} \partial_{x_1}(\e \p_t\mathring{\underline p}^{\varepsilon}) \right) = \e \p_t \mathcal{R}^\e + \e \p_{x_1} \left( e^{\mathring{\underline\theta}^{\varepsilon}} \p_t \mathring{\underline\theta}^{\varepsilon} \p_{x_1}\mathring{\underline p}^{\varepsilon} \right) := \tilde{\mathcal{R}}^\e,
		\end{align*}
		then a similar argument deduces that \(\e \p_t \mathring{\underline p}^{\varepsilon} \to 0\) strongly in \(L^2 (0,T_1; L^2_{loc}(\R))\). Moreover, we rewrite (\ref{equ1}) as follows:
		\begin{align*}
			(2\mathring{u}_1^\e - \kappa e^{-\e \mathring{\underline{p}}^\e + \mathring{\underline{\T}}^\e} \p_{x_1} \mathring{\underline{\T}}^\e)_{x_1} = -\e \p_t \mathring{\underline{p}}^\e + \e f_1^\e + g_1^\e.
		\end{align*}
		Similarly, in a same way, we can show the convergence of \((2\mathring{u}_1^\e - \kappa e^{-\e \mathring{\underline{p}}^\e + \mathring{\underline{\T}}^\e} \p_{x_1} \mathring{\underline{\T}}^\e)_{x_1} \to 0\) strongly in \(L^2(0,T_1; H^{s'-1}_{loc}(\R))\) for \(s' <s\), which completes the proof.
	\end{proof}
	Then, passing the limit in the equations (\ref{NS system in scaling}) for \((\mathring{\underline{p}}^\e, \mathring{\mathbf{u}}^\e, \mathring{\underline{\T}}^\e)\), one proves that the limit \((0, \bar{\mathbf{u}}, \bar{\underline{\T}})\) solves (\ref{limit system in scaling}) in the sense of distribution.
	Using the arguments in \cite{Métivier2001}, we can obtain that \((\bar{u}_1, \bar{u}_2, \bar{u}_3, \bar{\underline{\T}})\) satisfies the initial condition
	\begin{align}\label{initial data for limit sys with ill}
		(\bar{u}_1, \bar{u}_2, \bar{u}_3, \bar{\underline{\T}})|_{t=0} = (\underline{\mathring{w}}_0, \mathring{u}_{20}, \mathring{u}_{30}, \underline{\mathring{\T}}_0),
	\end{align}
	where \(\underline{\mathring{w}}_0\) is determined by \(\underline{\mathring{w}}_0 = \frac{1}{2} \kappa e^{\underline{\mathring{\T}}_0} \p_{x_1} \underline{\mathring{\T}}_0\). Moreover one can get the uniqueness of solutions to the zero-mode limit system with initial data (\ref{initial data for limit sys with ill}) by the energy method and then the above conclusions hold for the whole sequence \((\underline{\mathring{p}}^\e, \mathring{u}_1^\e, \mathring{u}_2^\e, \mathring{u}_3^\e, \underline{\mathring{\T}}^\e)\).
	
	On the other hand, it follows from (\ref{non-zero mode est for ill}) that
	\begin{align*}
		(\rho_{\neq}^\e, \mathbf{u}_{\neq}^\e, \T_{\neq}^\e) \to 0 \qquad \text{strongly in}~ L^\infty(0,T_1; H^{s} (\Omega)) \quad \text{as}~ \e \to 0,
	\end{align*}
	then we have
	\begin{align*}
		(\underline{p}_{\neq}^\e, \mathbf{u}_{\neq}^\e, \underline{\T}_{\neq}^\e) \to 0 \qquad \text{strongly in}~ L^2(0,T_1; H^{s} (\Omega)) \quad \text{as}~ \e \to 0.
	\end{align*}
	Consequently, we obtain
	\begin{align}
		(\underline{p}^\e, \mathbf{u}^\e, \underline{\T}^\e) = (\mathring{\underline{p}}^\e , \mathring{\mathbf{u}}^\e, \mathring{\underline{\T}}^\e) + (\underline{p}_{\neq}, \underline{\mathbf{u}}_{\neq}, \underline{\T}_{\neq}) \to (\bar{p}, \bar{\mathbf{u}}, \bar{\underline{\T}}) \qquad \text{strongly in}~ L^2(0,T_1; H^{s'}_{loc} (\Omega)) \quad \text{as}~ \e \to 0.
	\end{align}
	Finally, we have completed the proof of Theorem \ref{Convergence Thm for ill-prepared initial data}.
	
	\medskip
	\noindent {\bf Acknowledgment:} The research of Qiangchang Ju and Fanrui Meng are supported by the National Natural Science Foundation of China (No. 12131007).

	\medskip
	\noindent{\bf Data availability:} The manuscript contains no associated data.

	\medskip
	\noindent{\bf Conflict of Interest:} The authors declare that they have no conflict of interest.

\end{document}